\documentclass[a4paper]{article}
\usepackage{amsfonts}
\usepackage{amssymb}
\usepackage{amsmath}
\usepackage{amsthm}
\usepackage[dvipdfm]{graphicx,color}
\usepackage{moreverb}
\usepackage{fancybox}
\usepackage{fancyvrb}
\usepackage{comment}

\theoremstyle{plain}
\newtheorem{theorem}{Theorem}
\newtheorem{lemma}[theorem]{Lemma}
\newtheorem{cor}{Corollary}
\newtheorem{prop}[theorem]{Proposition}

\newtheorem{question}[theorem]{Question}

\newtheorem{sublem}{Sublemma}

\theoremstyle{definition}
\newtheorem{definition}[theorem]{Definition}
\newtheorem{example}[theorem]{Example}

\theoremstyle{definition}
\newtheorem*{remark}{Remark}
\newtheorem*{claim}{Claim}

\newtheorem{construction}{Construction.}

\newtheorem{ack}{Acknowledgment.}

\newcommand{\lrangle}[1]{\langle #1 \rangle}
\newcommand{\res}{\upharpoonright}

\newcommand{\fr}{\mbox{}^\frown}

\newcommand{\touch}[1]{\overset{#1}{\mathbf{\dashrightarrow}}}

\newcommand{\nn}{\mathbb{N}}

\title{Incomputability of Simply Connected Planar Continua}
\author{Takayuki Kihara\thanks{This work was supported by Grant-in-Aid for JSPS fellows.}}
\date{}

\begin{document}
\maketitle

\begin{abstract}
Le Roux and Ziegler asked whether every simply connected compact nonempty planar $\Pi^0_1$ set always contains a computable point.
In this paper, we solve the problem of le Roux and Ziegler by showing that there exists a planar $\Pi^0_1$ dendroid without computable points.
We also provide several pathological examples of tree-like $\Pi^0_1$ continua fulfilling certain global incomputability properties:
there is a computable dendrite which does not $\ast$-include a $\Pi^0_1$ tree;
there is a $\Pi^0_1$ dendrite which does not $\ast$-include a computable dendrite;
there is a computable dendroid which does not $\ast$-include a $\Pi^0_1$ dendrite.
Here, a continuum $A$ {\em $\ast$-includes} a member of a class $\mathcal{P}$ of continua if, for every positive real $\varepsilon$, $A$ includes a continuum $B\in\mathcal{P}$ such that the Hausdorff distance between $A$ and $B$ is smaller than $\varepsilon$.
%\keywords{Computable Analysis, Type-two-effectivity, Effectively closed sets}
\end{abstract}

\section{Background}

Every nonempty open set in a computable metric space (such as Euclidean $n$-space $\mathbb{R}^n$) contains a computable point.
In contrast, the Non-Basis Theorem asserts that a nonempty {\em co-c.e.\ closed} set (also called a $\Pi^0_1$ set) in Cantor space (hence, even in Euclidean $1$-space)  can avoid any computable points.
Non-Basis Theorems can shed new light on connections between {\em local} and {\em global} properties by incorporating the notions of {\em measure} and {\em category}.
For instance, Kreisel-Lacombe \cite{KrLa} and Tanaka \cite{Tan} showed that there is a $\Pi^0_1$ set with positive measure that contains no computable point.
Recent exciting progress in {\em Computable Analysis} \cite{Wei} naturally raises the question whether Non-Basis Theorems exist for {\em connected} $\Pi^0_1$ sets.
However, we observe that, if a nonempty $\Pi^0_1$ subset of $\mathbb{R}^1$ contains no computable points, then it must be totally disconnected.
Then, in higher dimensional Euclidean space, can there exist a connected $\Pi^0_1$ set containing no computable points?
It is easy to construct a nonempty connected $\Pi^0_1$ subset of $[0,1]^2$ without computable points, and a nonempty simply connected $\Pi^0_1$ subset of $[0,1]^3$ without computable points.
An open problem, formulated by Le Roux and Ziegler \cite{RZ} was whether every nonempty simply connected compact planar $\Pi^0_1$ set contains a computable point.
As mentioned in Penrose's book ``{\em Emperor's New Mind}'' \cite{Pen}, {\em the Mandelbrot set} is an example of a simply connected compact planar $\Pi^0_1$ set which contains a computable point, and he conjectured that the Mandelbrot set is not computable {\em as a closed set}.
Hertling \cite{Her} observed that the Penrose conjecture has an implication for a famous open problem on local connectivity of the Mandelbrot set.
Our interest is which topological assumption (especially, connectivity assumption) on a $\Pi^0_1$ set can force it to possess a given computability property.
Miller \cite{Mil} showed that every $\Pi^0_1$ sphere in $\mathbb{R}^n$ is computable, and so it contains a dense c.e.\ subset of computable points.
He also showed that every $\Pi^0_1$ ball in $\mathbb{R}^n$ contains a dense subset of computable points.
Iljazovi\'c \cite{Ilj} showed that chainable continua (e.g., arcs) in certain metric spaces are almost computable, and hence there always is a dense subset of computable points.
In this paper, we show that {\em not} every $\Pi^0_1$ dendrite is almost computable, by using a tree-immune $\Pi^0_1$ class in Cantor space.
This notion of immunity was introduced by Cenzer, Weber Wu, and the author \cite{CKWW}.
We also provide pathological examples of tree-like $\Pi^0_1$ continua fulfilling certain global incomputability properties:
there is a computable dendrite which does not $\ast$-include a $\Pi^0_1$ tree;
there is a computable dendroid which does not $\ast$-include a $\Pi^0_1$ dendrite.
Finally, we solve the problem of Le Roux and Ziegler \cite{RZ} by showing that there exists a planar $\Pi^0_1$ dendroid without computable points.
Indeed, our planar dendroid is contractible.
Hence, our dendroid is also the first example of a contractible Euclidean $\Pi^0_1$ set without computable points.

\section{Preliminaries}

\noindent
{\bf Basic Notation:}
$2^{<\nn}$ denotes the set of all finite binary strings.
Let $X$ be a topological space.
For a subset $Y\subseteq X$, $cl(Y)$ ($int(Y)$, resp.) denotes the closure (the interior, resp.) of $Y$.
Let $(X;d)$ be a metric space.
For any $x\in X$ and $r\in\mathbb{R}$, $B(x;r)$ denotes the open ball $B(x;r)=\{y\in X:d(x,y)<r\}$.
Then $x$ is called {\em the center of $B(x;r)$}, and $r$ is called {\em the radius of $B(x;r)$}.
For a given open ball $B=B(x;r)$, $\hat{B}$ denotes the corresponding closed ball $\hat{B}=\{y\in X:d(x,y)\leq r\}$.
For $a,b\in\mathbb{R}$, $[a,b]$ denotes the closed interval $[a,b]=\{x\in\mathbb{R}:a\leq x\leq b\}$, $(a,b)$ denotes the open interval $(a,b)=\{x\in\mathbb{R}:a<x<b\}$, and $\lrangle{a,b}$ denotes a point of Euclidean plane $\mathbb{R}^2$.
For $X\subseteq \mathbb{R}^n$, ${\rm diam}(X)$ denotes $\max\{d(x,y):x,y\in X\}$.

\medskip

\noindent
{\bf Continuum Theory:}
{\em A continuum} is a compact connected metric space.
For basic terminology concerning {\em Continuum Theory}, see Nadler \cite{Nad} and Illanes-Nadler \cite{IlNa}.

Let $X$ be a topological space.
The set $X$ is {\em a Peano continuum} if it is a locally connected continuum.
The set $X$ is {\em a dendrite} if it is a Peano continuum which contains no Jordan curve.
The set $X$ is {\em unicoherent} if $A\cap B$ is connected for every connected closed subsets $A,B\subseteq X$ with $A\cup B=X$.
The set $X$ is {\em hereditarily unicoherent} if every subcontinuum of $X$ is unicoherent.
The set $X$ is {\em a dendroid} if it is an arcwise connected hereditary unicoherent continuum.
For a point $x$ of a dendroid $X$, $r_X(x)$ denotes the cardinality of the set of arc-components of $X\setminus\{x\}$.
If $r_X(x)\geq 3$ then $x$ is said to be {\em a ramification point of $X$}.
The set $X$ is {\em a tree} if it is dendrite with finitely many ramification points.
Note that a topological space $X$ is a dendrite if and only if it is a locally connected dendroid.
Hahn-Mazurkiewicz's Theorem states that a Hausdorff space $X$ is a Peano continuum if and only if $X$ is an image of a continuous curve.

\begin{example}[Planar Dendroids]\label{exa:dend:dend}~
\begin{enumerate}
\item Put $\mathcal{B}_t=\{2^{-t}\}\times[0,2^{-t}]$.
Then the following set $\mathcal{B}\subseteq\mathbb{R}^2$ is dendrite.
\[\mathcal{B}=\bigcup_{t\in\nn}\mathcal{B}_t\cup([-1,1]\times\{0\}).\]
We call $\mathcal{B}$ {\em the basic dendrite}.
The set $\mathcal{B}_{t}$ is called {\em the $t$-th rising of $\mathcal{B}$}.
See Fig.\ \ref{fig:dendrite}.
\item The set $\mathcal{H}=cl((\{1/n:n\in\nn\}\times[0,1])\cup([0,1]\times\{0\}))$ is called {\em a harmonic comb}.
Then $\mathcal{H}$ is a dendroid, but not a dendrite.
The set $\{1/n\}\times[0,1]$ is called {\em the $n$-th rising of the comb $\mathcal{H}$}, and the set $[0,1]\times\{0\}$ is called {\em the grip of $\mathcal{H}$}.
See Fig.\ \ref{fig:Harmonic_comb}.
\item Let $C\subseteq\mathbb{R}^1$ be the middle third Cantor set.
Then the one-point compactification of $C\times(0,1]$ is called {\em the Cantor fan}.
(Equivalently, it is the quotient space ${\rm Cone}(C)=(C\times[0,1])/(C\times\{0\})$.)
The Cantor fan is a dendroid, but not a dendrite.
See Fig.\ \ref{fig:Cantor_fan}.
\end{enumerate}
\end{example}

\begin{figure}[t]\centering
 \begin{minipage}{0.3\hsize}
  \begin{center}
%WinTpicVersion3.08
\unitlength 0.1in
\begin{picture}( 12.0000,  6.0000)(  2.0000, -8.0000)
% LINE 2 0 3 0
% 14 1400 200 1400 800 1400 800 800 800 1100 800 1100 500 950 800 950 650 875 800 875 725 838 800 838 762 819 800 819 780
% 
\special{pn 8}%
\special{pa 1400 200}%
\special{pa 1400 800}%
\special{fp}%
\special{pa 1400 800}%
\special{pa 800 800}%
\special{fp}%
\special{pa 1100 800}%
\special{pa 1100 500}%
\special{fp}%
\special{pa 950 800}%
\special{pa 950 650}%
\special{fp}%
\special{pa 876 800}%
\special{pa 876 726}%
\special{fp}%
\special{pa 838 800}%
\special{pa 838 762}%
\special{fp}%
\special{pa 820 800}%
\special{pa 820 780}%
\special{fp}%
% POLYGON 2 5 1 0
% 6 800 800 819 800 819 780 800 800 800 800 800 800
% 
\special{pn 8}%
\special{sh 0.300}%
\special{pa 800 800}%
\special{pa 820 800}%
\special{pa 820 780}%
\special{pa 800 800}%
\special{pa 800 800}%
\special{pa 800 800}%
\special{ip}%
% LINE 2 0 3 0
% 2 800 800 200 800
% 
\special{pn 8}%
\special{pa 800 800}%
\special{pa 200 800}%
\special{fp}%
\end{picture}%
  \end{center}
 \vspace{-0.5em}
\caption{The basic dendrite}
  \label{fig:dendrite}
 \end{minipage}
 \begin{minipage}{0.35\hsize}
  \begin{center}
%WinTpicVersion3.08
\unitlength 0.1in
\begin{picture}( 12.0000,  6.0000)(  2.0000, -8.0000)
% LINE 2 0 3 0
% 14 1400 800 200 800 1400 800 1400 200 200 800 200 200 800 800 800 200 500 800 500 200 350 800 350 200 275 800 275 200
% 
\special{pn 8}%
\special{pa 1400 800}%
\special{pa 200 800}%
\special{fp}%
\special{pa 1400 800}%
\special{pa 1400 200}%
\special{fp}%
\special{pa 200 800}%
\special{pa 200 200}%
\special{fp}%
\special{pa 800 800}%
\special{pa 800 200}%
\special{fp}%
\special{pa 500 800}%
\special{pa 500 200}%
\special{fp}%
\special{pa 350 800}%
\special{pa 350 200}%
\special{fp}%
\special{pa 276 800}%
\special{pa 276 200}%
\special{fp}%
% LINE 2 0 3 0
% 2 240 800 240 200
% 
\special{pn 8}%
\special{pa 240 800}%
\special{pa 240 200}%
\special{fp}%
% BOX 2 5 1 0
% 2 225 800 200 200
% 
\special{pn 8}%
\special{sh 0.300}%
\special{pa 226 800}%
\special{pa 200 800}%
\special{pa 200 200}%
\special{pa 226 200}%
\special{pa 226 800}%
\special{ip}%
\end{picture}%
  \end{center}
\caption{The harmonic comb}
  \label{fig:Harmonic_comb}
 \end{minipage}
 \begin{minipage}{0.3\hsize}
  \begin{center}
%WinTpicVersion3.08
\unitlength 0.1in
\begin{picture}( 12.0000, 10.1300)(  4.0000,-12.0300)
% DOT 2 0 3 0
% 8 400 600 500 600 700 600 800 600 1200 600 1300 600 1500 600 1600 600
% 
\special{pn 8}%
\special{sh 1}%
\special{ar 400 600 10 10 0  6.28318530717959E+0000}%
\special{sh 1}%
\special{ar 500 600 10 10 0  6.28318530717959E+0000}%
\special{sh 1}%
\special{ar 700 600 10 10 0  6.28318530717959E+0000}%
\special{sh 1}%
\special{ar 800 600 10 10 0  6.28318530717959E+0000}%
\special{sh 1}%
\special{ar 1200 600 10 10 0  6.28318530717959E+0000}%
\special{sh 1}%
\special{ar 1300 600 10 10 0  6.28318530717959E+0000}%
\special{sh 1}%
\special{ar 1500 600 10 10 0  6.28318530717959E+0000}%
\special{sh 1}%
\special{ar 1600 600 10 10 0  6.28318530717959E+0000}%
% LINE 2 0 3 0
% 16 400 600 1000 1200 1600 600 1000 1200 500 600 1000 1200 700 600 1000 1200 800 600 1000 1200 1200 600 1000 1200 1300 600 1000 1200 1500 600 1000 1200
% 
\special{pn 8}%
\special{pa 400 600}%
\special{pa 1000 1200}%
\special{fp}%
\special{pa 1600 600}%
\special{pa 1000 1200}%
\special{fp}%
\special{pa 500 600}%
\special{pa 1000 1200}%
\special{fp}%
\special{pa 700 600}%
\special{pa 1000 1200}%
\special{fp}%
\special{pa 800 600}%
\special{pa 1000 1200}%
\special{fp}%
\special{pa 1200 600}%
\special{pa 1000 1200}%
\special{fp}%
\special{pa 1300 600}%
\special{pa 1000 1200}%
\special{fp}%
\special{pa 1500 600}%
\special{pa 1000 1200}%
\special{fp}%
% DOT 2 0 3 0
% 1 1000 1200
% 
\special{pn 8}%
\special{sh 1}%
\special{ar 1000 1200 10 10 0  6.28318530717959E+0000}%
% LINE 2 2 3 0
% 6 400 500 400 400 1600 400 1600 500 1600 400 400 400
% 
\special{pn 8}%
\special{pa 400 500}%
\special{pa 400 400}%
\special{dt 0.045}%
\special{pa 1600 400}%
\special{pa 1600 500}%
\special{dt 0.045}%
\special{pa 1600 400}%
\special{pa 400 400}%
\special{dt 0.045}%
% STR 2 0 3 0
% 3 600 260 600 360 2 0
% Cantor set
\put(6.0000,-3.6000){\makebox(0,0)[lb]{Cantor set}}%
\end{picture}%
  \end{center}
\vspace{-1.5em}
\caption{The Cantor fan}
  \label{fig:Cantor_fan}
 \end{minipage}
\end{figure}

Let $X$ be a topological space.
$X$ is {\em $n$-connected} if it is path-connected and $\pi_i(X)\equiv 0$ for any $1\leq i\leq n$, where $\pi_i(X)$ is the $i$-th homotopy group of $X$.
$X$ is {\em simply connected} if $X$ is 1-connected.
$X$ is {\em contractible} if the identity map on $X$ is null-homotopic.
Note that, if $X$ is contractible, then $X$ is $n$-connected for each $n\geq 1$.
It is easy to see that the dendroids in Example \ref{exa:dend:dend} are contractible.

\medskip

\noindent
{\bf Computability Theory:}
We assume that the reader is familiar with Computability Theory on the natural numbers $\mathbb{N}$, Cantor space $2^\mathbb{N}$, and Baire space $\mathbb{N}^\mathbb{N}$ (see also Soare \cite{Soa}).
For basic terminology concerning {\em Computable Analysis}, see Weihrauch \cite{Wei}, Brattka-Weihrauch \cite{BW}, and Brattka-Presser \cite{BP}.

Hereafter, we fix a countable base for the Euclidean $n$-space $\mathbb{R}^n$ by $\rho=\{B(x;r):x\in\mathbb{Q}^n\;\&\;r\in\mathbb{Q}^+\}$, where $\mathbb{Q}^+$ denotes the set of all positive rationals.
Let $\{\rho_n\}_{n\in\nn}$ be an effective enumeration of $\rho$.
We say that a point $x\in\mathbb{R}^n$ is {\em computable} if the code of its principal filter $\mathcal{F}(x)=\{i\in\mathbb{N}:x\in\rho_i\}$ is computably enumerable (hereafter c.e.)
A closed subset $F\subseteq\mathbb{R}^n$ is $\Pi^0_1$ if there is a c.e.\ set $W\subseteq\nn$ such that $F=X\setminus\bigcup_{e\in W}\rho_e$.
A closed subset $F\subseteq\mathbb{R}^n$ is {\em computably enumerable} (hereafter {\em c.e.}) if $\{e\in\nn:F\cap\rho_e\not=\emptyset\}$ is c.e.
A closed subset $F\subseteq\mathbb{R}^n$ is {\em computable} if it is $\Pi^0_1$ and c.e.\ on $\mathbb{R}^n$.

\medskip

\noindent
{\bf Almost Computability:}
Let $A_0,A_1$ be nonempty closed subsets of a metric space $(X,d)$.
Then {\em the Hausdorff distance} between $A_0$ and $A_1$ is defined by
\[d_H(A_0,A_1)=\max_{i<2}\sup_{x\in A_i}\inf_{y\in A_{1-i}}d(x,y).\]
Let $\mathcal{P}$ be a class of continua.
We say that a continuum {\em $A$ $\ast$-includes a member of $\mathcal{P}$} if $\inf\{d_H(A,B):A\supseteq B\in\mathcal{P}\}=0$.

\begin{prop}\label{prop:astinclude}
Every Euclidean dendroid $\ast$-includes a tree.
\end{prop}

\begin{proof}\upshape
Fix a Euclidean dendroid $D\subseteq\mathbb{R}^n$, and a positive rational $\varepsilon\in\mathbb{Q}$.
Then $D$ is covered by finitely many open rational balls $\{B_i\}_{i<n}$ of radius $\varepsilon/2$.
Choose $d_i\in D\cap B_i$ for each $i<n$ if $B_i$ intersects with $D$.
Note that $\{B(d_i;\varepsilon)\}_{i<n}$ covers $D$.
Since $D$ is dendroid, there is a unique arc $\gamma_{i,j}\subseteq D$ connecting $d_i$ and $d_j$ for each $i,j<n$.
Then, $E=\bigcup_{\{i,j\}\subseteq n}\gamma_{i,j}$ is connected and locally connected, since $E$ is a union of finitely many arcs (i.e., it is a graph, in the sense of Continuum Theory; see also Nadler \cite{Nad}).
It is easy to see that $E$ has no Jordan curve, since $E$ is a subset of the dendroid $D$.
Consequently, $E$ is a tree.
Moreover, clearly $d_H(E,D)<\varepsilon$, since $d_i\in E$ for each $i<n$.
%%%%
\end{proof}

The class $\mathcal{P}$ has {\em the almost computability property} if every $A\in\mathcal{P}$ $\ast$-includes a computable member of $\mathcal{P}$ as a closed set.
In this case, we simply say that {\em every $A\in\mathcal{P}$ is almost computable}.
Iljazovi\'c \cite{Ilj} showed that every $\Pi^0_1$ chainable continuum is almost computable, hence every $\Pi^0_1$ arc is almost computable.

\section{Incomputability of Dendrites}

By Proposition \ref{prop:astinclude}, topologically, every planar dendrite $\ast$-includes a tree.
However, if we try to effectivize this fact, we will find a counterexample.

\begin{theorem}\label{thm:main:dendrite1}
Not every computable planar dendrite $\ast$-includes a $\Pi^0_1$ tree.
\end{theorem}

\begin{proof}\upshape
Let $A\subseteq\mathbb{N}$ be an incomputable c.e.\ set.
Thus, there is a total computable function $f_A:\mathbb{N}\to\mathbb{N}$ such that ${\rm range}(f_A)=A$.
We may assume $f_A(s)\leq s$ for every $s\in\mathbb{N}$.
Let $A_s$ denote the finite set $\{f_A(u):u\leq s\}$.
Then ${\rm st}^A:\nn\to\nn$ is defined as ${\rm st}^A(n)=\min\{s\in\nn:n\in A_s\}$.
Note that ${\rm st}^A(n)\geq n$ by our assumption $f_A(s)\leq s$.

\begin{construction}\upshape
Recall the definition of the basic dendrite from Example \ref{exa:dend:dend}.
We construct a computable dendrite by modifying the basic dendrite $\mathcal{B}$.
For every $t\in\mathbb{N}$, we introduce {\em the width of the $t$-rising} $w(t)$ as follows:
\[
w(t)=
\begin{cases}
2^{-(2+{\rm st}^A(t))}, & \mbox{ if } t\in A,\\
0, & \mbox{ otherwise.}
\end{cases}
\]

Let $I_t$ be the closed interval $[2^{-t}-w(t),2^t+w(t)]$.
Since ${\rm st}^A(n)\geq n$, we have $I_t\cap I_s=\emptyset$ whenever $t\not=s$.
We observe that $\{w(t)\}_{t\in\nn}$ is a uniformly computable sequence of real numbers.
Now we define a computable dendrite $D\subseteq\mathbb{R}^2$ by:
\begin{align*}
D^0_t&=(\{2^{-t}-w(t)\}\cup\{2^{-t}+w(t)\})\times [0,2^{-t}]\\
D^1_t&=[2^{-t}-w(t),2^{-t}+w(t)]\times\{2^{-t}\}\\
D^2_t&=(2^{-t}-w(t),2^{-t}+w(t))\times(-1,2^{-t})\\
D&=\Big(\bigcup_{t\in\mathbb{N}}(D^0_{t}\cup D^1_{t})\Big)\cup\Big(([-1,1]\times \{0\})\setminus\bigcup_{t\in\mathbb{N}}D^2_{t,m}\Big).
\end{align*}

We call $D_t=D^0_t\cup D^1_t$ {\em the $t$-th rising of $D$}.
See Fig.\ \ref{fig:Basic_dendrite2}.
\end{construction}

\begin{figure}[t]\centering
  \begin{center}
%WinTpicVersion3.08
\unitlength 0.1in
\begin{picture}( 36.0000, 12.3300)( 14.0000,-14.0300)
% POLYGON 2 5 0 0
% 6 2600 1400 2676 1400 2676 1360 2600 1400 2600 1400 2600 1400
% 
\special{pn 8}%
\special{sh 0.600}%
\special{pa 2600 1400}%
\special{pa 2676 1400}%
\special{pa 2676 1360}%
\special{pa 2600 1400}%
\special{pa 2600 1400}%
\special{pa 2600 1400}%
\special{ip}%
% LINE 2 0 3 0
% 2 2600 1400 1400 1400
% 
\special{pn 8}%
\special{pa 2600 1400}%
\special{pa 1400 1400}%
\special{fp}%
% STR 2 0 3 0
% 3 4800 300 4800 400 2 0
% $D_0$
\put(48.0000,-4.0000){\makebox(0,0)[lb]{$D_0$}}%
% LINE 2 0 3 0
% 26 2600 1400 2860 1400 2860 1250 2940 1250 2940 1400 3200 1400 2940 1400 2940 1250 2860 1400 2860 1250 3200 1400 3200 1100 3200 1400 3600 1400 3600 1400 3600 800 3600 800 4000 800 4000 800 4000 1400 4000 1400 5000 1400 5000 1400 5000 200 2750 1400 2750 1320
% 
\special{pn 8}%
\special{pa 2600 1400}%
\special{pa 2860 1400}%
\special{fp}%
\special{pa 2860 1250}%
\special{pa 2940 1250}%
\special{fp}%
\special{pa 2940 1400}%
\special{pa 3200 1400}%
\special{fp}%
\special{pa 2940 1400}%
\special{pa 2940 1250}%
\special{fp}%
\special{pa 2860 1400}%
\special{pa 2860 1250}%
\special{fp}%
\special{pa 3200 1400}%
\special{pa 3200 1100}%
\special{fp}%
\special{pa 3200 1400}%
\special{pa 3600 1400}%
\special{fp}%
\special{pa 3600 1400}%
\special{pa 3600 800}%
\special{fp}%
\special{pa 3600 800}%
\special{pa 4000 800}%
\special{fp}%
\special{pa 4000 800}%
\special{pa 4000 1400}%
\special{fp}%
\special{pa 4000 1400}%
\special{pa 5000 1400}%
\special{fp}%
\special{pa 5000 1400}%
\special{pa 5000 200}%
\special{fp}%
\special{pa 2750 1400}%
\special{pa 2750 1320}%
\special{fp}%
% STR 2 0 3 0
% 3 3580 620 3580 720 2 0
% $D_1$
\put(35.8000,-7.2000){\makebox(0,0)[lb]{$D_1$}}%
% DOT 1 0 3 0
% 1 2600 1400
% 
\special{pn 13}%
\special{sh 1}%
\special{ar 2600 1400 10 10 0  6.28318530717959E+0000}%
% STR 2 0 3 0
% 3 2570 1430 2570 1530 2 0
% $0$
\put(25.7000,-15.3000){\makebox(0,0)[lb]{$0$}}%
% STR 2 0 3 0
% 3 3120 1450 3120 1550 2 0
% $1/4$
\put(31.2000,-15.5000){\makebox(0,0)[lb]{$1/4$}}%
% STR 2 0 3 0
% 3 4970 1430 4970 1530 2 0
% $1$
\put(49.7000,-15.3000){\makebox(0,0)[lb]{$1$}}%
% STR 2 0 3 0
% 3 3720 1450 3720 1550 2 0
% $1/2$
\put(37.2000,-15.5000){\makebox(0,0)[lb]{$1/2$}}%
% STR 2 0 3 0
% 3 3080 920 3080 1020 2 0
% $D_2$
\put(30.8000,-10.2000){\makebox(0,0)[lb]{$D_2$}}%
% LINE 2 2 3 0
% 4 3800 1400 3800 400 4000 800 4000 400
% 
\special{pn 8}%
\special{pa 3800 1400}%
\special{pa 3800 400}%
\special{dt 0.045}%
\special{pa 4000 800}%
\special{pa 4000 400}%
\special{dt 0.045}%
% VECTOR 2 0 3 0
% 4 3900 400 3800 400 3900 400 4000 400
% 
\special{pn 8}%
\special{pa 3900 400}%
\special{pa 3800 400}%
\special{fp}%
\special{sh 1}%
\special{pa 3800 400}%
\special{pa 3868 420}%
\special{pa 3854 400}%
\special{pa 3868 380}%
\special{pa 3800 400}%
\special{fp}%
\special{pa 3900 400}%
\special{pa 4000 400}%
\special{fp}%
\special{sh 1}%
\special{pa 4000 400}%
\special{pa 3934 380}%
\special{pa 3948 400}%
\special{pa 3934 420}%
\special{pa 4000 400}%
\special{fp}%
% STR 2 0 3 0
% 3 3790 240 3790 340 2 0
% $w(1)$
\put(37.9000,-3.4000){\makebox(0,0)[lb]{$w(1)$}}%
\end{picture}%
  \end{center}
 %\vspace{0.5em}
\caption{The dendrite $D$ for $0,2,4\not\in A$ and $1,3\in A$.}
  \label{fig:Basic_dendrite2}
\end{figure}
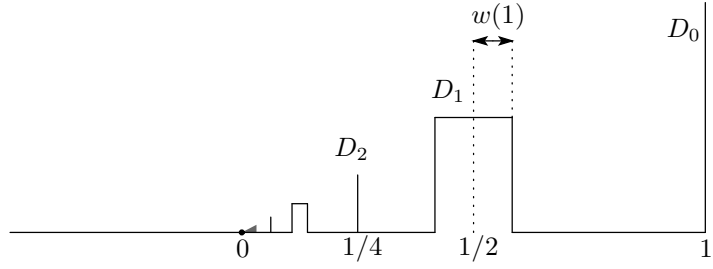

\begin{claim}
The set $D$ is a dendrite.
\end{claim}

To prove $D$ is a Peano continuum, by the Hahn-Mazurkiewicz Theorem, it suffices to show that $D={\rm Im}(h)$ for some continuous curve $h:[-1,1]\to\mathbb{R}^2$.
We divide the unit interval $[0,1]$ into infinitely many parts $I_t=[2^{-(t+1)},2^{-t}]$.
Furthermore, we also divide each interval $I_{2t}$ into three parts $I^0_{2t}$, $I^1_{2t}$, and $I^2_{2t}$, where $I^i_{2t}=[(5-i)\cdot 3^{-1}\cdot 2^{-(2t+1)},(6-i)\cdot 3^{-1}\cdot 2^{-(2t+1)}]$ for each $i<3$.
Then we define a desired curve $h$ as follows.
\[
h(x)\mbox{ moves in }
\begin{cases}
\{2^{-t}+w(t)\}\times[0,2^{-t}]&\mbox{if }x\in I^0_{2t},\\
[2^{-t}-w(t),2^{-t}+w(t)]\times\{2^{-t}\}&\mbox{if }x\in I^1_{2t},\\
\{2^{-t}-w(t)\}\times[0,2^{-t}]&\mbox{if }x\in I^2_{2t},\\
[2^{-(t+1)}+w(t+1),2^{-t}-w(t)]\times\{0\}&\mbox{if }x\in I_{2t+1},\\
[-1,0]\times\{0\}&\mbox{if }x\in [-1,0].
\end{cases}
\]

Clearly, $h$ can be continuous, and indeed computable, since the map $w:\mathbb{R}\to\mathbb{R}$ is computable.
It is easy to see that $D={\rm Im}(h)$.
Moreover, ${\rm Im}(h)$ contains no Jordan curve since $\pi_0(h(x))\leq\pi_0(h(y))$ whenever $x\leq y$, where $\pi_0(p)$ denotes the first coordinate of $p\in\mathbb{R}^2$.
Consequently, $D$ is a dendrite.

\medskip

Moreover, by construction, it is easy to see that $D$ is computable.

\begin{claim}
The computable dendrite $D$ does not $\ast$-include a $\Pi^0_1$ tree.
\end{claim}

Suppose that $D$ contains a $\Pi^0_1$ subtree $T\subseteq D$.
We consider a rational open ball $B_t$ with center $\lrangle{2^{-t},2^{-t}}$ and radius $2^{-(t+2)}$, for each $t\in\nn$.
Note that $B_t\cap D\subseteq D_t$ for every $t\in\nn$.
Since $T$ is $\Pi^0_1$ in $\mathbb{R}^2$, $B=\{t\in\mathbb{N}:\hat{B}_{t}\cap T=\emptyset\}$ is c.e.
If $w(t)>0$ (i.e., $t\in A$) then $D\setminus(D_t\cap B_t)$ is disconnected.
Therefore, either $T\subseteq [-1,2^{-t}]\times\mathbb{R}$ or $T\subseteq[2^{-t},1]\times\mathbb{R}$ holds whenever $\hat{B}_{t}\cap T=\emptyset$ (i.e., $t\in B$), since $T$ is connected.
Thus, if the condition $\#(A\cap B)=\aleph_0$ is satisfied, then either $T\subseteq [-1,0]\times\mathbb{R}$ or $T\subseteq [0,1]\times\mathbb{R}$ holds.
Consequently, we must have $d_H(T,D)\geq 1$.

Therefore, we may assume $\#A\cap B<\aleph_0$.
Since $A$ is coinfinite, $D$ has infinitely many ramification points $\lrangle{2^{-t},0}$ for $t\not\in A$.
However, by the definition of tree, $T$ has only finitely many ramification points.
Thus we must have $(D^0_{t}\cap T)\setminus \{\lrangle{2^{-t},0}\}=\emptyset$ for almost all $t\not\in A$.
Since $\hat{B}_{t}\cap T\subseteq (D^0_{t}\cap T)\setminus \{\lrangle{2^{-t},0}\}$, we have $t\in B$ for almost all $t\in\nn\setminus A$.
Consequently, we have $\#((\mathbb{N}\setminus A)\triangle B)<\aleph_0$.
This implies that $\mathbb{N}\setminus A$ is also c.e., since $B$ is c.e.
This contradicts that $A$ is incomputable.
%%%%
\end{proof}

Note that a Hausdorff space (hence every metric space) is (locally) arcwise connected if and only if it is (locally) pathwise connected.
However, Miller \cite{Mil} pointed out that the effective versions of arcwise connectivity and pathwise connectivity do {\em not} coincide.
Theorem \ref{thm:main:dendrite1} could give a result on effective connectivity properties.
Note that {\em effectively pathwise connectivity} is defined by Brattka \cite{Bra08}.
Clearly, the dendrite $D$ is effectively pathwise connected.
We now introduce a new effective version of arcwise connectivity property by thinking arcs as closed sets.
Let $\mathcal{A}_-(X)$ denote the hyperspace of closed subsets of $X$ with negative information (see also Brattka \cite{Bra08}).
\begin{definition}
A computable metric space $(X,d,\alpha)$ is {\em semi-effectively arcwise connected} if there exists a total computable multi-valued function $P:X^2\rightrightarrows\mathcal{A}_-(X)$ such that $P(x,y)$ is the set of all arcs $A$ whose two end points are $x$ and $y$, for any $x,y\in X$.
\end{definition}
Obviously $D$ is not semi-effectively arcwise connected.
Indeed, for every $\varepsilon>0$ there exists $x_0,x_1\in [0,1]$ with $d(x_0,x_1)<\varepsilon$ such that $\lrangle{x_0,0},\lrangle{x_1,0}\in D$ cannot be connected by any $\Pi^0_1$ arc.
Thus, we have the following corollary.

\begin{cor}
There exists an effectively pathwise connected Euclidean continuum $D$ such that $D$ is not semi-effectively arcwise connected.
\end{cor}

\begin{theorem}\label{thm:rite_alcom}
Not every $\Pi^0_1$ planar dendrite is almost computable.
\end{theorem}

To prove Theorem \ref{thm:rite_alcom}, we need to prepare some tools.
For a string $\sigma\in 2^{<\mathbb{N}}$, let $lh(\sigma)$ denote the length of $\sigma$.
Then
\[\psi(\sigma)=\left\lrangle{2^{-1}\cdot 3^{-i}+2\sum_{i<lh(\sigma)\;\&\; \sigma(i)=1}3^{-(i+1)},2^{-lh(\sigma)}\right}\in\mathbb{R}^2.\]
For two points $\vec{x},\vec{y}\in\mathbb{R}^2$, the closed line segment $L(\vec{x},\vec{y})$ from $\vec{x}$ to $\vec{y}$ is defined by $L(\vec{x},\vec{y})=\{(1-t)\vec{x}+t\vec{y}:t\in [0,1]\}$.
For a (possibly infinite) tree $T\subseteq 2^{<\mathbb{N}}$, we plot an embedded tree $\Psi(T)\subseteq\mathbb{R}^2$ by
\[\Psi(T)=cl\left(\bigcup\{L(\psi(\sigma),\psi(\tau)):\sigma,\tau\in T\;\&\;lh(\sigma)=lh(\tau)+1\}\right).\]
Then $\Psi(T)$ is a dendrite (but not necessarily a tree, in the sense of Continuum Theory), for any (possibly infinite) tree $T\subseteq 2^\mathbb{N}$.
See Fig.\ \ref{fig:Tree}.

\begin{figure}[t]\centering
  \begin{center}
%WinTpicVersion3.08
\unitlength 0.1in
\begin{picture}( 27.2000, 10.3000)(  0.3000,-12.2000)
% LINE 2 0 3 0
% 4 724 800 1324 400 1324 400 1924 800
% 
\special{pn 8}%
\special{pa 724 800}%
\special{pa 1324 400}%
\special{fp}%
\special{pa 1324 400}%
\special{pa 1924 800}%
\special{fp}%
% DOT 1 0 3 0
% 3 1324 400 724 800 1924 800
% 
\special{pn 13}%
\special{sh 1}%
\special{ar 1324 400 10 10 0  6.28318530717959E+0000}%
\special{sh 1}%
\special{ar 724 800 10 10 0  6.28318530717959E+0000}%
\special{sh 1}%
\special{ar 1924 800 10 10 0  6.28318530717959E+0000}%
% LINE 2 0 3 0
% 24 724 800 424 1000 724 800 1024 1000 1924 800 1624 1000 1924 800 2224 1000 424 1000 276 1100 424 1000 576 1100 1024 1000 876 1100 1024 1000 1176 1100 1624 1000 1476 1100 1624 1000 1776 1100 2224 1000 2076 1100 2224 1000 2376 1100
% 
\special{pn 8}%
\special{pa 724 800}%
\special{pa 424 1000}%
\special{fp}%
\special{pa 724 800}%
\special{pa 1024 1000}%
\special{fp}%
\special{pa 1924 800}%
\special{pa 1624 1000}%
\special{fp}%
\special{pa 1924 800}%
\special{pa 2224 1000}%
\special{fp}%
\special{pa 424 1000}%
\special{pa 276 1100}%
\special{fp}%
\special{pa 424 1000}%
\special{pa 576 1100}%
\special{fp}%
\special{pa 1024 1000}%
\special{pa 876 1100}%
\special{fp}%
\special{pa 1024 1000}%
\special{pa 1176 1100}%
\special{fp}%
\special{pa 1624 1000}%
\special{pa 1476 1100}%
\special{fp}%
\special{pa 1624 1000}%
\special{pa 1776 1100}%
\special{fp}%
\special{pa 2224 1000}%
\special{pa 2076 1100}%
\special{fp}%
\special{pa 2224 1000}%
\special{pa 2376 1100}%
\special{fp}%
% DOT 1 0 3 0
% 12 424 1000 276 1100 576 1100 876 1100 1024 1000 1176 1100 1476 1100 1624 1000 1776 1100 2076 1100 2224 1000 2376 1100
% 
\special{pn 13}%
\special{sh 1}%
\special{ar 424 1000 10 10 0  6.28318530717959E+0000}%
\special{sh 1}%
\special{ar 276 1100 10 10 0  6.28318530717959E+0000}%
\special{sh 1}%
\special{ar 576 1100 10 10 0  6.28318530717959E+0000}%
\special{sh 1}%
\special{ar 876 1100 10 10 0  6.28318530717959E+0000}%
\special{sh 1}%
\special{ar 1024 1000 10 10 0  6.28318530717959E+0000}%
\special{sh 1}%
\special{ar 1176 1100 10 10 0  6.28318530717959E+0000}%
\special{sh 1}%
\special{ar 1476 1100 10 10 0  6.28318530717959E+0000}%
\special{sh 1}%
\special{ar 1624 1000 10 10 0  6.28318530717959E+0000}%
\special{sh 1}%
\special{ar 1776 1100 10 10 0  6.28318530717959E+0000}%
\special{sh 1}%
\special{ar 2076 1100 10 10 0  6.28318530717959E+0000}%
\special{sh 1}%
\special{ar 2224 1000 10 10 0  6.28318530717959E+0000}%
\special{sh 1}%
\special{ar 2376 1100 10 10 0  6.28318530717959E+0000}%
% LINE 2 2 3 0
% 32 276 1100 200 1200 282 1115 350 1200 576 1100 500 1200 576 1100 650 1200 876 1100 800 1200 876 1100 950 1200 1176 1100 1100 1200 1176 1100 1250 1200 1476 1100 1400 1200 1476 1100 1550 1200 1776 1100 1700 1200 1776 1100 1850 1200 2076 1100 2000 1200 2076 1100 2150 1200 2376 1100 2300 1200 2376 1100 2450 1200
% 
\special{pn 8}%
\special{pa 276 1100}%
\special{pa 200 1200}%
\special{dt 0.045}%
\special{pa 282 1116}%
\special{pa 350 1200}%
\special{dt 0.045}%
\special{pa 576 1100}%
\special{pa 500 1200}%
\special{dt 0.045}%
\special{pa 576 1100}%
\special{pa 650 1200}%
\special{dt 0.045}%
\special{pa 876 1100}%
\special{pa 800 1200}%
\special{dt 0.045}%
\special{pa 876 1100}%
\special{pa 950 1200}%
\special{dt 0.045}%
\special{pa 1176 1100}%
\special{pa 1100 1200}%
\special{dt 0.045}%
\special{pa 1176 1100}%
\special{pa 1250 1200}%
\special{dt 0.045}%
\special{pa 1476 1100}%
\special{pa 1400 1200}%
\special{dt 0.045}%
\special{pa 1476 1100}%
\special{pa 1550 1200}%
\special{dt 0.045}%
\special{pa 1776 1100}%
\special{pa 1700 1200}%
\special{dt 0.045}%
\special{pa 1776 1100}%
\special{pa 1850 1200}%
\special{dt 0.045}%
\special{pa 2076 1100}%
\special{pa 2000 1200}%
\special{dt 0.045}%
\special{pa 2076 1100}%
\special{pa 2150 1200}%
\special{dt 0.045}%
\special{pa 2376 1100}%
\special{pa 2300 1200}%
\special{dt 0.045}%
\special{pa 2376 1100}%
\special{pa 2450 1200}%
\special{dt 0.045}%
% STR 2 0 3 0
% 3 1200 260 1200 360 2 0
% $\psi(\lrangle{})$
\put(12.0000,-3.6000){\makebox(0,0)[lb]{$\psi(\lrangle{})$}}%
% STR 2 0 3 0
% 3 460 640 460 740 2 0
% $\psi(\lrangle{0})$
\put(4.6000,-7.4000){\makebox(0,0)[lb]{$\psi(\lrangle{0})$}}%
% VECTOR 2 0 3 0
% 4 150 1200 150 200 150 1200 2750 1200
% 
\special{pn 8}%
\special{pa 150 1200}%
\special{pa 150 200}%
\special{fp}%
\special{sh 1}%
\special{pa 150 200}%
\special{pa 130 268}%
\special{pa 150 254}%
\special{pa 170 268}%
\special{pa 150 200}%
\special{fp}%
\special{pa 150 1200}%
\special{pa 2750 1200}%
\special{fp}%
\special{sh 1}%
\special{pa 2750 1200}%
\special{pa 2684 1180}%
\special{pa 2698 1200}%
\special{pa 2684 1220}%
\special{pa 2750 1200}%
\special{fp}%
% STR 2 0 3 0
% 3 1900 640 1900 740 2 0
% $\psi(\lrangle{1})$
\put(19.0000,-7.4000){\makebox(0,0)[lb]{$\psi(\lrangle{1})$}}%
% STR 2 0 3 0
% 3 200 840 200 940 2 0
% $\psi(\lrangle{00})$
\put(2.0000,-9.4000){\makebox(0,0)[lb]{$\psi(\lrangle{00})$}}%
% STR 2 0 3 0
% 3 920 840 920 940 2 0
% $\psi(\lrangle{01})$
\put(9.2000,-9.4000){\makebox(0,0)[lb]{$\psi(\lrangle{01})$}}%
% STR 2 0 3 0
% 3 1380 840 1380 940 2 0
% $\psi(\lrangle{10})$
\put(13.8000,-9.4000){\makebox(0,0)[lb]{$\psi(\lrangle{10})$}}%
% STR 2 0 3 0
% 3 2220 840 2220 940 2 0
% $\psi(\lrangle{11})$
\put(22.2000,-9.4000){\makebox(0,0)[lb]{$\psi(\lrangle{11})$}}%
% STR 2 0 3 0
% 3 110 1210 110 1310 2 0
% $0$
\put(1.1000,-13.1000){\makebox(0,0)[lb]{$0$}}%
% STR 2 0 3 0
% 3 2520 1210 2520 1310 2 0
% $1$
\put(25.2000,-13.1000){\makebox(0,0)[lb]{$1$}}%
% STR 2 0 3 0
% 3 30 330 30 430 2 0
% $1$
\put(0.3000,-4.3000){\makebox(0,0)[lb]{$1$}}%
% LINE 2 0 3 0
% 2 2520 1220 2520 1180
% 
\special{pn 8}%
\special{pa 2520 1220}%
\special{pa 2520 1180}%
\special{fp}%
% LINE 2 0 3 0
% 2 130 400 170 400
% 
\special{pn 8}%
\special{pa 130 400}%
\special{pa 170 400}%
\special{fp}%
\end{picture}%
  \end{center}
 \vspace{-0.5em}
\caption{The plotted tree $\Psi(2^{<\nn})$.}
  \label{fig:Tree}
\end{figure}

We can easily prove the following lemmata.

\begin{lemma}\label{lem:rite:1}
Let $T$ be a subtree of $2^{<\mathbb{N}}$, and $D$ be a planar subset such that $\psi(\lrangle{})\in D\subseteq \Psi(T)$ for the root $\lrangle{}\in 2^{<\mathbb{N}}$.
Then $D$ is a dendrite if and only if $D$ is homeomorphic to $\Psi(S)$ for a subtree $S\subseteq T$.
\end{lemma}

\begin{proof}\upshape
The ``if'' part is obvious.
Let $D$ be a dendrite.
For a binary string $\sigma$ which is not a root, let $\sigma^-$ be an immediate predecessor of $\sigma$.
We consider the set $S=\{\lrangle{}\}\cup\{\sigma\in 2^{<\mathbb{N}}:\sigma\not=\lrangle{}\;\&\;D\cap(L(\psi(\sigma^-),\psi(\sigma))\setminus\{\psi(\sigma^-)\})\not=\emptyset\}$.
Since $D$ is connected, $S$ is a subtree of $T$.
It is easy to see that $D$ is homeomorphic to $\Psi(S)$.
%%%%
\end{proof}

\begin{lemma}\label{lem:rite:2}
Let $T$ be a subtree of $2^{<\mathbb{N}}$.
Then $T$ is $\Pi^0_1$ (c.e., computable, resp.) if and only if $\Psi(T)$ is a $\Pi^0_1$ (c.e., computable, resp.) dendrite in $\mathbb{R}^2$.
\end{lemma}

\begin{proof}\upshape
With our definition of $\Psi$, the dendrite $\Psi(2^{<\mathbb{N}})$ is clearly a computable closed subset of $\mathbb{R}^2$.
So, if $T$ is $\Pi^0_1$, then it is easy to prove that $\Psi(T)$ is also $\Pi^0_1$.
Assume that $T$ is a c.e.\ tree.
At stage $s$, we compute whether $L(\psi(\sigma^-),\psi(\sigma))$ intersects with the $e$-th open rational ball $\rho_e$, for any $e<s$ and any string $\sigma$ which is already enumerated into $T$ by stage $s$.
If so, we enumerate $e$ into $W_T$ at stage $s$.
Then $\{e\in\nn:\Psi(T)\cap\rho_e\not=\emptyset\}=W_T$ is c.e.

Assume that $\Psi(T)$ is $\Pi^0_1$.
We consider an open rational ball $B_-(\sigma)=B(\psi(\sigma);2^{-(lh(\sigma)+2)})$ for each $\sigma\in 2^{<\mathbb{N}}$.
Note that $\hat{B}_-(\sigma)\cap\hat{B}_-(\tau)=\emptyset$ for $\sigma\not=\tau$.
Since $\Psi(T)$ is $\Pi^0_1$, $T^*=\{\sigma\in 2^{<\mathbb{N}}:\Psi(T)\cap\hat{B}_-(\sigma)=\emptyset\}$ is c.e., and it is easy to see that $T=2^{<\mathbb{N}}\setminus T^*$.
Thus, $T$ is a $\Pi^0_1$ tree of $2^{<\mathbb{N}}$.
We next assume that $\Psi(T)$ is c.e.
We can assume that $\Psi(T)$ contains the root $\psi(\lrangle{})$, otherwise $T=\emptyset$, and clearly it is c.e.
For a binary string $\sigma$ which is not a root, let $\sigma^-$ be an immediate predecessor of $\sigma$.
Pick an open rational ball $B_+(\sigma)$ such that $\Psi(2^{<\mathbb{N}})\cap B_+(\sigma)\subseteq L(\psi(\sigma^-),\psi(\sigma))$ for each $\sigma$.
Since $\Psi(T)$ is c.e., $T^*=\{\sigma\in 2^{<\mathbb{N}}:\Psi(T)\cap B_+(\sigma)\not=\emptyset\}$ is c.e.
If $\sigma$ is not a root and $\sigma\in T$ then $L(\psi(\sigma^-),\psi(\sigma))\subseteq\Psi(T)$, so $\Psi(T)\cap B_+(\sigma)\not=\emptyset$.
We observe that if $\sigma\not\in T$ then $L(\psi(\sigma^-),\psi(\sigma))\cap\Psi(T)=\emptyset$, so $\Psi(T)\cap B_+(\sigma)=\emptyset$.
Thus, we have $T=T^*$.
In the case that $\Psi(T)$ is computable, $\Psi(T)$ is c.e.\ and $\Pi^0_1$, hence $T$ is c.e.\ and $\Pi^0_1$, i.e., $T$ is computable.
%%%%
\end{proof}

\begin{lemma}\label{lem:rite:3}
Let $D$ be a computable subdendrite of $\Psi(2^{<\mathbb{N}})$.
Then there exists a computable subtree $T^+\subseteq 2^{<\mathbb{N}}$ such that $D\subseteq \Psi(T^+)$ and $([0,1]\times\{0\})\cap D=([0,1]\times\{0\})\cap\Psi(T^+)$.
\end{lemma}

\begin{proof}\upshape
We can assume $\psi(\lrangle{})\in D$, otherwise we connect $\psi(\lrangle{})$ and the root of $D$ by a subarc of $\Psi(2^{<\mathbb{N}})$.
Again we consider an open rational ball $B_-(\sigma)=B(\psi(\sigma);2^{-(lh(\sigma)+2)})$, and an open rational ball $B_+(\sigma)$ such that $\Psi(2^{<\mathbb{N}})\cap B_+(\sigma)\subseteq L(\psi(\sigma^-),\psi(\sigma))$ for each $\sigma\in 2^{<\nn}$.
Since $D$ is $\Pi^0_1$, $U^*=\{\sigma\in 2^{<\mathbb{N}}:D\cap\hat{B}_-(\sigma)=\emptyset\}$ is c.e.
Since $D$ is c.e., $T^*=\{\sigma\in 2^{<\mathbb{N}}:D\cap B_+(\sigma)\not=\emptyset\}$ is c.e., and it is a tree by Lemma \ref{lem:rite:1}.
For every $\sigma\in 2^{<\mathbb{N}}$, either $D\cap\hat{B}_-(\sigma)=\emptyset$ or $D\cap B_+(\sigma)\not=\emptyset$ holds.
Therefore, we have $T^*\cup U^*=2^{<\nn}$.
Moreover, for the set of {\em leaves of $T^*$}, $L_T^*=\{\rho\in T^*:(\forall i<2)\;\rho\fr\lrangle{i}\not\in T^*\}$, we observe that $T^*\cap U^*\subseteq L_T^*$.
Recall that the pointclass $\Sigma^0_1$ has the reduction property, that is, for two c.e.\ sets $T^*$ and $U^*$, there exist c.e.\ subsets $T\subseteq T^*$ and $U\subseteq U^*$ such that $T\cup U=T^*\cup U^*$ and $T\cap U=\emptyset$.
This is because, for $\sigma\in T^*\cap U^*$, $\sigma$ is enumerated into $T$ when $\sigma$ is enumerated into $T^*$ before it is enumerated into $U^*$; $\sigma$ is enumerated into $U$ otherwise.
Since $T^*\cap U^*\subseteq L_T^*$, $T$ must be tree.
Furthermore, $T$ is c.e., and $U=2^{<\mathbb{N}}\setminus T$ is also c.e.
Thus, $T$ is a computable tree.
Therefore, $T^+=\{\sigma\fr\lrangle{i}:\sigma\in T\;\&\;i<2\}$ is also a computable tree.
Then, $D\subseteq\Psi(T^+)$, and we have $([0,1]\times\{0\})\cap D=([0,1]\times\{0\})\cap\Psi(T^+)$ since the set of all infinite paths of $T$ and that of $T^+$ coincide.
%%%%
\end{proof}

Cenzer, Weber and Wu, and the author \cite{CKWW} introduced the notion of {\em tree-immunity} for closed sets in Cantor space $2^\mathbb{N}$.
For $\sigma\in 2^{<\mathbb{N}}$, define $I_\sigma$ as $\{f\in 2^\mathbb{N}:(\forall n<lh(\sigma))\;f(n)=\sigma(n)\}$.
Note that $\{I_\sigma:\sigma\in 2^{<\mathbb{N}}\}$ is a countable base for Cantor space.

\begin{definition}[Cenzer-Kihara-Weber-Wu \cite{CKWW}]
A nonempty closed set $F\subseteq 2^\mathbb{N}$ is said to be {\em tree-immune} if the tree $T_F=\{\sigma\in 2^{<\mathbb{N}}:F\cap I_\sigma\not=\emptyset\}\subseteq 2^{<\mathbb{N}}$ contains no infinite computable subtree.
\end{definition}

For a nonempty $\Pi^0_1$ subset $P\subseteq 2^\mathbb{N}$, the corresponding tree $T_P$ is $\Pi^0_1$, and it has no dead ends.
The set of {\em all complete consistent extensions of Peano Arithmetic} is an example of a tree-immune $\Pi^0_1$ subset of $2^\mathbb{N}$.
Tree-immune $\Pi^0_1$ sets have the following remarkable property.

\begin{lemma}\label{lem:rite:4}
Let $P$ be a tree-immune $\Pi^0_1$ subset of $2^\mathbb{N}$ and let $D\subseteq\Psi(T_P)$ be any computable subdendrite.
Then $([0,1]\times\{0\})\cap D=\emptyset$ holds.
\end{lemma}

\begin{proof}\upshape
By Lemma \ref{lem:rite:3}, there exists a computable subtree $T\subseteq 2^{<\mathbb{N}}$ such that $D\subseteq\Psi(T)$ and $\Psi(T)$ agrees with $D$ on $[0,1]\times\{0\}$.
Since $D\subseteq\Psi(T_P)$, and since $T_P$ has no dead ends, $T\subseteq T_P$ holds.
Since $P$ is tree-immune, $T$ must be finite.
By using weak K\"onig's lemma (or, equivalently, compactness of Cantor space), $T\subseteq 2^l$ holds for some $l\in\mathbb{N}$.
Thus, $D\subseteq\Psi(T)\subseteq [0,1]\times [2^{-l},1]$ as desired.
%%%%
\end{proof}

Note that if $P$ is a nonempty $\Pi^0_1$ set in Cantor space $2^\mathbb{N}$, then for every $\delta>0$ it holds that $((0,1)\times(0,\delta))\cap\Psi(T_P)\not=\emptyset$.
Finally, we are ready to prove Theorem \ref{thm:rite_alcom}.

\begin{proof}[Proof of Theorem \ref{thm:rite_alcom}]\upshape
Again, we adapt the construction in the proof of Theorem \ref{thm:main:dendrite1}.
We fix a nonempty tree-immune $\Pi^0_1$ set $P\subseteq 2^\mathbb{N}$.
For $\sigma\in 2^{<\mathbb{N}}$, put $E(\sigma)=\{\tau\in 2^{<\mathbb{N}}:\tau\supseteq\sigma\}$.
For a $\Pi^0_1$ tree $T_P\subseteq 2^{<\mathbb{N}}$, there exists a computable function $f_P:\mathbb{N}\to 2^{<\mathbb{N}}$ such that $T_P=2^{<\mathbb{N}}\setminus\bigcup_nE(f_P(n))$ and such that for each $\sigma\in 2^{<\mathbb{N}}$ and $s\in\mathbb{N}$ we have $\sigma\in\bigcup_{t<s}E(f_P(t))$ whenever $\sigma\fr 0,\sigma\fr 1\in\bigcup_{t<s}E(f_P(t))$.
For such a computable function $f_P:\mathbb{N}\to 2^{<\mathbb{N}}$, we let $T_{P,s}$ denote $2^{<\mathbb{N}}\setminus\bigcup_{t<s}E(f_P(t))$.
Then $T_{P,s}$ is a tree without dead ends, and $\{T_{P,s}:s\in\nn\}$ is computable uniformly in $s$.

\begin{construction}\upshape
Let $\vec{e}_1$ denote $\lrangle{1,0}\in\mathbb{R}^2$.
For a tree $T\subseteq 2^{<\mathbb{N}}$ and $w\in\mathbb{Q}$, we define $\Psi(T;w)$, {\em the edge of the fat approximation of the tree $T$ of width $w$}, by
\begin{align*}
\Psi(T;w)=cl\biggl(\bigcup\Big\{L&\left(\psi(\sigma)\pm(3^{-|\sigma|}\cdot w)\vec{e}_1,\psi(\tau)\pm(3^{-|\tau|}\cdot w)\vec{e}_1\right)\\
&:\pm\in\{-,+\}\;\&\;\sigma,\tau\in T\;\&\;lh(\sigma)=lh(\tau)+1\Big\}\biggr).
\end{align*}

If $\lim_s w_s=0$ then we have $\lim_s\Psi(T;w_s)=\Psi(T)$.
Moreover, if $\{w_s:s\in\mathbb{N}\}$ is a uniformly computable sequence of rational numbers, then $\{\Psi(T;w_s):s\in\mathbb{N}\}$ is also a uniformly computable sequence of computable closed sets.
Additionally, the set $\Psi(T;w,c,t,q)$, for a tree $T\subseteq 2^{<\mathbb{N}}$, for $w,c,q\in\mathbb{Q}$, and for $t\in\mathbb{N}$, is defined by
\[\Psi(T;w,c,t,q)=\left\{\left\lrangle{c+q\cdot\left(x-\frac{1}{2}\right),\frac{2-y}{2^{t+1}}\right}\in\mathbb{R}^2:\lrangle{x,y}\in\Psi(T;w)\right\}.\]

\begin{figure}[t]\centering
 \begin{minipage}{0.48\hsize}
  \begin{center}
%WinTpicVersion3.08
\unitlength 0.1in
\begin{picture}( 13.6000, 10.2000)(  0.3000,-12.2000)
% VECTOR 2 0 3 0
% 4 90 1200 90 200 90 1200 1390 1200
% 
\special{pn 8}%
\special{pa 90 1200}%
\special{pa 90 200}%
\special{fp}%
\special{sh 1}%
\special{pa 90 200}%
\special{pa 70 268}%
\special{pa 90 254}%
\special{pa 110 268}%
\special{pa 90 200}%
\special{fp}%
\special{pa 90 1200}%
\special{pa 1390 1200}%
\special{fp}%
\special{sh 1}%
\special{pa 1390 1200}%
\special{pa 1324 1180}%
\special{pa 1338 1200}%
\special{pa 1324 1220}%
\special{pa 1390 1200}%
\special{fp}%
% STR 2 0 3 0
% 3 70 1210 70 1310 2 0
% $0$
\put(0.7000,-13.1000){\makebox(0,0)[lb]{$0$}}%
% STR 2 0 3 0
% 3 1275 1210 1275 1310 2 0
% $1$
\put(12.7500,-13.1000){\makebox(0,0)[lb]{$1$}}%
% STR 2 0 3 0
% 3 30 330 30 430 2 0
% $1$
\put(0.3000,-4.3000){\makebox(0,0)[lb]{$1$}}%
% LINE 2 0 3 0
% 2 1275 1220 1275 1180
% 
\special{pn 8}%
\special{pa 1276 1220}%
\special{pa 1276 1180}%
\special{fp}%
% LINE 2 0 3 0
% 2 80 400 100 400
% 
\special{pn 8}%
\special{pa 80 400}%
\special{pa 100 400}%
\special{fp}%
% POLYGON 2 5 0 0
% 6 572 400 772 400 1022 800 922 800 922 800 572 400
% 
\special{pn 8}%
\special{sh 0.600}%
\special{pa 572 400}%
\special{pa 772 400}%
\special{pa 1022 800}%
\special{pa 922 800}%
\special{pa 922 800}%
\special{pa 572 400}%
\special{ip}%
% POLYGON 2 5 0 0
% 6 572 400 772 400 422 800 322 800 322 800 572 400
% 
\special{pn 8}%
\special{sh 0.600}%
\special{pa 572 400}%
\special{pa 772 400}%
\special{pa 422 800}%
\special{pa 322 800}%
\special{pa 322 800}%
\special{pa 572 400}%
\special{ip}%
% POLYGON 2 5 0 0
% 6 322 800 198 1000 248 1000 422 800 422 800 322 800
% 
\special{pn 8}%
\special{sh 0.600}%
\special{pa 322 800}%
\special{pa 198 1000}%
\special{pa 248 1000}%
\special{pa 422 800}%
\special{pa 422 800}%
\special{pa 322 800}%
\special{ip}%
% POLYGON 2 5 0 0
% 6 322 800 422 800 548 1000 498 1000 498 1000 322 800
% 
\special{pn 8}%
\special{sh 0.600}%
\special{pa 322 800}%
\special{pa 422 800}%
\special{pa 548 1000}%
\special{pa 498 1000}%
\special{pa 498 1000}%
\special{pa 322 800}%
\special{ip}%
% POLYGON 2 5 0 0
% 6 922 800 798 1000 848 1000 1022 800 1022 800 922 800
% 
\special{pn 8}%
\special{sh 0.600}%
\special{pa 922 800}%
\special{pa 798 1000}%
\special{pa 848 1000}%
\special{pa 1022 800}%
\special{pa 1022 800}%
\special{pa 922 800}%
\special{ip}%
% POLYGON 2 5 0 0
% 6 922 800 1098 1000 1148 1000 1022 800 1022 800 922 800
% 
\special{pn 8}%
\special{sh 0.600}%
\special{pa 922 800}%
\special{pa 1098 1000}%
\special{pa 1148 1000}%
\special{pa 1022 800}%
\special{pa 1022 800}%
\special{pa 922 800}%
\special{ip}%
% POLYGON 2 5 0 0
% 6 198 1000 135 1100 160 1100 248 1000 248 1000 198 1000
% 
\special{pn 8}%
\special{sh 0.600}%
\special{pa 198 1000}%
\special{pa 136 1100}%
\special{pa 160 1100}%
\special{pa 248 1000}%
\special{pa 248 1000}%
\special{pa 198 1000}%
\special{ip}%
% POLYGON 2 5 0 0
% 6 498 1000 435 1100 460 1100 548 1000 548 1000 498 1000
% 
\special{pn 8}%
\special{sh 0.600}%
\special{pa 498 1000}%
\special{pa 436 1100}%
\special{pa 460 1100}%
\special{pa 548 1000}%
\special{pa 548 1000}%
\special{pa 498 1000}%
\special{ip}%
% POLYGON 2 5 0 0
% 6 798 1000 735 1100 760 1100 848 1000 848 1000 798 1000
% 
\special{pn 8}%
\special{sh 0.600}%
\special{pa 798 1000}%
\special{pa 736 1100}%
\special{pa 760 1100}%
\special{pa 848 1000}%
\special{pa 848 1000}%
\special{pa 798 1000}%
\special{ip}%
% POLYGON 2 5 0 0
% 6 1095 1000 1032 1100 1058 1100 1145 1000 1145 1000 1095 1000
% 
\special{pn 8}%
\special{sh 0.600}%
\special{pa 1096 1000}%
\special{pa 1032 1100}%
\special{pa 1058 1100}%
\special{pa 1146 1000}%
\special{pa 1146 1000}%
\special{pa 1096 1000}%
\special{ip}%
% POLYGON 2 5 0 0
% 6 1148 1000 1210 1100 1185 1100 1098 1000 1098 1000 1148 1000
% 
\special{pn 8}%
\special{sh 0.600}%
\special{pa 1148 1000}%
\special{pa 1210 1100}%
\special{pa 1186 1100}%
\special{pa 1098 1000}%
\special{pa 1098 1000}%
\special{pa 1148 1000}%
\special{ip}%
% POLYGON 2 5 0 0
% 6 848 1000 910 1100 885 1100 798 1000 798 1000 848 1000
% 
\special{pn 8}%
\special{sh 0.600}%
\special{pa 848 1000}%
\special{pa 910 1100}%
\special{pa 886 1100}%
\special{pa 798 1000}%
\special{pa 798 1000}%
\special{pa 848 1000}%
\special{ip}%
% POLYGON 2 5 0 0
% 6 548 1000 610 1100 585 1100 498 1000 498 1000 548 1000
% 
\special{pn 8}%
\special{sh 0.600}%
\special{pa 548 1000}%
\special{pa 610 1100}%
\special{pa 586 1100}%
\special{pa 498 1000}%
\special{pa 498 1000}%
\special{pa 548 1000}%
\special{ip}%
% POLYGON 2 5 0 0
% 6 248 1000 310 1100 285 1100 198 1000 198 1000 248 1000
% 
\special{pn 8}%
\special{sh 0.600}%
\special{pa 248 1000}%
\special{pa 310 1100}%
\special{pa 286 1100}%
\special{pa 198 1000}%
\special{pa 198 1000}%
\special{pa 248 1000}%
\special{ip}%
% POLYGON 2 5 1 0
% 7 135 1100 90 1200 205 1200 160 1100 135 1100 135 1100 135 1100
% 
\special{pn 8}%
\special{sh 0.300}%
\special{pa 136 1100}%
\special{pa 90 1200}%
\special{pa 206 1200}%
\special{pa 160 1100}%
\special{pa 136 1100}%
\special{pa 136 1100}%
\special{pa 136 1100}%
\special{ip}%
% POLYGON 2 5 1 0
% 7 285 1100 240 1200 355 1200 310 1100 285 1100 285 1100 285 1100
% 
\special{pn 8}%
\special{sh 0.300}%
\special{pa 286 1100}%
\special{pa 240 1200}%
\special{pa 356 1200}%
\special{pa 310 1100}%
\special{pa 286 1100}%
\special{pa 286 1100}%
\special{pa 286 1100}%
\special{ip}%
% POLYGON 2 5 1 0
% 7 435 1100 390 1200 505 1200 460 1100 435 1100 435 1100 435 1100
% 
\special{pn 8}%
\special{sh 0.300}%
\special{pa 436 1100}%
\special{pa 390 1200}%
\special{pa 506 1200}%
\special{pa 460 1100}%
\special{pa 436 1100}%
\special{pa 436 1100}%
\special{pa 436 1100}%
\special{ip}%
% POLYGON 2 5 1 0
% 7 585 1100 540 1200 655 1200 610 1100 585 1100 585 1100 585 1100
% 
\special{pn 8}%
\special{sh 0.300}%
\special{pa 586 1100}%
\special{pa 540 1200}%
\special{pa 656 1200}%
\special{pa 610 1100}%
\special{pa 586 1100}%
\special{pa 586 1100}%
\special{pa 586 1100}%
\special{ip}%
% POLYGON 2 5 1 0
% 7 735 1100 690 1200 805 1200 760 1100 735 1100 735 1100 735 1100
% 
\special{pn 8}%
\special{sh 0.300}%
\special{pa 736 1100}%
\special{pa 690 1200}%
\special{pa 806 1200}%
\special{pa 760 1100}%
\special{pa 736 1100}%
\special{pa 736 1100}%
\special{pa 736 1100}%
\special{ip}%
% POLYGON 2 5 1 0
% 7 885 1100 840 1200 955 1200 910 1100 885 1100 885 1100 885 1100
% 
\special{pn 8}%
\special{sh 0.300}%
\special{pa 886 1100}%
\special{pa 840 1200}%
\special{pa 956 1200}%
\special{pa 910 1100}%
\special{pa 886 1100}%
\special{pa 886 1100}%
\special{pa 886 1100}%
\special{ip}%
% POLYGON 2 5 1 0
% 7 1035 1100 990 1200 1105 1200 1060 1100 1035 1100 1035 1100 1035 1100
% 
\special{pn 8}%
\special{sh 0.300}%
\special{pa 1036 1100}%
\special{pa 990 1200}%
\special{pa 1106 1200}%
\special{pa 1060 1100}%
\special{pa 1036 1100}%
\special{pa 1036 1100}%
\special{pa 1036 1100}%
\special{ip}%
% POLYGON 2 5 1 0
% 7 1185 1100 1140 1200 1255 1200 1210 1100 1185 1100 1185 1100 1185 1100
% 
\special{pn 8}%
\special{sh 0.300}%
\special{pa 1186 1100}%
\special{pa 1140 1200}%
\special{pa 1256 1200}%
\special{pa 1210 1100}%
\special{pa 1186 1100}%
\special{pa 1186 1100}%
\special{pa 1186 1100}%
\special{ip}%
\end{picture}%
  \end{center}
 \vspace{-0.5em}
\caption{The fat approximation $\Psi(T;w)$.}
  \label{fig:Treec}
 \end{minipage}
 \begin{minipage}{0.48\hsize}
  \begin{center}
%WinTpicVersion3.08
\unitlength 0.1in
\begin{picture}( 20.0000, 11.3000)(  2.0000,-12.2000)
% VECTOR 2 0 3 0
% 4 600 1200 600 200 600 1200 2200 1200
% 
\special{pn 8}%
\special{pa 600 1200}%
\special{pa 600 200}%
\special{fp}%
\special{sh 1}%
\special{pa 600 200}%
\special{pa 580 268}%
\special{pa 600 254}%
\special{pa 620 268}%
\special{pa 600 200}%
\special{fp}%
\special{pa 600 1200}%
\special{pa 2200 1200}%
\special{fp}%
\special{sh 1}%
\special{pa 2200 1200}%
\special{pa 2134 1180}%
\special{pa 2148 1200}%
\special{pa 2134 1220}%
\special{pa 2200 1200}%
\special{fp}%
% LINE 2 0 3 0
% 4 580 400 620 400 580 800 620 800
% 
\special{pn 8}%
\special{pa 580 400}%
\special{pa 620 400}%
\special{fp}%
\special{pa 580 800}%
\special{pa 620 800}%
\special{fp}%
% LINE 2 0 3 0
% 2 1400 1220 1400 1180
% 
\special{pn 8}%
\special{pa 1400 1220}%
\special{pa 1400 1180}%
\special{fp}%
% STR 2 0 3 0
% 3 360 350 360 450 2 0
% $2^{-t}$
\put(3.6000,-4.5000){\makebox(0,0)[lb]{$2^{-t}$}}%
% STR 2 0 3 0
% 3 200 750 200 850 2 0
% $2^{-(t+1)}$
\put(2.0000,-8.5000){\makebox(0,0)[lb]{$2^{-(t+1)}$}}%
% STR 2 0 3 0
% 3 1370 1220 1370 1320 2 0
% $c$
\put(13.7000,-13.2000){\makebox(0,0)[lb]{$c$}}%
% LINE 2 2 3 0
% 2 1400 1200 1400 800
% 
\special{pn 8}%
\special{pa 1400 1200}%
\special{pa 1400 800}%
\special{dt 0.045}%
% STR 2 0 3 0
% 3 1370 160 1370 260 2 0
% $q$
\put(13.7000,-2.6000){\makebox(0,0)[lb]{$q$}}%
% LINE 2 2 3 0
% 4 600 800 1400 800 600 400 800 400
% 
\special{pn 8}%
\special{pa 600 800}%
\special{pa 1400 800}%
\special{dt 0.045}%
\special{pa 600 400}%
\special{pa 800 400}%
\special{dt 0.045}%
% POLYGON 2 5 0 0
% 6 1302 800 1502 800 1752 600 1652 600 1652 600 1302 800
% 
\special{pn 8}%
\special{sh 0.600}%
\special{pa 1302 800}%
\special{pa 1502 800}%
\special{pa 1752 600}%
\special{pa 1652 600}%
\special{pa 1652 600}%
\special{pa 1302 800}%
\special{ip}%
% POLYGON 2 5 0 0
% 6 1302 800 1502 800 1152 600 1052 600 1052 600 1302 800
% 
\special{pn 8}%
\special{sh 0.600}%
\special{pa 1302 800}%
\special{pa 1502 800}%
\special{pa 1152 600}%
\special{pa 1052 600}%
\special{pa 1052 600}%
\special{pa 1302 800}%
\special{ip}%
% POLYGON 2 5 0 0
% 6 1052 600 928 500 978 500 1152 600 1152 600 1052 600
% 
\special{pn 8}%
\special{sh 0.600}%
\special{pa 1052 600}%
\special{pa 928 500}%
\special{pa 978 500}%
\special{pa 1152 600}%
\special{pa 1152 600}%
\special{pa 1052 600}%
\special{ip}%
% POLYGON 2 5 0 0
% 6 1052 600 1152 600 1278 500 1228 500 1228 500 1052 600
% 
\special{pn 8}%
\special{sh 0.600}%
\special{pa 1052 600}%
\special{pa 1152 600}%
\special{pa 1278 500}%
\special{pa 1228 500}%
\special{pa 1228 500}%
\special{pa 1052 600}%
\special{ip}%
% POLYGON 2 5 0 0
% 6 1652 600 1528 500 1578 500 1752 600 1752 600 1652 600
% 
\special{pn 8}%
\special{sh 0.600}%
\special{pa 1652 600}%
\special{pa 1528 500}%
\special{pa 1578 500}%
\special{pa 1752 600}%
\special{pa 1752 600}%
\special{pa 1652 600}%
\special{ip}%
% POLYGON 2 5 0 0
% 6 1652 600 1828 500 1878 500 1752 600 1752 600 1652 600
% 
\special{pn 8}%
\special{sh 0.600}%
\special{pa 1652 600}%
\special{pa 1828 500}%
\special{pa 1878 500}%
\special{pa 1752 600}%
\special{pa 1752 600}%
\special{pa 1652 600}%
\special{ip}%
% POLYGON 2 5 0 0
% 6 928 500 865 450 890 450 978 500 978 500 928 500
% 
\special{pn 8}%
\special{sh 0.600}%
\special{pa 928 500}%
\special{pa 866 450}%
\special{pa 890 450}%
\special{pa 978 500}%
\special{pa 978 500}%
\special{pa 928 500}%
\special{ip}%
% POLYGON 2 5 0 0
% 6 1228 500 1165 450 1190 450 1278 500 1278 500 1228 500
% 
\special{pn 8}%
\special{sh 0.600}%
\special{pa 1228 500}%
\special{pa 1166 450}%
\special{pa 1190 450}%
\special{pa 1278 500}%
\special{pa 1278 500}%
\special{pa 1228 500}%
\special{ip}%
% POLYGON 2 5 0 0
% 6 1528 500 1465 450 1490 450 1578 500 1578 500 1528 500
% 
\special{pn 8}%
\special{sh 0.600}%
\special{pa 1528 500}%
\special{pa 1466 450}%
\special{pa 1490 450}%
\special{pa 1578 500}%
\special{pa 1578 500}%
\special{pa 1528 500}%
\special{ip}%
% POLYGON 2 5 0 0
% 6 1825 500 1762 450 1788 450 1875 500 1875 500 1825 500
% 
\special{pn 8}%
\special{sh 0.600}%
\special{pa 1826 500}%
\special{pa 1762 450}%
\special{pa 1788 450}%
\special{pa 1876 500}%
\special{pa 1876 500}%
\special{pa 1826 500}%
\special{ip}%
% POLYGON 2 5 0 0
% 6 1878 500 1940 450 1915 450 1828 500 1828 500 1878 500
% 
\special{pn 8}%
\special{sh 0.600}%
\special{pa 1878 500}%
\special{pa 1940 450}%
\special{pa 1916 450}%
\special{pa 1828 500}%
\special{pa 1828 500}%
\special{pa 1878 500}%
\special{ip}%
% POLYGON 2 5 0 0
% 6 1578 500 1640 450 1615 450 1528 500 1528 500 1578 500
% 
\special{pn 8}%
\special{sh 0.600}%
\special{pa 1578 500}%
\special{pa 1640 450}%
\special{pa 1616 450}%
\special{pa 1528 500}%
\special{pa 1528 500}%
\special{pa 1578 500}%
\special{ip}%
% POLYGON 2 5 0 0
% 6 1278 500 1340 450 1315 450 1228 500 1228 500 1278 500
% 
\special{pn 8}%
\special{sh 0.600}%
\special{pa 1278 500}%
\special{pa 1340 450}%
\special{pa 1316 450}%
\special{pa 1228 500}%
\special{pa 1228 500}%
\special{pa 1278 500}%
\special{ip}%
% POLYGON 2 5 0 0
% 6 978 500 1040 450 1015 450 928 500 928 500 978 500
% 
\special{pn 8}%
\special{sh 0.600}%
\special{pa 978 500}%
\special{pa 1040 450}%
\special{pa 1016 450}%
\special{pa 928 500}%
\special{pa 928 500}%
\special{pa 978 500}%
\special{ip}%
% POLYGON 2 5 1 0
% 7 865 450 820 400 935 400 890 450 865 450 865 450 865 450
% 
\special{pn 8}%
\special{sh 0.300}%
\special{pa 866 450}%
\special{pa 820 400}%
\special{pa 936 400}%
\special{pa 890 450}%
\special{pa 866 450}%
\special{pa 866 450}%
\special{pa 866 450}%
\special{ip}%
% POLYGON 2 5 1 0
% 7 1015 450 970 400 1085 400 1040 450 1015 450 1015 450 1015 450
% 
\special{pn 8}%
\special{sh 0.300}%
\special{pa 1016 450}%
\special{pa 970 400}%
\special{pa 1086 400}%
\special{pa 1040 450}%
\special{pa 1016 450}%
\special{pa 1016 450}%
\special{pa 1016 450}%
\special{ip}%
% POLYGON 2 5 1 0
% 7 1165 450 1120 400 1235 400 1190 450 1165 450 1165 450 1165 450
% 
\special{pn 8}%
\special{sh 0.300}%
\special{pa 1166 450}%
\special{pa 1120 400}%
\special{pa 1236 400}%
\special{pa 1190 450}%
\special{pa 1166 450}%
\special{pa 1166 450}%
\special{pa 1166 450}%
\special{ip}%
% POLYGON 2 5 1 0
% 7 1315 450 1270 400 1385 400 1340 450 1315 450 1315 450 1315 450
% 
\special{pn 8}%
\special{sh 0.300}%
\special{pa 1316 450}%
\special{pa 1270 400}%
\special{pa 1386 400}%
\special{pa 1340 450}%
\special{pa 1316 450}%
\special{pa 1316 450}%
\special{pa 1316 450}%
\special{ip}%
% POLYGON 2 5 1 0
% 7 1465 450 1420 400 1535 400 1490 450 1465 450 1465 450 1465 450
% 
\special{pn 8}%
\special{sh 0.300}%
\special{pa 1466 450}%
\special{pa 1420 400}%
\special{pa 1536 400}%
\special{pa 1490 450}%
\special{pa 1466 450}%
\special{pa 1466 450}%
\special{pa 1466 450}%
\special{ip}%
% POLYGON 2 5 1 0
% 7 1615 450 1570 400 1685 400 1640 450 1615 450 1615 450 1615 450
% 
\special{pn 8}%
\special{sh 0.300}%
\special{pa 1616 450}%
\special{pa 1570 400}%
\special{pa 1686 400}%
\special{pa 1640 450}%
\special{pa 1616 450}%
\special{pa 1616 450}%
\special{pa 1616 450}%
\special{ip}%
% POLYGON 2 5 1 0
% 7 1765 450 1720 400 1835 400 1790 450 1765 450 1765 450 1765 450
% 
\special{pn 8}%
\special{sh 0.300}%
\special{pa 1766 450}%
\special{pa 1720 400}%
\special{pa 1836 400}%
\special{pa 1790 450}%
\special{pa 1766 450}%
\special{pa 1766 450}%
\special{pa 1766 450}%
\special{ip}%
% POLYGON 2 5 1 0
% 7 1915 450 1870 400 1985 400 1940 450 1915 450 1915 450 1915 450
% 
\special{pn 8}%
\special{sh 0.300}%
\special{pa 1916 450}%
\special{pa 1870 400}%
\special{pa 1986 400}%
\special{pa 1940 450}%
\special{pa 1916 450}%
\special{pa 1916 450}%
\special{pa 1916 450}%
\special{ip}%
% VECTOR 2 0 3 0
% 4 1400 300 2000 300 1400 300 800 300
% 
\special{pn 8}%
\special{pa 1400 300}%
\special{pa 2000 300}%
\special{fp}%
\special{sh 1}%
\special{pa 2000 300}%
\special{pa 1934 280}%
\special{pa 1948 300}%
\special{pa 1934 320}%
\special{pa 2000 300}%
\special{fp}%
\special{pa 1400 300}%
\special{pa 800 300}%
\special{fp}%
\special{sh 1}%
\special{pa 800 300}%
\special{pa 868 320}%
\special{pa 854 300}%
\special{pa 868 280}%
\special{pa 800 300}%
\special{fp}%
\end{picture}%
  \end{center}
 \vspace{-0.5em}
\caption{The basic object $\Psi(T;w,c,t,q)$.}
  \label{fig:Tree2}
 \end{minipage}
\end{figure}

Note that $\Psi(T;w,c,t,q)\subseteq [c-q/2,c+q/2]\times [2^{-(t+1)},2^{-t}]$ as in Fig. \ref{fig:Tree2}.
For $t\in\mathbb{N}$, and for ${\rm st}^A(t)=\min\{s:t\in A_s\}$ in the proof of Theorem \ref{thm:main:dendrite1}, let $l(t)\in 2^\mathbb{N}$ be the leftmost path of $T_{P,{\rm st}^A(t)}$.
If ${\rm st}^A(t)$ is undefined (i.e., $t\not\in A$) then $l(t)$ is also undefined.
For each $t\in\mathbb{N}$ we define $F(t)=\{\sigma\in 2^{<\mathbb{N}}:\sigma\subseteq l(t)\}$ if $l(t)$ is defined; $F(t)=T_P$ otherwise.
Then $\{F(t):t\in\mathbb{N}\}$ is a computable sequence of $\Pi^0_1$ subsets of $2^{<\mathbb{N}}$.
Furthermore, we have $\Psi(F(t))\cap([0,1]\times\{0\})\not=\emptyset$, since $F(t)$ has a path for every $t\in\mathbb{N}$.
For each $t\in\nn$, $w(t)$ is defined again as in the proof of Theorem \ref{thm:main:dendrite1}.
Now we define a $\Pi^0_1$ dendrite $H\subseteq\mathbb{R}^2$ as follows:
\begin{align*}
H^*_{t}&=\Psi(F(t);w(t),2^{-t},t,2^{-(t+2)})\\
H^0_{t}&=(\{2^{-t}-w(t)\}\cup\{2^{-t}+w(t)\})\times [0,2^{-(t+1)}]\\
H^{**}_{t}&=(2^{-t}-w(t),2^{-t}+w(t))\times\{2^{-(t+1)}\}\\
H^2_{t}&=(2^{-t}-w(t),2^{-t}+w(t))\times(-1,2^{-(t+1)})\\
H&=\Big(\bigcup_{t\in\mathbb{N}}\left(H^*_{t}\cup H^0_{t}\setminus(H^{**}_t\cup int H^*_t)\right)\Big)\cup\Big(([-1,1]\times \{0\})\setminus\bigcup_{t\in\mathbb{N}}H^2_{t}\Big).
\end{align*}
\end{construction}

Put $H_t=H^*_{t}\setminus(H^{**}_t\cup int H^*_t)$ (see Fig.\ \ref{fig:Tree_dendrite}).
We can also show that $H$ is a $\Pi^0_1$ dendrite in the same way as for Theorem \ref{thm:main:dendrite1}.

\begin{claim}
The $\Pi^0_1$ dendrite $H$ does not $\ast$-include a computable dendrite.
\end{claim}

Let $J$ be a computable subdendrite of $H$.
Put $S(t)=[3\cdot 2^{-(t+2)},5\cdot 2^{-(t+2)}]\times[2^{-(t+1)},2^{-t}]$.
Then, we note that $J(t)=J\cap S(t)$ is also a computable dendrite, since $H_t\subseteq S(t)$ and it is a dendrite.
However, by Lemma \ref{lem:rite:4}, if $t\not\in A$ then we have $J(t)\cap(\mathbb{R}\times\{2^{-t}\})=\emptyset$.
So we consider the following set:
\begin{align*}
C=\{t\in\mathbb{N}:J(t)\cap\left([3\cdot 2^{-(t+2)},5\cdot 2^{-(t+2)}]\times[2^{-t},1]\right)=\emptyset\}.
\end{align*}
Since $J(t)$ is uniformly computable in $t$, the set $C$ is clearly c.e., and we have $\mathbb{N}\setminus A\subseteq C$.
However, if $\mathbb{N}\setminus A=C$, then this contradicts the incomputability of $A$.
Thus, there must be infinitely many $t\in A$ such that $t$ is enumerated into $C$.
However, if $t\in A$ is enumerated into $C$, it {\em cuts} the dendrite $H$.
In other words, since $J\subseteq H$ is connected, either $J\subseteq [-1,5\cdot 2^{-(t+2)}]\times\mathbb{R}$ or $J\subseteq [3\cdot 2^{-(t+2)},1]\times\mathbb{R}$.
Hence we must have $d_H(J,H)\geq 1$.
%%%%
\end{proof}

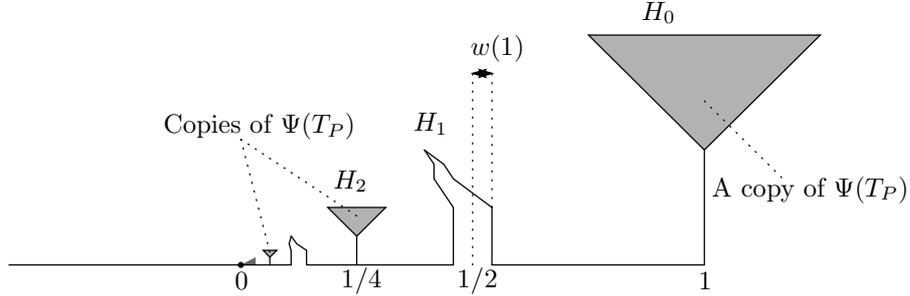
\begin{figure}[t]\centering
  \begin{center}
%WinTpicVersion3.08
\unitlength 0.1in
\begin{picture}( 42.0000, 14.3300)(  2.0000,-14.5800)
% POLYGON 2 5 0 0
% 6 1400 1455 1476 1455 1476 1415 1400 1455 1400 1455 1400 1455
% 
\special{pn 8}%
\special{sh 0.600}%
\special{pa 1400 1456}%
\special{pa 1476 1456}%
\special{pa 1476 1416}%
\special{pa 1400 1456}%
\special{pa 1400 1456}%
\special{pa 1400 1456}%
\special{ip}%
% LINE 2 0 3 0
% 2 1400 1455 200 1455
% 
\special{pn 8}%
\special{pa 1400 1456}%
\special{pa 200 1456}%
\special{fp}%
% STR 2 0 3 0
% 3 3475 95 3475 195 2 0
% $H_0$
\put(34.7500,-1.9500){\makebox(0,0)[lb]{$H_0$}}%
% STR 2 0 3 0
% 3 2295 675 2295 775 2 0
% $H_1$
\put(22.9500,-7.7500){\makebox(0,0)[lb]{$H_1$}}%
% DOT 1 0 3 0
% 1 1400 1455
% 
\special{pn 13}%
\special{sh 1}%
\special{ar 1400 1456 10 10 0  6.28318530717959E+0000}%
% STR 2 0 3 0
% 3 1370 1485 1370 1585 2 0
% $0$
\put(13.7000,-15.8500){\makebox(0,0)[lb]{$0$}}%
% STR 2 0 3 0
% 3 1920 1505 1920 1605 2 0
% $1/4$
\put(19.2000,-16.0500){\makebox(0,0)[lb]{$1/4$}}%
% STR 2 0 3 0
% 3 3770 1485 3770 1585 2 0
% $1$
\put(37.7000,-15.8500){\makebox(0,0)[lb]{$1$}}%
% STR 2 0 3 0
% 3 2520 1505 2520 1605 2 0
% $1/2$
\put(25.2000,-16.0500){\makebox(0,0)[lb]{$1/2$}}%
% STR 2 0 3 0
% 3 1880 975 1880 1075 2 0
% $H_2$
\put(18.8000,-10.7500){\makebox(0,0)[lb]{$H_2$}}%
% STR 2 0 3 0
% 3 2590 295 2590 395 2 0
% $w(1)$
\put(25.9000,-3.9500){\makebox(0,0)[lb]{$w(1)$}}%
% LINE 2 0 3 0
% 22 3800 855 3800 1455 3800 1455 2700 1455 2700 1455 2700 1155 2500 1455 2500 1155 2500 1455 2000 1455 2000 1455 2000 1305 2000 1455 1740 1455 1740 1455 1740 1380 1660 1455 1660 1380 1660 1455 1400 1455 1550 1455 1550 1415
% 
\special{pn 8}%
\special{pa 3800 856}%
\special{pa 3800 1456}%
\special{fp}%
\special{pa 3800 1456}%
\special{pa 2700 1456}%
\special{fp}%
\special{pa 2700 1456}%
\special{pa 2700 1156}%
\special{fp}%
\special{pa 2500 1456}%
\special{pa 2500 1156}%
\special{fp}%
\special{pa 2500 1456}%
\special{pa 2000 1456}%
\special{fp}%
\special{pa 2000 1456}%
\special{pa 2000 1306}%
\special{fp}%
\special{pa 2000 1456}%
\special{pa 1740 1456}%
\special{fp}%
\special{pa 1740 1456}%
\special{pa 1740 1380}%
\special{fp}%
\special{pa 1660 1456}%
\special{pa 1660 1380}%
\special{fp}%
\special{pa 1660 1456}%
\special{pa 1400 1456}%
\special{fp}%
\special{pa 1550 1456}%
\special{pa 1550 1416}%
\special{fp}%
% POLYGON 2 0 1 0
% 5 3800 855 3200 255 4400 255 4400 255 3800 855
% 
\special{pn 8}%
\special{sh 0.300}%
\special{pa 3800 856}%
\special{pa 3200 256}%
\special{pa 4400 256}%
\special{pa 4400 256}%
\special{pa 3800 856}%
\special{fp}%
% POLYGON 2 0 1 0
% 5 2000 1305 1850 1155 2150 1155 2150 1155 2000 1305
% 
\special{pn 8}%
\special{sh 0.300}%
\special{pa 2000 1306}%
\special{pa 1850 1156}%
\special{pa 2150 1156}%
\special{pa 2150 1156}%
\special{pa 2000 1306}%
\special{fp}%
% POLYGON 2 0 1 0
% 6 1550 1415 1585 1380 1515 1380 1550 1415 1550 1415 1550 1415
% 
\special{pn 8}%
\special{sh 0.300}%
\special{pa 1550 1416}%
\special{pa 1586 1380}%
\special{pa 1516 1380}%
\special{pa 1550 1416}%
\special{pa 1550 1416}%
\special{pa 1550 1416}%
\special{fp}%
% LINE 2 0 3 0
% 12 2500 1155 2400 1005 2700 1155 2500 1005 2500 1005 2450 930 2400 1005 2400 930 2400 930 2350 855 2350 855 2450 930
% 
\special{pn 8}%
\special{pa 2500 1156}%
\special{pa 2400 1006}%
\special{fp}%
\special{pa 2700 1156}%
\special{pa 2500 1006}%
\special{fp}%
\special{pa 2500 1006}%
\special{pa 2450 930}%
\special{fp}%
\special{pa 2400 1006}%
\special{pa 2400 930}%
\special{fp}%
\special{pa 2400 930}%
\special{pa 2350 856}%
\special{fp}%
\special{pa 2350 856}%
\special{pa 2450 930}%
\special{fp}%
% LINE 2 0 3 0
% 8 1660 1380 1650 1345 1740 1380 1690 1345 1690 1345 1660 1305 1660 1305 1650 1345
% 
\special{pn 8}%
\special{pa 1660 1380}%
\special{pa 1650 1346}%
\special{fp}%
\special{pa 1740 1380}%
\special{pa 1690 1346}%
\special{fp}%
\special{pa 1690 1346}%
\special{pa 1660 1306}%
\special{fp}%
\special{pa 1660 1306}%
\special{pa 1650 1346}%
\special{fp}%
% LINE 2 2 3 0
% 4 2600 1455 2600 455 2700 1155 2700 455
% 
\special{pn 8}%
\special{pa 2600 1456}%
\special{pa 2600 456}%
\special{dt 0.045}%
\special{pa 2700 1156}%
\special{pa 2700 456}%
\special{dt 0.045}%
% VECTOR 2 0 3 0
% 4 2650 455 2600 455 2650 455 2700 455
% 
\special{pn 8}%
\special{pa 2650 456}%
\special{pa 2600 456}%
\special{fp}%
\special{sh 1}%
\special{pa 2600 456}%
\special{pa 2668 476}%
\special{pa 2654 456}%
\special{pa 2668 436}%
\special{pa 2600 456}%
\special{fp}%
\special{pa 2650 456}%
\special{pa 2700 456}%
\special{fp}%
\special{sh 1}%
\special{pa 2700 456}%
\special{pa 2634 436}%
\special{pa 2648 456}%
\special{pa 2634 476}%
\special{pa 2700 456}%
\special{fp}%
% LINE 2 2 3 0
% 2 3800 600 4200 1000
% 
\special{pn 8}%
\special{pa 3800 600}%
\special{pa 4200 1000}%
\special{dt 0.045}%
% STR 2 0 3 0
% 3 3840 1030 3840 1130 2 0
% A copy of $\Psi(T_P)$
\put(38.4000,-11.3000){\makebox(0,0)[lb]{A copy of $\Psi(T_P)$}}%
% LINE 2 2 3 0
% 4 2000 1200 1400 800 1400 800 1550 1390
% 
\special{pn 8}%
\special{pa 2000 1200}%
\special{pa 1400 800}%
\special{dt 0.045}%
\special{pa 1400 800}%
\special{pa 1550 1390}%
\special{dt 0.045}%
% STR 2 0 3 0
% 3 1000 700 1000 800 2 0
% Copies of $\Psi(T_P)$
\put(10.0000,-8.0000){\makebox(0,0)[lb]{Copies of $\Psi(T_P)$}}%
\end{picture}%
  \end{center}
 \vspace{-0.5em}
\caption{The dendrite $H$ for $0,2,4\not\in A$ and $1,3\in A$.}
  \label{fig:Tree_dendrite}
\end{figure}

\begin{cor}
There exists a nonempty $\Pi^0_1$ subset of $[0,1]^2$ which is contractible, locally contractible, and $*$-includes no connected computable closed subset.
\end{cor}

%begin{cor}
%Not every locally simply connected planar $\Pi^0_1$ continuum is almost computable.
%\end{cor}

\section{Incomputability of Dendroids}

\begin{theorem}\label{thm:roid:alinc}
Not every computable planar dendroid $\ast$-includes a $\Pi^0_1$ dendrite.
\end{theorem}

\begin{lemma}\label{lem:comp_roid}
There exists a limit computable function $f$ such that, for every uniformly c.e.\ sequence $\{U_n:n\in\mathbb{N}\}$ of cofinite c.e.\ sets, we have $f(n)\in U_n$ for almost all $n\in\mathbb{N}$.
\end{lemma}

\begin{proof}\upshape
Let $\{V_e:e\in\mathbb{N}\}$ be an effective enumeration of all uniformly c.e.\ non-increasing sequences $\{U_n:n\in\mathbb{N}\}$ of c.e.\ sets such that $\min U_n\geq n$, where $(V_e)_n=U_n=\{x\in\mathbb{N}:(n,x)\in V_e\}$.
{\em The $e$-state of $y$} is defined by $\sigma(e,y)=\{i\leq e:y\in(V_i)_e\}$, and {\em the maximal $e$-state} is defined by $\sigma(e)=\max_z\sigma(e,z)$.
The construction of $f:\nn\to\nn$ is to maximize the $e$-state.
For each $e\in\nn$, $f(e)$ chooses the least $y\in\nn$ having the maximal $e$-state $\sigma(e,y)=\sigma(e)$.
Since $\{\sigma(e,y):e,y\in\mathbb{N}\}$ is uniformly c.e., and $\sigma(e,y)\in 2^e$, the function $e\mapsto\sigma(e)=\max_z\sigma(e,z)$ is total limit computable.
Thus, $f$ is limit computable.
It is easy to see that $\lim_e\sigma(e)(n)$ exists for each $n\in\mathbb{N}$.
%Now we claim that $\lim_e\sigma(e)(n)$ exists for each $n\in\mathbb{N}$.
%Inductively assume that for $n$ there exists $i$ such that $\sigma(i)\res n=\sigma(j)\res n$ for any $j\geq i$.
%Since ...
Let $U=\{U_n:n\in\mathbb{N}\}$ be a uniformly c.e.\ sequence of cofinite c.e.\ sets.
Then $V=\{\bigcap_{m\leq n}U_m:n\in\mathbb{N}\}$ is a uniformly c.e.\ non-increasing sequence of cofinite c.e.\ sets.
Thus, $V_i=V$ for some index $i$.
Then $i\in\sigma(e,y)$ for almost all $e,y\in\mathbb{N}$.
This ensures that $i\in\sigma(e)$ for almost all $e\in\mathbb{N}$ by our assumption $\min U_n\geq n$.
Hence we have $f(n)\in U_n$ for almost all $n\in\mathbb{N}$.
%%%%
\end{proof}

\begin{remark}
The proof of Lemma \ref{lem:comp_roid} is similar to the standard construction of a maximal c.e.\ set (see Soare \cite{Soa}).
Recall that the principal function of the complement of a maximal c.e.\ set is {\em dominant}, i.e., it dominates all total computable functions.
The limit computable function $f$ in Lemma \ref{lem:comp_roid} is also dominant.
Indeed, for any total computable function $g$, if we set $U^g_n=\{y\in\mathbb{N}:y\geq g(n)\}$ then $\{U^g_n:n\in\mathbb{N}\}$ is a uniformly c.e.\ sequence of cofinite c.e.\ sets, and if $f(n)\in U^g_n$ holds then we have $f(n)\geq g(n)$.
\end{remark}

\begin{proof}[Proof of Theorem \ref{thm:roid:alinc}]\upshape
Pick a limit computable function $f=\lim_sf_s$ in Lemma \ref{lem:comp_roid}.
For every $t,u\in\mathbb{N}$, put $v(t,u)=2^{-s}$ for the least $s$ such that $f_s(t)=u$ if such $s$ exists; $v(t,u)=0$ otherwise.
Since $\{f_s:s\in\mathbb{N}\}$ is uniformly computable, $v:\mathbb{N}^2\to\mathbb{R}$ is computable.

\begin{construction}\upshape
For each $t\in\mathbb{N}$, {\em the center position of the $u$-th rising of the $t$-th comb} is defined as $c_*(t,u)=2^{-(2t+1)}+2^{-(2t+u+1)}$, and {\em the width of the $u$-th rising of the $t$-th comb} is defined as $v_*(t,u)=v(t,u)\cdot 2^{-(2t+u+3)}$.
Then, we define {\em the $t$-th harmonic comb} $K_{t}$ for each $t\in\mathbb{N}$ as follows:
\begin{align*}
K^*_{t}&=\{2^{-(2t+1)}\}\times [0,2^{-t}]\\
K^0_{t,u}&=\{c_*(t,u)-v_*(t,u),c_*(t,u)+v_*(t,u)\}\times [0,2^{-t}]\\
K^1_{t,u}&=[c_*(t,u)-v_*(t,u),c_*(t,u)+v_*(t,u)]\times\{2^{-t}\}\\
K^2_{t,u}&=(c_*(t,u)-v_*(t,u),c_*(t,u)+v_*(t,u))\times(-1,2^{-t})\\
K_{t}&=\left(K^*_{t}\cup\bigcup_{i<2}\bigcup_{u\in\mathbb{N}}K^i_{t,u}\right)\cup\left(([2^{-(2t+1)},2^{-2t}]\times \{0\})\setminus\bigcup_{u\in\mathbb{N}}K^2_{t,u}\right).
\end{align*}
Note that $K_{t}$ is homeomorphic to the harmonic comb $\mathcal{H}$ for each $t\in\mathbb{N}$.
This is because, for fixed $t\in\mathbb{N}$, since $\lim_sf_s(t)$ exists we have $v(t,u)=0$ for almost all $u\in\mathbb{N}$.
Then the desired computable dendroid is defined as follows.
\[K=([-1,0]\times\{0\})\cup\bigcup_{t\in\mathbb{N}}\left(\big([2^{-(2t+2)},2^{-(2t+1)}]\times\{0\}\big)\cup K_{t}\right).\]
\end{construction}

\begin{figure}[t]\centering
 \begin{minipage}{0.48\hsize}
  \begin{center}
%WinTpicVersion3.08
\unitlength 0.1in
\begin{picture}( 18.0000,  8.2300)(  2.0000,-10.8300)
% LINE 2 0 3 0
% 2 800 1080 200 1080
% 
\special{pn 8}%
\special{pa 800 1080}%
\special{pa 200 1080}%
\special{fp}%
% DOT 1 0 3 0
% 1 800 1080
% 
\special{pn 13}%
\special{sh 1}%
\special{ar 800 1080 10 10 0  6.28318530717959E+0000}%
% STR 2 0 3 0
% 3 780 1160 780 1210 2 0
% $0$
\put(7.8000,-12.1000){\makebox(0,0)[lb]{$0$}}%
% STR 2 0 3 0
% 3 1020 1180 1020 1230 2 0
% $1/4$
\put(10.2000,-12.3000){\makebox(0,0)[lb]{$1/4$}}%
% STR 2 0 3 0
% 3 1980 1160 1980 1210 2 0
% $1$
\put(19.8000,-12.1000){\makebox(0,0)[lb]{$1$}}%
% STR 2 0 3 0
% 3 1320 1180 1320 1230 2 0
% $1/2$
\put(13.2000,-12.3000){\makebox(0,0)[lb]{$1/2$}}%
% LINE 2 0 3 0
% 14 2000 1080 1400 1080 2000 1080 2000 480 1400 1080 1400 480 1700 1080 1700 480 1550 1080 1550 480 1475 1080 1475 480 1438 1080 1438 480
% 
\special{pn 8}%
\special{pa 2000 1080}%
\special{pa 1400 1080}%
\special{fp}%
\special{pa 2000 1080}%
\special{pa 2000 480}%
\special{fp}%
\special{pa 1400 1080}%
\special{pa 1400 480}%
\special{fp}%
\special{pa 1700 1080}%
\special{pa 1700 480}%
\special{fp}%
\special{pa 1550 1080}%
\special{pa 1550 480}%
\special{fp}%
\special{pa 1476 1080}%
\special{pa 1476 480}%
\special{fp}%
\special{pa 1438 1080}%
\special{pa 1438 480}%
\special{fp}%
% LINE 2 0 3 0
% 2 1420 1080 1420 480
% 
\special{pn 8}%
\special{pa 1420 1080}%
\special{pa 1420 480}%
\special{fp}%
% BOX 2 5 1 0
% 2 1412 1080 1400 480
% 
\special{pn 8}%
\special{sh 0.300}%
\special{pa 1412 1080}%
\special{pa 1400 1080}%
\special{pa 1400 480}%
\special{pa 1412 480}%
\special{pa 1412 1080}%
\special{ip}%
% LINE 2 0 3 0
% 14 1100 1080 950 1080 1100 1080 1100 780 950 1080 950 780 1025 1080 1025 780 988 1080 988 780 969 1080 969 780 960 1080 960 780
% 
\special{pn 8}%
\special{pa 1100 1080}%
\special{pa 950 1080}%
\special{fp}%
\special{pa 1100 1080}%
\special{pa 1100 780}%
\special{fp}%
\special{pa 950 1080}%
\special{pa 950 780}%
\special{fp}%
\special{pa 1026 1080}%
\special{pa 1026 780}%
\special{fp}%
\special{pa 988 1080}%
\special{pa 988 780}%
\special{fp}%
\special{pa 970 1080}%
\special{pa 970 780}%
\special{fp}%
\special{pa 960 1080}%
\special{pa 960 780}%
\special{fp}%
% LINE 2 0 3 0
% 2 955 1080 955 780
% 
\special{pn 8}%
\special{pa 956 1080}%
\special{pa 956 780}%
\special{fp}%
% BOX 2 5 1 0
% 2 953 1080 950 780
% 
\special{pn 8}%
\special{sh 0.300}%
\special{pa 954 1080}%
\special{pa 950 1080}%
\special{pa 950 780}%
\special{pa 954 780}%
\special{pa 954 1080}%
\special{ip}%
% BOX 2 0 1 0
% 2 875 1080 838 930
% 
\special{pn 8}%
\special{sh 0.300}%
\special{pa 876 1080}%
\special{pa 838 1080}%
\special{pa 838 930}%
\special{pa 876 930}%
\special{pa 876 1080}%
\special{fp}%
% POLYGON 2 5 1 0
% 5 800 1080 818 1005 818 1080 818 1080 800 1080
% 
\special{pn 8}%
\special{sh 0.300}%
\special{pa 800 1080}%
\special{pa 818 1006}%
\special{pa 818 1080}%
\special{pa 818 1080}%
\special{pa 800 1080}%
\special{ip}%
% LINE 2 0 3 0
% 2 800 1080 1400 1080
% 
\special{pn 8}%
\special{pa 800 1080}%
\special{pa 1400 1080}%
\special{fp}%
% STR 2 0 3 0
% 3 1500 405 1500 430 2 0
% $K_0$
\put(15.0000,-4.3000){\makebox(0,0)[lb]{$K_0$}}%
% STR 2 0 3 0
% 3 980 705 980 730 2 0
% $K_1$
\put(9.8000,-7.3000){\makebox(0,0)[lb]{$K_1$}}%
% STR 2 0 3 0
% 3 780 780 780 880 2 0
% $K_2$
\put(7.8000,-8.8000){\makebox(0,0)[lb]{$K_2$}}%
\end{picture}%
  \end{center}
 \vspace{-0.5em}
\caption{The dendroid $K$.}
  \label{fig:Dendroid1}
 \end{minipage}
 \begin{minipage}{0.48\hsize}
  \begin{center}
%WinTpicVersion3.08
\unitlength 0.1in
\begin{picture}( 17.4000, 14.2000)(  9.6000,-16.0000)
% LINE 2 0 3 0
% 2 1240 1600 1240 400
% 
\special{pn 8}%
\special{pa 1240 1600}%
\special{pa 1240 400}%
\special{fp}%
% BOX 2 5 1 0
% 2 1224 1600 1200 400
% 
\special{pn 8}%
\special{sh 0.300}%
\special{pa 1224 1600}%
\special{pa 1200 1600}%
\special{pa 1200 400}%
\special{pa 1224 400}%
\special{pa 1224 1600}%
\special{ip}%
% LINE 2 0 3 0
% 20 1460 1600 1460 400 1460 400 1540 400 1540 400 1540 1600 1540 1600 2100 1600 2100 1600 2100 400 2100 400 2700 400 2700 400 2700 1600 1800 1600 1800 400 1460 1600 1200 1600 1200 1600 1200 400
% 
\special{pn 8}%
\special{pa 1460 1600}%
\special{pa 1460 400}%
\special{fp}%
\special{pa 1460 400}%
\special{pa 1540 400}%
\special{fp}%
\special{pa 1540 400}%
\special{pa 1540 1600}%
\special{fp}%
\special{pa 1540 1600}%
\special{pa 2100 1600}%
\special{fp}%
\special{pa 2100 1600}%
\special{pa 2100 400}%
\special{fp}%
\special{pa 2100 400}%
\special{pa 2700 400}%
\special{fp}%
\special{pa 2700 400}%
\special{pa 2700 1600}%
\special{fp}%
\special{pa 1800 1600}%
\special{pa 1800 400}%
\special{fp}%
\special{pa 1460 1600}%
\special{pa 1200 1600}%
\special{fp}%
\special{pa 1200 1600}%
\special{pa 1200 400}%
\special{fp}%
% LINE 2 0 3 0
% 4 1350 400 1350 1600 1275 1600 1275 400
% 
\special{pn 8}%
\special{pa 1350 400}%
\special{pa 1350 1600}%
\special{fp}%
\special{pa 1276 1600}%
\special{pa 1276 400}%
\special{fp}%
% STR 2 0 3 0
% 3 2340 250 2340 350 2 0
% $K_{t,0}$
\put(23.4000,-3.5000){\makebox(0,0)[lb]{$K_{t,0}$}}%
% STR 2 0 3 0
% 3 1750 250 1750 350 2 0
% $K_{t,1}$
\put(17.5000,-3.5000){\makebox(0,0)[lb]{$K_{t,1}$}}%
% STR 2 0 3 0
% 3 1430 250 1430 350 2 0
% $K_{t,2}$
\put(14.3000,-3.5000){\makebox(0,0)[lb]{$K_{t,2}$}}%
% LINE 2 2 3 0
% 2 2400 400 2400 1600
% 
\special{pn 8}%
\special{pa 2400 400}%
\special{pa 2400 1600}%
\special{dt 0.045}%
% STR 2 0 3 0
% 3 2300 1650 2300 1750 2 0
% $2^{-2t}$
\put(23.0000,-17.5000){\makebox(0,0)[lb]{$2^{-2t}$}}%
% STR 2 0 3 0
% 3 960 1650 960 1750 2 0
% $2^{-(2t+1)}$
\put(9.6000,-17.5000){\makebox(0,0)[lb]{$2^{-(2t+1)}$}}%
\end{picture}%
  \end{center}
 \vspace{-0.5em}
\caption{The harmonic comb $K_t$ for $f_0(t)=0$, $f_1(t)=0$, $f_2(t)=2$, $\dots$}
  \label{fig:Dendroid2}
 \end{minipage}
\end{figure}

\begin{claim}
The set $K$ is a computable dendroid.
\end{claim}

Clearly $K$ is a computable closed set.
To show that $K$ is pathwise connected, we consider the following subcontinuum $K_{t}^-$, {\em the grip of the comb $K_{t,m}$}, for each $t\in\mathbb{N}$.
\begin{align*}
K^-_{t}&=(\bigcup_{i<2}\bigcup_{v(t,u)>0}K^i_{t,u})\cup\big(([2^{-(2t+1)},2^{-2t}]\times \{0\})\setminus\bigcup_{v(t,u)>0}K^2_{t,u}\big).
\end{align*}
Then $K^-=([-1,0]\times\{0\})\cup\bigcup_{t\in\mathbb{N}}\left(([2^{-(2t+2)},2^{-(2t+1)}]\times\{0\})\cup K^-_{t}\right)$ has no ramification points.
We claim that $K^-$ is connected and $K^-$ is even an arc.
To show this claim, we first observe that $K^-_{t}$ is an arc for any $t\in\mathbb{N}$, since $v(t,u)>0$ occurs for finitely many $u\in\mathbb{N}$.
Moreover $K^-_{t}\subseteq S(t)$, and $\lim_t{\rm diam}(S(t))=0$ holds.
Therefore, we see that $K^-$ is locally connected and, hence, an arc.
For points $p,q\in K$, if $p,q\in K^-$ then $p$ and $q$ are connected by a subarc of $K^-$.
In the case $p\in K\setminus K^-$, the point $p$ lies on $K^0_{t,u}$ for some $t,u$ such that $v(t,u)=0$.
If $q\in K^-$ then there is a subarc $A\subseteq K^-$ (one of whose endpoints must be $\lrangle{c_*(t,u),0}$) such that $A\cup K^0_{t,u}$ is an arc containing $p$ and $q$.
In the case $q\in K\setminus K^-$, similarly we can connect $p$ and $q$ by an arc in $K$.
Hence, $K$ is pathwise connected.
$K$ is hereditarily unicoherent, since the harmonic comb is hereditarily unicoherent.
Thus, $K$ is a dendroid.

\begin{claim}
The computable dendroid $K$ does not $\ast$-include a $\Pi^0_1$ dendrite.
\end{claim}

What remains to show is that every $\Pi^0_1$ subdendrite $R\subseteq K$ satisfies $d_H(R,K)$
$\geq 1$.
Let $R\subseteq K$ be a $\Pi^0_1$ dendrite.
Put $S(t)=[2^{-(2t+1)},2^{-2t}]\times[0,2^{-t}]$.
Since $R$ is locally connected, $R\cap S(t)=R\cap K_t$ is also locally connected for each $t\in\mathbb{N}$ and $m<2^t$.
Thus, for fixed $t\in\mathbb{N}$, let $K^{1*}_{t,u}=[c_*(t,u)-2^{-(2t+u+3)},c_*(t,u)+2^{-(2t+u+3)}]\times\{2^{-t}\}$.
For any continuum $R^*\subset K_{t}$, if $R^*\cap K^{1*}_{t,u}\not=\emptyset$ for infinitely many $u\in\nn$, then $R^*$ must be homeomorphic to the harmonic comb $\mathcal{H}$.
Hence, $R^*$ is not locally connected.
Therefore, we have $R\cap K^{1*}_{t,u}=\emptyset$ for almost all $u\in\mathbb{N}$.
Since $K^{1*}_{t,u}$ and $K^{1*}_{s,v}$ is disjoint whenever $\lrangle{t,u}\not=\lrangle{s,v}$, and since $R$ is $\Pi^0_1$, we can effectively enumerate $U_t=\{u\in\mathbb{N}:R\cap K^{1*}_{t,u}=\emptyset\}$, i.e., $\{U_t:t\in\mathbb{N}\}$ is uniformly c.e.
Moreover, $U_t$ is cofinite for every $t\in\mathbb{N}$.
Then, by our definition of $f=\lim_sf_s$ in Lemma \ref{lem:comp_roid}, there exists $t^*\in\mathbb{N}$ such that $f(t)\in U_t$ for all $t\geq t^*$.
Note that $v(t,f(t))>0$ and thus the condition $f(t)\in U_t$ (i.e., $R\cap K^{1*}_{t,f(t)}=\emptyset$) implies that, for every $t\geq t^*$, either $R\subseteq [-1,c_*(t,u)+v_*(t,u)]\times[0,1]$ or $R\subseteq [c_*(t,u)-v_*(t,u),1]\times[0,1]$ holds.
Thus we obtain the desired condition $d_H(R,K)\geq 1$.
%%%%
\end{proof}

\begin{remark}
It is easy to see that the dendroid constructed in the proof of Theorem \ref{thm:roid:alinc} is contractible. 
\end{remark}

\begin{cor}
There exists a nonempty contractible planar computable closed subset of $[0,1]^2$ which $\ast$-includes no $\Pi^0_1$ subset which is connected and locally connected.
\end{cor}

\begin{theorem}\label{thm:special_roid}
Not every nonempty $\Pi^0_1$ planar dendroid contains a computable point.
\end{theorem}

\begin{proof}\upshape
One can easily construct a $\Pi^0_1$ Cantor fan $F$ containing at most one computable point $p\in F$, and such $p$ is the unique ramification point of $F$.
Our basic idea is to construct a topological copy of the Cantor fan $F$ along a pathological located arc $A$ constructed by Miller \cite[Example 4.1]{Mil}.
We can guarantee that moving the fan $F$ along the arc $A$ cannot introduce new computable points.
Additionally, this moving will make a ramification point $p^*$ in a copy of $F$ incomputable.

\medskip

\noindent
{\bf Fat Approximation.}
To archive this construction, we consider a fat approximation of a subset $P$ of the middle third Cantor set $C\subseteq\mathbb{R}^1$, by modifying the standard construction of $C$.
For a tree $T\subseteq 2^{<\mathbb{N}}$, put $\pi(\sigma)=3^{-1}+2\sum_{i<lh(\sigma)\;\&\; \sigma(i)=1}3^{-(i+2)}$ for $\sigma\in T$, and $J(\sigma)=[\pi(\sigma)-3^{-(lh(\sigma)+1)},\pi(\sigma)+2\cdot 3^{-(lh(\sigma)+1)}]$.
If a binary string $\sigma$ is incomparable with a binary string $\tau$ then $J(\sigma)\cap J(\tau)=\emptyset$.
We extend $\pi$ to a homeomorphism $\pi_*$ between Cantor space $2^\mathbb{N}$ and $C\cap[1/3,2/3]$ by defining $\pi_*(f)=3^{-1}+2\sum_{f(i)=1}3^{-(i+2)}$ for $f\in 2^\mathbb{N}$.
Let $P^*\subseteq 2^\mathbb{N}$ be a nonempty $\Pi^0_1$ set without computable elements.
Then there exists a computable tree $T_P$ such that $P^*$ is the set of all paths of $T_P$, since $P^*$ is $\Pi^0_1$.
{\em A fat approximation $\{P_s:s\in\mathbb{N}\}$ of $P=\pi_*(P^*)$} is defined as $P_s=\bigcup\{J(\sigma):lh(\sigma)=s\;\&\;\sigma\in T_P\}$.
Then $\{P_s:s\in\mathbb{N}\}$ is a computable decreasing sequence of computable closed sets, and we have $P=\bigcap_sP_s$.
Since $P$ is a nonempty bounded closed subset of a real line $\mathbb{R}^1$, both $\min P$ and $\max P$ exist.
By the same reason, both $l_s^-=\min P_s$ and $r_s^+=\max P_s$ also exist, for each $s\in\mathbb{N}$, and $\lim_sl_s^-=\min P$ and $\lim_sr_s^+=\max P$, where $\{l_s:s\in\mathbb{N}\}$ is increasing, and $\{r_s:s\in\mathbb{N}\}$ is decreasing.
Let $l_s=l_s^-+3^{-(s+1)}$ and $r_s=r_s^+-3^{-(s+1)}$.
We also set $l_s^*=l_s^-+3^{-(s+2)}$ and $r_s=r_s^+-3^{-(s+2)}$.
Note that $l_s<r_s$, $\lim_sl_s=\min P$, and $\lim_sr_s=\max P$.
Since $\min P,\max P\in P$ and $P$ contains no computable points, $\min P$ and $\max P$ are non-computable, and so $l_s<\min P<\max P<r_s$ holds for any $s\in\mathbb{N}$.
The fat approximation of $P$ has the following remarkable property:
\[[l_s^-,l_s]\subseteq P_s,\ [l_s^-,l_s]\cap P=\emptyset,\ [r_s,r_s^+]\subseteq P_s,\ \mbox{and }[r_s,r_s^+]\cap P=\emptyset.\]
To simplify the construction, we may also assume that $P$ has the following property:
\[P=\{1-x\in\mathbb{R}:x\in P\}\]
Because, for any $\Pi^0_1$ subset $A\subseteq C$, the $\Pi^0_1$ set $A^*=\{x/3:x\in A\}\cup\{1-x/3:x\in A\}\subseteq C$ has that property.

\medskip

\noindent
{\bf Basic Notation.}
For each $i,j<2$, for each $a,b\in\mathbb{R}^2$, and for each $q,r\in\mathbb{R}$, {\em the $2$-cube} $\Delta_{ij}(a,b;q,r)\subseteq [a,a+q]\times [b,b+r]$ is defined as the smallest convex set containing the three points $\{(a,b),(a+q,b),(a,b+r),(a+q,b+r)\}\setminus\{(a+(1-i)q,b+(1-j)r)\}$.
Namely,
\begin{align*}
\Delta_{ij}(a,b;q,r)=\{&\lrangle{(-1)^ix+a+iq,(-1)^jy+b+jr}\in\mathbb{R}^2\\
&:x,y\geq 0\;\&\;rx+qy\leq qr\}.
\end{align*}
For a set $R\subseteq\mathbb{R}^1$ and real numbers $r,y\in\mathbb{R}$, put $\Theta(R;r,y)=\{rx+y\in\mathbb{R}:x\in R\}$.
Clearly $\Theta(R;r,y)$ is computably homeomorphic to $R$.
Let four symbols $\llcorner$, $\urcorner$, $\lrcorner$, and $\ulcorner$ denote $\lrangle{10,01}$, $\lrangle{01,10}$, $\lrangle{00,11}$, and $\lrangle{11,00}$, respectively.
For $v\in\{\llcorner,\urcorner,\lrcorner,\ulcorner\}$ and for any $R\subseteq[0,1]$, $a,b\in\mathbb{R}^2$, and $q,r\in\mathbb{R}$, we define $[v](R;a,b;q,r)\subseteq[a,a+q]\times[b,b+r]$ as follows:
\begin{align*}
[v](R;a,b;q,r)=\big(([a,a+q]\times\Theta(R;r,b))\cap\Delta_{v(0)}(a,b;q,r)\big)\\
\cup\big((\Theta(R;q,a)\times [b,b+r])\cap\Delta_{v(1)}(a,b;q,r)\big).
\end{align*}

\begin{figure}[t]\centering
 \begin{minipage}{0.48\hsize}
  \begin{center}
%WinTpicVersion3.08
\unitlength 0.1in
\begin{picture}( 17.0000, 11.7000)(  1.0000,-12.0000)
% POLYGON 2 0 1 0
% 5 800 400 800 1000 1400 400 1400 400 800 400
% 
\special{pn 8}%
\special{sh 0.300}%
\special{pa 800 400}%
\special{pa 800 1000}%
\special{pa 1400 400}%
\special{pa 1400 400}%
\special{pa 800 400}%
\special{fp}%
% POLYGON 2 0 1 0
% 6 800 1000 1400 400 1400 400 1400 1000 1400 1000 800 1000
% 
\special{pn 8}%
\special{sh 0.300}%
\special{pa 800 1000}%
\special{pa 1400 400}%
\special{pa 1400 400}%
\special{pa 1400 1000}%
\special{pa 1400 1000}%
\special{pa 800 1000}%
\special{fp}%
% VECTOR 2 0 3 0
% 4 400 1200 400 200 400 1200 1800 1200
% 
\special{pn 8}%
\special{pa 400 1200}%
\special{pa 400 200}%
\special{fp}%
\special{sh 1}%
\special{pa 400 200}%
\special{pa 380 268}%
\special{pa 400 254}%
\special{pa 420 268}%
\special{pa 400 200}%
\special{fp}%
\special{pa 400 1200}%
\special{pa 1800 1200}%
\special{fp}%
\special{sh 1}%
\special{pa 1800 1200}%
\special{pa 1734 1180}%
\special{pa 1748 1200}%
\special{pa 1734 1220}%
\special{pa 1800 1200}%
\special{fp}%
% LINE 2 2 3 0
% 8 400 1000 800 1000 800 1000 800 1200 1400 1200 1400 1000 400 400 800 400
% 
\special{pn 8}%
\special{pa 400 1000}%
\special{pa 800 1000}%
\special{dt 0.045}%
\special{pa 800 1000}%
\special{pa 800 1200}%
\special{dt 0.045}%
\special{pa 1400 1200}%
\special{pa 1400 1000}%
\special{dt 0.045}%
\special{pa 400 400}%
\special{pa 800 400}%
\special{dt 0.045}%
% LINE 2 2 3 0
% 2 1000 600 800 200
% 
\special{pn 8}%
\special{pa 1000 600}%
\special{pa 800 200}%
\special{dt 0.045}%
% STR 2 0 3 0
% 3 600 100 600 200 2 0
% $\Delta_{01}(a,b;q,r)$
\put(6.0000,-2.0000){\makebox(0,0)[lb]{$\Delta_{01}(a,b;q,r)$}}%
% LINE 2 2 3 0
% 2 1200 800 1600 600
% 
\special{pn 8}%
\special{pa 1200 800}%
\special{pa 1600 600}%
\special{dt 0.045}%
% STR 2 0 3 0
% 3 1450 500 1450 600 2 0
% $\Delta_{10}(a,b;q,r)$
\put(14.5000,-6.0000){\makebox(0,0)[lb]{$\Delta_{10}(a,b;q,r)$}}%
% STR 2 0 3 0
% 3 760 1230 760 1330 2 0
% $a$
\put(7.6000,-13.3000){\makebox(0,0)[lb]{$a$}}%
% STR 2 0 3 0
% 3 1260 1250 1260 1350 2 0
% $a+q$
\put(12.6000,-13.5000){\makebox(0,0)[lb]{$a+q$}}%
% STR 2 0 3 0
% 3 310 960 310 1060 2 0
% $b$
\put(3.1000,-10.6000){\makebox(0,0)[lb]{$b$}}%
% STR 2 0 3 0
% 3 100 370 100 470 2 0
% $b+r$
\put(1.0000,-4.7000){\makebox(0,0)[lb]{$b+r$}}%
\end{picture}%
  \end{center}
 \vspace{-0.5em}
\caption{The cubes $\Delta_{ij}(a,b,q,r)$.}
  \label{fig:Cube1}
 \end{minipage}
 \begin{minipage}{0.48\hsize}
  \begin{center}
%WinTpicVersion3.08
\unitlength 0.1in
\begin{picture}( 17.0000, 10.0000)(  2.0000,-12.0000)
% VECTOR 2 0 3 0
% 4 500 1200 500 200 500 1200 1900 1200
% 
\special{pn 8}%
\special{pa 500 1200}%
\special{pa 500 200}%
\special{fp}%
\special{sh 1}%
\special{pa 500 200}%
\special{pa 480 268}%
\special{pa 500 254}%
\special{pa 520 268}%
\special{pa 500 200}%
\special{fp}%
\special{pa 500 1200}%
\special{pa 1900 1200}%
\special{fp}%
\special{sh 1}%
\special{pa 1900 1200}%
\special{pa 1834 1180}%
\special{pa 1848 1200}%
\special{pa 1834 1220}%
\special{pa 1900 1200}%
\special{fp}%
% STR 2 0 3 0
% 3 860 1230 860 1330 2 0
% $a$
\put(8.6000,-13.3000){\makebox(0,0)[lb]{$a$}}%
% STR 2 0 3 0
% 3 1360 1250 1360 1350 2 0
% $a+q$
\put(13.6000,-13.5000){\makebox(0,0)[lb]{$a+q$}}%
% STR 2 0 3 0
% 3 410 960 410 1060 2 0
% $b$
\put(4.1000,-10.6000){\makebox(0,0)[lb]{$b$}}%
% STR 2 0 3 0
% 3 200 370 200 470 2 0
% $b+r$
\put(2.0000,-4.7000){\makebox(0,0)[lb]{$b+r$}}%
% POLYGON 2 5 0 0
% 8 1500 900 1000 900 1000 400 1200 400 1200 700 1500 700 1500 700 1500 900
% 
\special{pn 8}%
\special{sh 0.600}%
\special{pa 1500 900}%
\special{pa 1000 900}%
\special{pa 1000 400}%
\special{pa 1200 400}%
\special{pa 1200 700}%
\special{pa 1500 700}%
\special{pa 1500 700}%
\special{pa 1500 900}%
\special{ip}%
% LINE 2 2 3 0
% 2 900 1000 1500 400
% 
\special{pn 8}%
\special{pa 900 1000}%
\special{pa 1500 400}%
\special{dt 0.045}%
% LINE 2 2 3 0
% 8 500 1000 900 1000 900 1200 900 1000 1500 1200 1500 400 1500 400 500 400
% 
\special{pn 8}%
\special{pa 500 1000}%
\special{pa 900 1000}%
\special{dt 0.045}%
\special{pa 900 1200}%
\special{pa 900 1000}%
\special{dt 0.045}%
\special{pa 1500 1200}%
\special{pa 1500 400}%
\special{dt 0.045}%
\special{pa 1500 400}%
\special{pa 500 400}%
\special{dt 0.045}%
% VECTOR 2 2 3 0
% 2 1700 800 1500 800
% 
\special{pn 8}%
\special{pa 1700 800}%
\special{pa 1500 800}%
\special{dt 0.045}%
\special{sh 1}%
\special{pa 1500 800}%
\special{pa 1568 820}%
\special{pa 1554 800}%
\special{pa 1568 780}%
\special{pa 1500 800}%
\special{fp}%
% STR 2 0 3 0
% 3 1720 805 1720 855 2 0
% A copy of $R$
\put(17.2000,-8.5500){\makebox(0,0)[lb]{A copy of $R$}}%
\end{picture}%
  \end{center}
 \vspace{-0.5em}
\caption{$[\llcorner](R;a,b;q,r)$ for $R=[1/6,1/2]$}
  \label{fig:Cube2}
 \end{minipage}
\end{figure}

\begin{sublem}
$[v](P;a,b;q,r)$ is computably homeomorphic to $P\times [0,1]$.
In particular, $[v](P;a,b;q,r)$ contains no computable points.
%%%%
\end{sublem}

To simplify our argument, we use the normalization $\tilde{P}_t^s$ of $P_t$ for $t\geq s$, that is defined by $\tilde{P}_t^s=\{(x-l_s^-)/(r_s^+-l_s^-)\in\mathbb{R}:x\in P_t\}$, for each $s\in\mathbb{N}$.
Note that $\tilde{P}^s_t\subseteq [0,1]$ for $t\geq s$, and $0,1\in\tilde{P}^s_s$ holds for each $s\in\mathbb{N}$.
Put $[v]^s_t([a,a+q]\times[b,b+r])=[v](\tilde{P}_t^s;a,b;q,r)$ for $t\geq s$.
We also introduce the following two notions:
\begin{align*}
[-]^s_t([a,a+q]\times[b,b+r])&=[a,a+q]\times\Theta(\tilde{P}_t^s;r,b);\\
[\;\mid\;]^s_t([a,a+q]\times[b,b+r])&=\Theta(\tilde{P}_t^s;q,a)\times[b,b+r].
\end{align*}
Here we code two symbols $-$ and $\mid$ as $0$ and $1$ respectively.

\begin{sublem}
$[v]^s_t([a,a+q]\times[b,b+r])\subseteq[a,a+q]\times[b,b+r]$, and $[v]^s_t([a,a+q]\times[b,b+r])$ intersects with the boundary of $[a,a+q]\times[b,b+r]$.
\end{sublem}

\begin{sublem}
There is a computable homeomorphism between $[v]^s_t(a,b;q,r)$ and $P_t\times [0,1]$ for any $t\in\nn$.
Therefore, $\bigcap_t[v]^s_t(a,b;q,r)$ is computably homeomorphic to $P\times [0,1]$.
%%%%
\end{sublem}

\noindent
{\bf Blocks.}
A {\em block} is a set $Z\subseteq\mathbb{R}^2$ with {\em a bounding box ${\rm Box}(Z)=[a,a+q]\times[b,b+r]$}.
Each $\delta\in 2^2$ is called a {\em direction}, and directions $\lrangle{00}$, $\lrangle{01}$, $\lrangle{10}$, and $\lrangle{11}$ are also denoted by $[\leftarrow]$, $[\rightarrow]$, $[\downarrow]$, and $[\uparrow]$, respectively.
For $\delta\in 2^2$, $\delta^\circ=\lrangle{\delta(0),1-\delta(0)}$ is called {\em the reverse direction of $\delta$}.
Put ${\rm Line}(Z;[\leftarrow])=\{a\}\times[b,b+r]$; ${\rm Line}(Z;[\rightarrow])=\{a+q\}\times[b,b+r]$; ${\rm Line}(Z;[\downarrow])=[a,a+q]\times\{b\}$; ${\rm Line}(Z;[\uparrow])=[a,a+q]\times\{b+r\}$.
Assume that a class $\mathcal{Z}$ of blocks is given.
We introduce a relation $\touch{\delta}$ on $\mathcal{Z}$ for each direction $\delta$.
Fix a block $Z_{\rm first}\in\mathcal{Z}$, and we call it {\em the first block}.
Then we declare that $\touch{[\leftarrow]}Z_{\rm first}$ holds.
We inductively define the relation $\touch{\delta}$ on $\mathcal{Z}$.
If $Z\touch{\delta}Z_0$ (resp.\ $Z_0\touch{\delta}Z$) for some $Z$ and $\delta$, then we also write it as $\touch{\delta}Z_0$ (resp.\ $Z_0\touch{\delta}$).
For any two blocks $Z_0$ and $Z_1$, the relation $Z_0\touch{\delta}Z_1$ holds if the following three conditions are satisfied:
\begin{enumerate}
\item $Z_0\cap Z_1={\rm Line}(Z_0;\delta)\cap Z_0={\rm Line}(Z_1;\delta^\circ)\cap Z_1\not=\emptyset$.
\item $\touch{\varepsilon}Z_0$ has been already satisfied for some direction $\varepsilon$.
\item $Z_1\touch{\varepsilon}Z_0$ does not satisfied for any direction $\varepsilon$
\end{enumerate}
If $Z_0\touch{\delta}Z_1$ for some $\delta$, then we say that {\em $Z_1$ is a successor of $Z_0$} ($Z_0$ is a predecessor of $Z_1$), and we also write it as $Z_0\touch{}Z_1$.

We will construct a partial computable function $\mathcal{Z}:\nn^3\to\mathcal{A}(\mathbb{R}^2)$ with a computable function $k:\nn\to\nn$ and ${\rm dom}(\mathcal{Z})=\{(u,i,t)\in\nn^3:u\leq t\;\&\;i<k(u)\}$ such that $\mathcal{Z}(u,i,t)$ is a block with a bounding box for any $(u,i,t)\in{\rm dom}(\mathcal{Z})$, and the block $\mathcal{Z}(u,i,t)$ is computably homeomorphic to $P_t\times[0,1]$ uniformly in $(u,i,t)$.
Here $\mathcal{A}(\mathbb{R}^2)$ is the hyperspace of all closed subsets in $\mathbb{R}^2$ with positive and negative information.
For each stage $t$, $\mathcal{Z}_t(u)=\{\mathcal{Z}(t,u,i):i<k(u)\}$ for each $u\leq t$ is defined.
Let $\mathcal{Z}(u)$ denote the finite set $\{\lambda t.\mathcal{Z}(t,u,i):i<k(u)\}$ of functions, for each $u\in\nn$.
The relation $\touch{}$ induces a pre-ordering $\prec$ on $\bigcup_{u\in\nn}\mathcal{Z}(u)$ as follows:
$Z_0\prec Z_1$ if there is a finite path from $Z_0(t)$ to $Z_1(t)$ on the finite directed graph $(\bigcup_{u\leq t}\mathcal{Z}_t(u),\touch{})$ at some stage $t\in\nn$.
We will assure that $\prec$ is a well-ordering of order type $\omega$, and $Z_0\prec Z_1$ whenever $Z_0\in\mathcal{Z}(u)$, $Z_1\in\mathcal{Z}(v)$, and $u<v$.
In particular, for every $Z\in\bigcup_{u\in\nn}\mathcal{Z}(u)$, the predecessor $Z_{\rm pre}$ of $Z$ and the successor $Z_{\rm suc}$ of $Z$ under $\prec$ are uniquely determined.
If $Z_{\rm pre}(t)\touch{\delta}Z(t)\touch{\varepsilon}Z_{\rm suc}(t)$, then we say that {\em $Z$ moves from $\delta$ to $\varepsilon$}, and that $\lrangle{\delta,\varepsilon}$ is {\em the direction of $Z$}.
%The last block has no direction in this sense, but we also say that $\lrangle{[\leftarrow],[\leftarrow]}$ is the direction of the last block.
\begin{example}\label{example}
Fig. \ref{fig:mthmblock} is an example satisfying $\touch{[\leftarrow]}Z_{\rm first}\touch{[\leftarrow]}Z_0\touch{[\downarrow]}Z_1\touch{[\rightarrow]}Z_2$.
\end{example}

\begin{figure}[t]\centering
  \begin{center}
%WinTpicVersion3.08
\unitlength 0.1in
\begin{picture}( 11.0000, 10.0000)(  0.0000,-16.0000)
% BOX 2 5 0 0
% 2 1100 1100 500 1000
% 
\special{pn 8}%
\special{sh 0.600}%
\special{pa 1100 1100}%
\special{pa 500 1100}%
\special{pa 500 1000}%
\special{pa 1100 1000}%
\special{pa 1100 1100}%
\special{ip}%
% BOX 2 5 1 0
% 2 500 700 1100 1000
% 
\special{pn 8}%
\special{sh 0.300}%
\special{pa 500 700}%
\special{pa 1100 700}%
\special{pa 1100 1000}%
\special{pa 500 1000}%
\special{pa 500 700}%
\special{ip}%
% BOX 2 5 0 0
% 2 1100 700 500 600
% 
\special{pn 8}%
\special{sh 0.600}%
\special{pa 1100 700}%
\special{pa 500 700}%
\special{pa 500 600}%
\special{pa 1100 600}%
\special{pa 1100 700}%
\special{ip}%
% BOX 2 5 0 0
% 2 500 600 0 700
% 
\special{pn 8}%
\special{sh 0.600}%
\special{pa 500 600}%
\special{pa 0 600}%
\special{pa 0 700}%
\special{pa 500 700}%
\special{pa 500 600}%
\special{ip}%
% BOX 2 5 0 0
% 2 0 700 100 1100
% 
\special{pn 8}%
\special{sh 0.600}%
\special{pa 0 700}%
\special{pa 100 700}%
\special{pa 100 1100}%
\special{pa 0 1100}%
\special{pa 0 700}%
\special{ip}%
% BOX 2 5 0 0
% 2 500 1000 300 1100
% 
\special{pn 8}%
\special{sh 0.600}%
\special{pa 500 1000}%
\special{pa 300 1000}%
\special{pa 300 1100}%
\special{pa 500 1100}%
\special{pa 500 1000}%
\special{ip}%
% BOX 2 5 1 0
% 2 500 700 100 1000
% 
\special{pn 8}%
\special{sh 0.300}%
\special{pa 500 700}%
\special{pa 100 700}%
\special{pa 100 1000}%
\special{pa 500 1000}%
\special{pa 500 700}%
\special{ip}%
% BOX 2 5 1 0
% 2 300 1000 100 1100
% 
\special{pn 8}%
\special{sh 0.300}%
\special{pa 300 1000}%
\special{pa 100 1000}%
\special{pa 100 1100}%
\special{pa 300 1100}%
\special{pa 300 1000}%
\special{ip}%
% BOX 2 5 1 0
% 2 100 1100 300 1500
% 
\special{pn 8}%
\special{sh 0.300}%
\special{pa 100 1100}%
\special{pa 300 1100}%
\special{pa 300 1500}%
\special{pa 100 1500}%
\special{pa 100 1100}%
\special{ip}%
% BOX 2 5 0 0
% 2 0 1100 100 1600
% 
\special{pn 8}%
\special{sh 0.600}%
\special{pa 0 1100}%
\special{pa 100 1100}%
\special{pa 100 1600}%
\special{pa 0 1600}%
\special{pa 0 1100}%
\special{ip}%
% BOX 2 5 0 0
% 2 300 1100 400 1400
% 
\special{pn 8}%
\special{sh 0.600}%
\special{pa 300 1100}%
\special{pa 400 1100}%
\special{pa 400 1400}%
\special{pa 300 1400}%
\special{pa 300 1100}%
\special{ip}%
% BOX 2 5 0 0
% 2 100 1500 800 1600
% 
\special{pn 8}%
\special{sh 0.600}%
\special{pa 100 1500}%
\special{pa 800 1500}%
\special{pa 800 1600}%
\special{pa 100 1600}%
\special{pa 100 1500}%
\special{ip}%
% BOX 2 5 0 0
% 2 400 1300 800 1400
% 
\special{pn 8}%
\special{sh 0.600}%
\special{pa 400 1300}%
\special{pa 800 1300}%
\special{pa 800 1400}%
\special{pa 400 1400}%
\special{pa 400 1300}%
\special{ip}%
% BOX 2 5 1 0
% 2 300 1400 800 1500
% 
\special{pn 8}%
\special{sh 0.300}%
\special{pa 300 1400}%
\special{pa 800 1400}%
\special{pa 800 1500}%
\special{pa 300 1500}%
\special{pa 300 1400}%
\special{ip}%
% BOX 1 0 3 0
% 2 1100 1100 500 600
% 
\special{pn 13}%
\special{pa 1100 1100}%
\special{pa 500 1100}%
\special{pa 500 600}%
\special{pa 1100 600}%
\special{pa 1100 1100}%
\special{fp}%
% BOX 1 0 3 0
% 2 500 600 0 1100
% 
\special{pn 13}%
\special{pa 500 600}%
\special{pa 0 600}%
\special{pa 0 1100}%
\special{pa 500 1100}%
\special{pa 500 600}%
\special{fp}%
% BOX 1 0 3 0
% 2 400 1100 0 1600
% 
\special{pn 13}%
\special{pa 400 1100}%
\special{pa 0 1100}%
\special{pa 0 1600}%
\special{pa 400 1600}%
\special{pa 400 1100}%
\special{fp}%
% BOX 1 0 3 0
% 2 400 1300 800 1600
% 
\special{pn 13}%
\special{pa 400 1300}%
\special{pa 800 1300}%
\special{pa 800 1600}%
\special{pa 400 1600}%
\special{pa 400 1300}%
\special{fp}%
% STR 2 0 3 0
% 3 700 850 700 900 2 0
% $Z_{\rm first}$
\put(7.0000,-9.0000){\makebox(0,0)[lb]{$Z_{\rm first}$}}%
% STR 2 0 3 0
% 3 200 850 200 900 2 0
% $Z_0$
\put(2.0000,-9.0000){\makebox(0,0)[lb]{$Z_0$}}%
% STR 2 0 3 0
% 3 135 1330 135 1380 2 0
% $Z_1$
\put(1.3500,-13.8000){\makebox(0,0)[lb]{$Z_1$}}%
% STR 2 0 3 0
% 3 500 1450 500 1500 2 0
% $Z_2$
\put(5.0000,-15.0000){\makebox(0,0)[lb]{$Z_2$}}%
\end{picture}%
  \end{center}
 \vspace{-0.5em}
\caption{Example \ref{example}.}
  \label{fig:mthmblock}
\end{figure}

\medskip

\noindent
{\bf Destination Point.}
Basically, our construction is similar as the construction by Miller \cite{Mil}.
Pick the standard homeomorphism $\rho$ between $2^\mathbb{N}$ and the middle third Cantor set, i.e., $\rho(M)=2\sum_{i\in M}(1/3)^{i+1}$ for $M\subseteq\mathbb{N}$, and pick a non-computable c.e. set $B\subseteq\mathbb{N}$ and put $\gamma=\rho(B)$.
We will construct a Cantor fan so that the first coordinate of the unique ramification point is $\gamma$, hence the fan will have a non-computable ramification point.
Let $\{B_s:s\in\mathbb{N}\}$ be a computable enumeration of $B$, and let $n_s$ denote the element enumerated into $B$ at stage $s$, where we may assume just one element is enumerated into $B$ at each stage.
Put $\gamma^{\min}_s=\rho(B_s)$ and $\gamma^{\max}_s=\rho(B_s\cup\{i\in\mathbb{N}:i\geq n_s\})$.
Note that $\gamma$ is not computable, and so $\gamma^{\min}_s\not=\gamma$ and $\gamma^{\max}_s\not=\gamma$ for any $s\in\mathbb{N}$.
This means that for every $s\in\mathbb{N}$ there exists $t>s$ such that $\gamma^{\min}_s\not=\gamma^{\min}_t$ and $\gamma^{\max}_s\not=\gamma^{\max}_t$.
By this observation, without loss of generality, we can assume that $\gamma^{\min}_s\not=\gamma^{\min}_t$ and $\gamma^{\max}_s\not=\gamma^{\max}_t$ whenever $s\not=t$.
We can also assume $1/3\leq\gamma^{\min}_s\leq\gamma^{\max}_s\leq 2/3$ for any $s\in\mathbb{N}$.

\medskip

\noindent
{\bf Stage $0$.}
We now start to construct a $\Pi^0_1$ Cantor fan $Q=\bigcap_{s\in\mathbb{N}}Q_s$.
At the first stage $0$, and for each $t\geq 0$, we define the following sets:
\[Z_{0,t}^{\rm st}=[-]^s_t([\gamma_0^{\min},\gamma_0^{\max}]\times[l_0^-,r_0^+]);\ Z_0^{\rm end}=[\gamma_0^{\min}-1/3,\gamma_0^{\min}]\times[l_0^-,r_0^+].\]
Moreover, we set $Q_0=Z_{0,0}^{\rm st}\cup Z_0^{\rm end}$.
By our choice of $P_0$, actually $Q_0=[\gamma_0^{\min}-1/3,\gamma_0^{\max}]\times[l_0^-,r_0^+]$.
$Z_{0,0}^{\rm st}$ is called {\em the straight block from $2/3$ to $1/3$ at stage $0$}, and $Z_0^{\rm end}$ is called {\em the end box at stage $0$}.
The {\em bounding box} of the block $Z_0^{\rm st}$ is defined by $[\gamma_0^{\min},\gamma_0^{\max}]\times[l_0^-,r_0^+]$.
{\em The collection of $0$-blocks at stage $t$} is $\mathcal{Z}_t(0)=\{Z_{0,t}^{\rm st}\}$.
We declare that $Z_0^{\rm st}$ is the first block, and that $\touch{[\leftarrow]}Z_0^{\rm st}$.

\medskip

\noindent
{\bf Stage $s+1$.}
Inductively assume that we have already constructed the collection of $u$-blocks $\mathcal{Z}_t(u)$ at stage $t\geq u$ is already defined for every $u\leq s$.
For any $u$, we think of the collection $\mathcal{Z}(u)=\{\mathcal{Z}_t(u):t\geq u\}$ as a finite set $\{Z^u_i\}_{i<\#\mathcal{Z}_u(u)}$ of computable functions $Z_i^u:\{t\in\nn:t\geq u\}\to\bigcup_{t}\mathcal{Z}_t(u)$ such that $\mathcal{Z}_t(u)=\{Z^u_i(t):i<\#\mathcal{Z}_u(u)\}$ for each $t\geq u$.
We inductively assume that the collection $\mathcal{Z}(u)=\{\mathcal{Z}_t(u):t\geq u\}$ satisfies the following conditions:
\begin{enumerate}
\item[(IH1)] For each $Z\in\mathcal{Z}(u)$ and for each $t\geq v\geq u$, $Z(t)\subseteq Z(v)$.
\item[(IH2)] There is a computable function $f:\mathbb{R}^2\to\mathbb{R}^2$ such that $f\res\bigcup\bigcup_{u\leq s}\mathcal{Z}_t(u)$ is a homeomorphism between $\bigcup\bigcup_{u\leq s}\mathcal{Z}_t(u)$ and $P_t\times [0,1]$ for any $t\geq s$.
\item[(IH3)] There are $y,z,\zeta\in\mathbb{Q}$ such that the blocks $Z_{s,t}^{\rm st}$ and $Z_s^{\rm end}$ are constructed as follows:
\begin{align*}
Z_{s,t}^{\rm st}&=[-]^s_t([\gamma_s^{\min},\gamma_s^{\max}]\times[y+zl_s^-,y+zr_s^+]);\\
Z_s^{\rm end}&=[\gamma_s^{\min}-\zeta,\gamma_s^{\min}]\times[y+zl_s^-,y+zr_s^+].
\end{align*}
Here, a computable closed set $Q_s$, {\em an approximation of our $\Pi^0_1$ Cantor fan $Q$ at stage $s$}, is defined by $Q_s=Z_s^{\rm end}\cup\bigcup\bigcup_{u\leq s}\mathcal{Z}_{s}(u)$.
\end{enumerate}

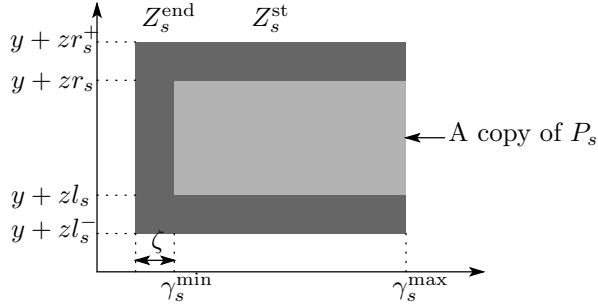
\begin{figure}[t]\centering
  \begin{center}
%WinTpicVersion3.08
\unitlength 0.1in
\begin{picture}( 24.5000, 14.2000)(  1.4000,-18.0000)
% BOX 2 5 1 0
% 2 990 800 2190 1400
% 
\special{pn 8}%
\special{sh 0.300}%
\special{pa 990 800}%
\special{pa 2190 800}%
\special{pa 2190 1400}%
\special{pa 990 1400}%
\special{pa 990 800}%
\special{ip}%
% BOX 2 5 0 0
% 2 2190 1600 990 1400
% 
\special{pn 8}%
\special{sh 0.600}%
\special{pa 2190 1600}%
\special{pa 990 1600}%
\special{pa 990 1400}%
\special{pa 2190 1400}%
\special{pa 2190 1600}%
\special{ip}%
% BOX 2 5 0 0
% 2 2190 800 990 600
% 
\special{pn 8}%
\special{sh 0.600}%
\special{pa 2190 800}%
\special{pa 990 800}%
\special{pa 990 600}%
\special{pa 2190 600}%
\special{pa 2190 800}%
\special{ip}%
% BOX 2 5 0 0
% 2 990 1600 790 600
% 
\special{pn 8}%
\special{sh 0.600}%
\special{pa 990 1600}%
\special{pa 790 1600}%
\special{pa 790 600}%
\special{pa 990 600}%
\special{pa 990 1600}%
\special{ip}%
% VECTOR 2 0 3 0
% 4 590 1800 590 400 590 1800 2590 1800
% 
\special{pn 8}%
\special{pa 590 1800}%
\special{pa 590 400}%
\special{fp}%
\special{sh 1}%
\special{pa 590 400}%
\special{pa 570 468}%
\special{pa 590 454}%
\special{pa 610 468}%
\special{pa 590 400}%
\special{fp}%
\special{pa 590 1800}%
\special{pa 2590 1800}%
\special{fp}%
\special{sh 1}%
\special{pa 2590 1800}%
\special{pa 2524 1780}%
\special{pa 2538 1800}%
\special{pa 2524 1820}%
\special{pa 2590 1800}%
\special{fp}%
% LINE 2 2 3 0
% 14 790 1800 790 1600 990 1800 990 1600 2190 1800 2190 1600 790 1600 590 1600 790 1400 590 1400 790 600 590 600 790 800 590 800
% 
\special{pn 8}%
\special{pa 790 1800}%
\special{pa 790 1600}%
\special{dt 0.045}%
\special{pa 990 1800}%
\special{pa 990 1600}%
\special{dt 0.045}%
\special{pa 2190 1800}%
\special{pa 2190 1600}%
\special{dt 0.045}%
\special{pa 790 1600}%
\special{pa 590 1600}%
\special{dt 0.045}%
\special{pa 790 1400}%
\special{pa 590 1400}%
\special{dt 0.045}%
\special{pa 790 600}%
\special{pa 590 600}%
\special{dt 0.045}%
\special{pa 790 800}%
\special{pa 590 800}%
\special{dt 0.045}%
% VECTOR 2 0 3 0
% 2 2390 1100 2190 1100
% 
\special{pn 8}%
\special{pa 2390 1100}%
\special{pa 2190 1100}%
\special{fp}%
\special{sh 1}%
\special{pa 2190 1100}%
\special{pa 2258 1120}%
\special{pa 2244 1100}%
\special{pa 2258 1080}%
\special{pa 2190 1100}%
\special{fp}%
% STR 2 0 3 0
% 3 2410 1095 2410 1145 2 0
% A copy of $P_s$
\put(24.1000,-11.4500){\makebox(0,0)[lb]{A copy of $P_s$}}%
% STR 2 0 3 0
% 3 2120 1900 2120 1950 2 0
% $\gamma_s^{\max}$
\put(21.2000,-19.5000){\makebox(0,0)[lb]{$\gamma_s^{\max}$}}%
% STR 2 0 3 0
% 3 920 1900 920 1950 2 0
% $\gamma_s^{\min}$
\put(9.2000,-19.5000){\makebox(0,0)[lb]{$\gamma_s^{\min}$}}%
% VECTOR 2 0 3 0
% 4 890 1740 990 1740 890 1740 790 1740
% 
\special{pn 8}%
\special{pa 890 1740}%
\special{pa 990 1740}%
\special{fp}%
\special{sh 1}%
\special{pa 990 1740}%
\special{pa 924 1720}%
\special{pa 938 1740}%
\special{pa 924 1760}%
\special{pa 990 1740}%
\special{fp}%
\special{pa 890 1740}%
\special{pa 790 1740}%
\special{fp}%
\special{sh 1}%
\special{pa 790 1740}%
\special{pa 858 1760}%
\special{pa 844 1740}%
\special{pa 858 1720}%
\special{pa 790 1740}%
\special{fp}%
% STR 2 0 3 0
% 3 870 1650 870 1700 2 0
% $\zeta$
\put(8.7000,-17.0000){\makebox(0,0)[lb]{$\zeta$}}%
% STR 2 0 3 0
% 3 145 1600 145 1650 2 0
% $y+zl_s^-$
\put(1.4500,-16.5000){\makebox(0,0)[lb]{$y+zl_s^-$}}%
% STR 2 0 3 0
% 3 145 1415 145 1465 2 0
% $y+zl_s$
\put(1.4500,-14.6500){\makebox(0,0)[lb]{$y+zl_s$}}%
% STR 2 0 3 0
% 3 145 815 145 865 2 0
% $y+zr_s$
\put(1.4500,-8.6500){\makebox(0,0)[lb]{$y+zr_s$}}%
% STR 2 0 3 0
% 3 140 610 140 660 2 0
% $y+zr_s^+$
\put(1.4000,-6.6000){\makebox(0,0)[lb]{$y+zr_s^+$}}%
% STR 2 0 3 0
% 3 1390 500 1390 550 2 0
% $Z_s^{\rm st}$
\put(13.9000,-5.5000){\makebox(0,0)[lb]{$Z_s^{\rm st}$}}%
% STR 2 0 3 0
% 3 820 500 820 550 2 0
% $Z_s^{\rm end}$
\put(8.2000,-5.5000){\makebox(0,0)[lb]{$Z_s^{\rm end}$}}%
\end{picture}%
  \end{center}
 \vspace{-0.5em}
\caption{The active block $Z_s^{\rm st}\cup Z_s^{\rm end}$ at stage $s$.}
  \label{fig:mthm2}
\end{figure}

%\medskip

\noindent
{\bf Non-injured Case.}
First we consider the case $[\gamma^{\min}_{s+1},\gamma^{\max}_{s+1}]\subseteq[\gamma^{\min}_s,\gamma^{\max}_s]$, i.e., this is the case that our construction is {\em not injured} at stage $s+1$.
In this case, we construct $(s+1)$-blocks in the active block $Z_s^{\rm st}\cup Z_s^{\rm end}$.
We will define $Z_t(s,i,j)$ and ${\rm Box}(s,i,j)={\rm Box}(Z_t(s,i,j))$ for each $j<6$.
{\em The first two corner blocks} at stage $t\geq s+1$ are defined by:
\begin{align*}
{\rm Box}(s,0)=&[\gamma_s^{\min}-\zeta,\gamma_s^{\min}]\times[y+zl_s^-,y+zr_s^*],\\
Z_t(s,0)=&[\llcorner]^s_t([\gamma_s^{\min}-\zeta,\gamma_s^{\min}]\times[y+zl_s^-,y+zr_s^+])\cap{\rm Box}(s,0),\\
{\rm Box}(s,1)=&[\gamma_s^{\min}-\zeta,\gamma_s^{\min}]\times[y+zr_s^*,y+zr_s^+],\\
Z_t(s,1)=&[\ulcorner]^s_t({\rm Box}(s,1)).
\end{align*}

\begin{sublem}\label{sublem:1}
$Z_t(s,0)\cup Z_t(s,1)\subseteq Z_s^{\rm end}$ for any $t\geq s+1$.
%%%%
\end{sublem}

\begin{sublem}\label{sublem:2}
$Z_{s,t}^{\rm st}\touch{[\leftarrow]}Z_t(s,0)\touch{[\uparrow]}Z_t(s,1)$, for any $t\geq s+1$.
%%%%
\end{sublem}

%We consider the following two line segments $L^*=\{\gamma_s^{\min}\}\times[y+zl_s,y+zr_s]$, and $L_s^{01}=[\gamma_s^{\min}-\zeta,\gamma_s^{\min}]\times\{y+z(r_s+2\cdot 3^{-(s+2)})\}$.

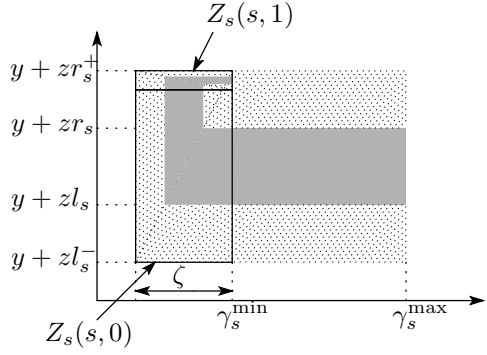
\begin{figure}[t]\centering
  \begin{center}
%WinTpicVersion3.08
\unitlength 0.1in
\begin{picture}( 24.5000, 16.9000)(  1.0000,-17.9500)
% BOX 2 5 1 0
% 2 1250 795 2150 1195
% 
\special{pn 8}%
\special{sh 0.300}%
\special{pa 1250 796}%
\special{pa 2150 796}%
\special{pa 2150 1196}%
\special{pa 1250 1196}%
\special{pa 1250 796}%
\special{ip}%
% VECTOR 2 0 3 0
% 4 550 1695 550 295 550 1695 2550 1695
% 
\special{pn 8}%
\special{pa 550 1696}%
\special{pa 550 296}%
\special{fp}%
\special{sh 1}%
\special{pa 550 296}%
\special{pa 530 362}%
\special{pa 550 348}%
\special{pa 570 362}%
\special{pa 550 296}%
\special{fp}%
\special{pa 550 1696}%
\special{pa 2550 1696}%
\special{fp}%
\special{sh 1}%
\special{pa 2550 1696}%
\special{pa 2484 1676}%
\special{pa 2498 1696}%
\special{pa 2484 1716}%
\special{pa 2550 1696}%
\special{fp}%
% STR 2 0 3 0
% 3 2080 1795 2080 1845 2 0
% $\gamma_s^{\max}$
\put(20.8000,-18.4500){\makebox(0,0)[lb]{$\gamma_s^{\max}$}}%
% STR 2 0 3 0
% 3 1170 1795 1170 1845 2 0
% $\gamma_s^{\min}$
\put(11.7000,-18.4500){\makebox(0,0)[lb]{$\gamma_s^{\min}$}}%
% STR 2 0 3 0
% 3 940 1580 940 1630 2 0
% $\zeta$
\put(9.4000,-16.3000){\makebox(0,0)[lb]{$\zeta$}}%
% STR 2 0 3 0
% 3 105 1495 105 1545 2 0
% $y+zl_s^-$
\put(1.0500,-15.4500){\makebox(0,0)[lb]{$y+zl_s^-$}}%
% STR 2 0 3 0
% 3 105 1180 105 1230 2 0
% $y+zl_s$
\put(1.0500,-12.3000){\makebox(0,0)[lb]{$y+zl_s$}}%
% STR 2 0 3 0
% 3 105 785 105 835 2 0
% $y+zr_s$
\put(1.0500,-8.3500){\makebox(0,0)[lb]{$y+zr_s$}}%
% STR 2 0 3 0
% 3 100 505 100 555 2 0
% $y+zr_s^+$
\put(1.0000,-5.5500){\makebox(0,0)[lb]{$y+zr_s^+$}}%
% POLYGON 3 5 3 0
% 6 2150 1195 2150 1495 1250 1495 1250 1195 1250 1195 2150 1195
% 
\special{pn 4}%
\special{pa 2150 1196}%
\special{pa 2150 1496}%
\special{pa 1250 1496}%
\special{pa 1250 1196}%
\special{pa 1250 1196}%
\special{pa 2150 1196}%
\special{ip}%
% POLYGON 3 5 3 0
% 6 2150 495 2150 795 1250 795 1250 495 1250 495 2150 495
% 
\special{pn 4}%
\special{pa 2150 496}%
\special{pa 2150 796}%
\special{pa 1250 796}%
\special{pa 1250 496}%
\special{pa 1250 496}%
\special{pa 2150 496}%
\special{ip}%
% POLYGON 2 5 3 0
% 6 1250 1495 750 1495 750 495 1250 495 1250 495 1250 1495
% 
\special{pn 8}%
\special{pa 1250 1496}%
\special{pa 750 1496}%
\special{pa 750 496}%
\special{pa 1250 496}%
\special{pa 1250 496}%
\special{pa 1250 1496}%
\special{ip}%
% LINE 3 2 3 0
% 60 1750 495 1450 795 1720 495 1420 795 1690 495 1390 795 1660 495 1360 795 1630 495 1330 795 1600 495 1300 795 1570 495 1270 795 1540 495 1250 785 1510 495 1250 755 1480 495 1250 725 1450 495 1250 695 1420 495 1250 665 1390 495 1250 635 1360 495 1250 605 1330 495 1250 575 1300 495 1250 545 1270 495 1250 515 1780 495 1480 795 1810 495 1510 795 1840 495 1540 795 1870 495 1570 795 1900 495 1600 795 1930 495 1630 795 1960 495 1660 795 1990 495 1690 795 2020 495 1720 795 2050 495 1750 795 2080 495 1780 795 2110 495 1810 795 2140 495 1840 795
% 
\special{pn 4}%
\special{pa 1750 496}%
\special{pa 1450 796}%
\special{dt 0.027}%
\special{pa 1720 496}%
\special{pa 1420 796}%
\special{dt 0.027}%
\special{pa 1690 496}%
\special{pa 1390 796}%
\special{dt 0.027}%
\special{pa 1660 496}%
\special{pa 1360 796}%
\special{dt 0.027}%
\special{pa 1630 496}%
\special{pa 1330 796}%
\special{dt 0.027}%
\special{pa 1600 496}%
\special{pa 1300 796}%
\special{dt 0.027}%
\special{pa 1570 496}%
\special{pa 1270 796}%
\special{dt 0.027}%
\special{pa 1540 496}%
\special{pa 1250 786}%
\special{dt 0.027}%
\special{pa 1510 496}%
\special{pa 1250 756}%
\special{dt 0.027}%
\special{pa 1480 496}%
\special{pa 1250 726}%
\special{dt 0.027}%
\special{pa 1450 496}%
\special{pa 1250 696}%
\special{dt 0.027}%
\special{pa 1420 496}%
\special{pa 1250 666}%
\special{dt 0.027}%
\special{pa 1390 496}%
\special{pa 1250 636}%
\special{dt 0.027}%
\special{pa 1360 496}%
\special{pa 1250 606}%
\special{dt 0.027}%
\special{pa 1330 496}%
\special{pa 1250 576}%
\special{dt 0.027}%
\special{pa 1300 496}%
\special{pa 1250 546}%
\special{dt 0.027}%
\special{pa 1270 496}%
\special{pa 1250 516}%
\special{dt 0.027}%
\special{pa 1780 496}%
\special{pa 1480 796}%
\special{dt 0.027}%
\special{pa 1810 496}%
\special{pa 1510 796}%
\special{dt 0.027}%
\special{pa 1840 496}%
\special{pa 1540 796}%
\special{dt 0.027}%
\special{pa 1870 496}%
\special{pa 1570 796}%
\special{dt 0.027}%
\special{pa 1900 496}%
\special{pa 1600 796}%
\special{dt 0.027}%
\special{pa 1930 496}%
\special{pa 1630 796}%
\special{dt 0.027}%
\special{pa 1960 496}%
\special{pa 1660 796}%
\special{dt 0.027}%
\special{pa 1990 496}%
\special{pa 1690 796}%
\special{dt 0.027}%
\special{pa 2020 496}%
\special{pa 1720 796}%
\special{dt 0.027}%
\special{pa 2050 496}%
\special{pa 1750 796}%
\special{dt 0.027}%
\special{pa 2080 496}%
\special{pa 1780 796}%
\special{dt 0.027}%
\special{pa 2110 496}%
\special{pa 1810 796}%
\special{dt 0.027}%
\special{pa 2140 496}%
\special{pa 1840 796}%
\special{dt 0.027}%
% LINE 3 2 3 1
% 18 2150 515 1870 795 2150 545 1900 795 2150 575 1930 795 2150 605 1960 795 2150 635 1990 795 2150 665 2020 795 2150 695 2050 795 2150 725 2080 795 2150 755 2110 795
% 
\special{pn 4}%
\special{pa 2150 516}%
\special{pa 1870 796}%
\special{dt 0.027}%
\special{pa 2150 546}%
\special{pa 1900 796}%
\special{dt 0.027}%
\special{pa 2150 576}%
\special{pa 1930 796}%
\special{dt 0.027}%
\special{pa 2150 606}%
\special{pa 1960 796}%
\special{dt 0.027}%
\special{pa 2150 636}%
\special{pa 1990 796}%
\special{dt 0.027}%
\special{pa 2150 666}%
\special{pa 2020 796}%
\special{dt 0.027}%
\special{pa 2150 696}%
\special{pa 2050 796}%
\special{dt 0.027}%
\special{pa 2150 726}%
\special{pa 2080 796}%
\special{dt 0.027}%
\special{pa 2150 756}%
\special{pa 2110 796}%
\special{dt 0.027}%
% LINE 3 2 3 0
% 60 1860 1195 1560 1495 1830 1195 1530 1495 1800 1195 1500 1495 1770 1195 1470 1495 1740 1195 1440 1495 1710 1195 1410 1495 1680 1195 1380 1495 1650 1195 1350 1495 1620 1195 1320 1495 1590 1195 1290 1495 1560 1195 1260 1495 1530 1195 1250 1475 1500 1195 1250 1445 1470 1195 1250 1415 1440 1195 1250 1385 1410 1195 1250 1355 1380 1195 1250 1325 1350 1195 1250 1295 1320 1195 1250 1265 1290 1195 1250 1235 1890 1195 1590 1495 1920 1195 1620 1495 1950 1195 1650 1495 1980 1195 1680 1495 2010 1195 1710 1495 2040 1195 1740 1495 2070 1195 1770 1495 2100 1195 1800 1495 2130 1195 1830 1495 2150 1205 1860 1495
% 
\special{pn 4}%
\special{pa 1860 1196}%
\special{pa 1560 1496}%
\special{dt 0.027}%
\special{pa 1830 1196}%
\special{pa 1530 1496}%
\special{dt 0.027}%
\special{pa 1800 1196}%
\special{pa 1500 1496}%
\special{dt 0.027}%
\special{pa 1770 1196}%
\special{pa 1470 1496}%
\special{dt 0.027}%
\special{pa 1740 1196}%
\special{pa 1440 1496}%
\special{dt 0.027}%
\special{pa 1710 1196}%
\special{pa 1410 1496}%
\special{dt 0.027}%
\special{pa 1680 1196}%
\special{pa 1380 1496}%
\special{dt 0.027}%
\special{pa 1650 1196}%
\special{pa 1350 1496}%
\special{dt 0.027}%
\special{pa 1620 1196}%
\special{pa 1320 1496}%
\special{dt 0.027}%
\special{pa 1590 1196}%
\special{pa 1290 1496}%
\special{dt 0.027}%
\special{pa 1560 1196}%
\special{pa 1260 1496}%
\special{dt 0.027}%
\special{pa 1530 1196}%
\special{pa 1250 1476}%
\special{dt 0.027}%
\special{pa 1500 1196}%
\special{pa 1250 1446}%
\special{dt 0.027}%
\special{pa 1470 1196}%
\special{pa 1250 1416}%
\special{dt 0.027}%
\special{pa 1440 1196}%
\special{pa 1250 1386}%
\special{dt 0.027}%
\special{pa 1410 1196}%
\special{pa 1250 1356}%
\special{dt 0.027}%
\special{pa 1380 1196}%
\special{pa 1250 1326}%
\special{dt 0.027}%
\special{pa 1350 1196}%
\special{pa 1250 1296}%
\special{dt 0.027}%
\special{pa 1320 1196}%
\special{pa 1250 1266}%
\special{dt 0.027}%
\special{pa 1290 1196}%
\special{pa 1250 1236}%
\special{dt 0.027}%
\special{pa 1890 1196}%
\special{pa 1590 1496}%
\special{dt 0.027}%
\special{pa 1920 1196}%
\special{pa 1620 1496}%
\special{dt 0.027}%
\special{pa 1950 1196}%
\special{pa 1650 1496}%
\special{dt 0.027}%
\special{pa 1980 1196}%
\special{pa 1680 1496}%
\special{dt 0.027}%
\special{pa 2010 1196}%
\special{pa 1710 1496}%
\special{dt 0.027}%
\special{pa 2040 1196}%
\special{pa 1740 1496}%
\special{dt 0.027}%
\special{pa 2070 1196}%
\special{pa 1770 1496}%
\special{dt 0.027}%
\special{pa 2100 1196}%
\special{pa 1800 1496}%
\special{dt 0.027}%
\special{pa 2130 1196}%
\special{pa 1830 1496}%
\special{dt 0.027}%
\special{pa 2150 1206}%
\special{pa 1860 1496}%
\special{dt 0.027}%
% LINE 3 2 3 1
% 18 2150 1235 1890 1495 2150 1265 1920 1495 2150 1295 1950 1495 2150 1325 1980 1495 2150 1355 2010 1495 2150 1385 2040 1495 2150 1415 2070 1495 2150 1445 2100 1495 2150 1475 2130 1495
% 
\special{pn 4}%
\special{pa 2150 1236}%
\special{pa 1890 1496}%
\special{dt 0.027}%
\special{pa 2150 1266}%
\special{pa 1920 1496}%
\special{dt 0.027}%
\special{pa 2150 1296}%
\special{pa 1950 1496}%
\special{dt 0.027}%
\special{pa 2150 1326}%
\special{pa 1980 1496}%
\special{dt 0.027}%
\special{pa 2150 1356}%
\special{pa 2010 1496}%
\special{dt 0.027}%
\special{pa 2150 1386}%
\special{pa 2040 1496}%
\special{dt 0.027}%
\special{pa 2150 1416}%
\special{pa 2070 1496}%
\special{dt 0.027}%
\special{pa 2150 1446}%
\special{pa 2100 1496}%
\special{dt 0.027}%
\special{pa 2150 1476}%
\special{pa 2130 1496}%
\special{dt 0.027}%
% LINE 3 2 3 0
% 60 1250 1115 870 1495 1250 1085 840 1495 1250 1055 810 1495 1250 1025 780 1495 1250 995 755 1490 1250 965 750 1465 1250 935 750 1435 1250 905 750 1405 1250 875 750 1375 1250 845 750 1345 1250 815 750 1315 1250 785 750 1285 1250 755 750 1255 1250 725 750 1225 1250 695 750 1195 1250 665 750 1165 1250 635 750 1135 1250 605 750 1105 1250 575 750 1075 1250 545 750 1045 1250 515 750 1015 1240 495 750 985 1210 495 750 955 1180 495 750 925 1150 495 750 895 1120 495 750 865 1090 495 750 835 1060 495 750 805 1030 495 750 775 1000 495 750 745
% 
\special{pn 4}%
\special{pa 1250 1116}%
\special{pa 870 1496}%
\special{dt 0.027}%
\special{pa 1250 1086}%
\special{pa 840 1496}%
\special{dt 0.027}%
\special{pa 1250 1056}%
\special{pa 810 1496}%
\special{dt 0.027}%
\special{pa 1250 1026}%
\special{pa 780 1496}%
\special{dt 0.027}%
\special{pa 1250 996}%
\special{pa 756 1490}%
\special{dt 0.027}%
\special{pa 1250 966}%
\special{pa 750 1466}%
\special{dt 0.027}%
\special{pa 1250 936}%
\special{pa 750 1436}%
\special{dt 0.027}%
\special{pa 1250 906}%
\special{pa 750 1406}%
\special{dt 0.027}%
\special{pa 1250 876}%
\special{pa 750 1376}%
\special{dt 0.027}%
\special{pa 1250 846}%
\special{pa 750 1346}%
\special{dt 0.027}%
\special{pa 1250 816}%
\special{pa 750 1316}%
\special{dt 0.027}%
\special{pa 1250 786}%
\special{pa 750 1286}%
\special{dt 0.027}%
\special{pa 1250 756}%
\special{pa 750 1256}%
\special{dt 0.027}%
\special{pa 1250 726}%
\special{pa 750 1226}%
\special{dt 0.027}%
\special{pa 1250 696}%
\special{pa 750 1196}%
\special{dt 0.027}%
\special{pa 1250 666}%
\special{pa 750 1166}%
\special{dt 0.027}%
\special{pa 1250 636}%
\special{pa 750 1136}%
\special{dt 0.027}%
\special{pa 1250 606}%
\special{pa 750 1106}%
\special{dt 0.027}%
\special{pa 1250 576}%
\special{pa 750 1076}%
\special{dt 0.027}%
\special{pa 1250 546}%
\special{pa 750 1046}%
\special{dt 0.027}%
\special{pa 1250 516}%
\special{pa 750 1016}%
\special{dt 0.027}%
\special{pa 1240 496}%
\special{pa 750 986}%
\special{dt 0.027}%
\special{pa 1210 496}%
\special{pa 750 956}%
\special{dt 0.027}%
\special{pa 1180 496}%
\special{pa 750 926}%
\special{dt 0.027}%
\special{pa 1150 496}%
\special{pa 750 896}%
\special{dt 0.027}%
\special{pa 1120 496}%
\special{pa 750 866}%
\special{dt 0.027}%
\special{pa 1090 496}%
\special{pa 750 836}%
\special{dt 0.027}%
\special{pa 1060 496}%
\special{pa 750 806}%
\special{dt 0.027}%
\special{pa 1030 496}%
\special{pa 750 776}%
\special{dt 0.027}%
\special{pa 1000 496}%
\special{pa 750 746}%
\special{dt 0.027}%
% LINE 3 2 3 1
% 38 970 495 750 715 940 495 750 685 910 495 750 655 880 495 750 625 850 495 750 595 820 495 750 565 790 495 750 535 1250 1145 900 1495 1250 1175 930 1495 1250 1205 960 1495 1250 1235 990 1495 1250 1265 1020 1495 1250 1295 1050 1495 1250 1325 1080 1495 1250 1355 1110 1495 1250 1385 1140 1495 1250 1415 1170 1495 1250 1445 1200 1495 1250 1475 1230 1495
% 
\special{pn 4}%
\special{pa 970 496}%
\special{pa 750 716}%
\special{dt 0.027}%
\special{pa 940 496}%
\special{pa 750 686}%
\special{dt 0.027}%
\special{pa 910 496}%
\special{pa 750 656}%
\special{dt 0.027}%
\special{pa 880 496}%
\special{pa 750 626}%
\special{dt 0.027}%
\special{pa 850 496}%
\special{pa 750 596}%
\special{dt 0.027}%
\special{pa 820 496}%
\special{pa 750 566}%
\special{dt 0.027}%
\special{pa 790 496}%
\special{pa 750 536}%
\special{dt 0.027}%
\special{pa 1250 1146}%
\special{pa 900 1496}%
\special{dt 0.027}%
\special{pa 1250 1176}%
\special{pa 930 1496}%
\special{dt 0.027}%
\special{pa 1250 1206}%
\special{pa 960 1496}%
\special{dt 0.027}%
\special{pa 1250 1236}%
\special{pa 990 1496}%
\special{dt 0.027}%
\special{pa 1250 1266}%
\special{pa 1020 1496}%
\special{dt 0.027}%
\special{pa 1250 1296}%
\special{pa 1050 1496}%
\special{dt 0.027}%
\special{pa 1250 1326}%
\special{pa 1080 1496}%
\special{dt 0.027}%
\special{pa 1250 1356}%
\special{pa 1110 1496}%
\special{dt 0.027}%
\special{pa 1250 1386}%
\special{pa 1140 1496}%
\special{dt 0.027}%
\special{pa 1250 1416}%
\special{pa 1170 1496}%
\special{dt 0.027}%
\special{pa 1250 1446}%
\special{pa 1200 1496}%
\special{dt 0.027}%
\special{pa 1250 1476}%
\special{pa 1230 1496}%
\special{dt 0.027}%
% LINE 3 2 3 0
% 4 750 1495 1250 495 750 495 1250 595
% 
\special{pn 4}%
\special{pa 750 1496}%
\special{pa 1250 496}%
\special{dt 0.027}%
\special{pa 750 496}%
\special{pa 1250 596}%
\special{dt 0.027}%
% POLYGON 3 5 1 0
% 6 1250 795 1100 795 900 1195 1250 1195 1250 1195 1250 795
% 
\special{pn 4}%
\special{sh 0.300}%
\special{pa 1250 796}%
\special{pa 1100 796}%
\special{pa 900 1196}%
\special{pa 1250 1196}%
\special{pa 1250 1196}%
\special{pa 1250 796}%
\special{ip}%
% POLYGON 2 5 1 0
% 6 900 1195 900 525 1100 565 1100 795 1100 795 900 1195
% 
\special{pn 8}%
\special{sh 0.300}%
\special{pa 900 1196}%
\special{pa 900 526}%
\special{pa 1100 566}%
\special{pa 1100 796}%
\special{pa 1100 796}%
\special{pa 900 1196}%
\special{ip}%
% POLYGON 2 5 1 0
% 7 900 525 1100 565 1250 565 1250 525 900 525 900 525 900 525
% 
\special{pn 8}%
\special{sh 0.300}%
\special{pa 900 526}%
\special{pa 1100 566}%
\special{pa 1250 566}%
\special{pa 1250 526}%
\special{pa 900 526}%
\special{pa 900 526}%
\special{pa 900 526}%
\special{ip}%
% LINE 2 2 3 0
% 14 550 1195 750 1195 550 1495 750 1495 550 795 750 795 550 495 750 495 1250 1495 1250 1695 750 1495 750 1695 2150 1495 2150 1695
% 
\special{pn 8}%
\special{pa 550 1196}%
\special{pa 750 1196}%
\special{dt 0.045}%
\special{pa 550 1496}%
\special{pa 750 1496}%
\special{dt 0.045}%
\special{pa 550 796}%
\special{pa 750 796}%
\special{dt 0.045}%
\special{pa 550 496}%
\special{pa 750 496}%
\special{dt 0.045}%
\special{pa 1250 1496}%
\special{pa 1250 1696}%
\special{dt 0.045}%
\special{pa 750 1496}%
\special{pa 750 1696}%
\special{dt 0.045}%
\special{pa 2150 1496}%
\special{pa 2150 1696}%
\special{dt 0.045}%
% VECTOR 2 0 3 0
% 4 1050 1645 1250 1645 1050 1645 750 1645
% 
\special{pn 8}%
\special{pa 1050 1646}%
\special{pa 1250 1646}%
\special{fp}%
\special{sh 1}%
\special{pa 1250 1646}%
\special{pa 1184 1626}%
\special{pa 1198 1646}%
\special{pa 1184 1666}%
\special{pa 1250 1646}%
\special{fp}%
\special{pa 1050 1646}%
\special{pa 750 1646}%
\special{fp}%
\special{sh 1}%
\special{pa 750 1646}%
\special{pa 818 1666}%
\special{pa 804 1646}%
\special{pa 818 1626}%
\special{pa 750 1646}%
\special{fp}%
% BOX 2 0 3 0
% 2 1250 595 750 1495
% 
\special{pn 8}%
\special{pa 1250 596}%
\special{pa 750 596}%
\special{pa 750 1496}%
\special{pa 1250 1496}%
\special{pa 1250 596}%
\special{fp}%
% BOX 2 0 3 0
% 2 750 495 1250 595
% 
\special{pn 8}%
\special{pa 750 496}%
\special{pa 1250 496}%
\special{pa 1250 596}%
\special{pa 750 596}%
\special{pa 750 496}%
\special{fp}%
% VECTOR 2 0 3 0
% 2 1150 295 1050 495
% 
\special{pn 8}%
\special{pa 1150 296}%
\special{pa 1050 496}%
\special{fp}%
\special{sh 1}%
\special{pa 1050 496}%
\special{pa 1098 444}%
\special{pa 1074 448}%
\special{pa 1062 426}%
\special{pa 1050 496}%
\special{fp}%
% STR 2 0 3 0
% 3 1115 225 1115 275 2 0
% $Z_s(s,1)$
\put(11.1500,-2.7500){\makebox(0,0)[lb]{$Z_s(s,1)$}}%
% VECTOR 2 0 3 0
% 2 550 1795 850 1495
% 
\special{pn 8}%
\special{pa 550 1796}%
\special{pa 850 1496}%
\special{fp}%
\special{sh 1}%
\special{pa 850 1496}%
\special{pa 790 1528}%
\special{pa 812 1534}%
\special{pa 818 1556}%
\special{pa 850 1496}%
\special{fp}%
% STR 2 0 3 0
% 3 280 1890 280 1940 2 0
% $Z_s(s,0)$
\put(2.8000,-19.4000){\makebox(0,0)[lb]{$Z_s(s,0)$}}%
\end{picture}%
  \end{center}
 \vspace{-0.5em}
\caption{The first two corner blocks $Z_s(s,0)$ and $Z_s(s,1)$.}
  \label{fig:mthm3}
\end{figure}

The next block is {\em a straight block from $\gamma_s^{\min}$ to $\gamma_{s+1}^{\max}$} which is defined as follows:
\begin{align*}
{\rm Box}(s,2)=&[\gamma_s^{\min},\gamma_s^{\max}]\times[y+zr_s^*,y+zr_s^+].\\
Z_t(s,2)=&[-]({\rm Box}_t(s,2)).
\end{align*}
For given $a,b,\alpha,\beta\in\mathbb{Q}$, we can calculate $N_{0,s}(a,b;\alpha,\beta)$ and $N_{1,s}(a,b;\alpha,\beta)$ satisfying $N_{0,s}(a,b;\alpha,\beta)+N_{1,s}(a,b;\alpha,\beta)\cdot l_s^-=a+b\alpha$, and $N_{0,s}(a,b;\alpha,\beta)+N_{1,s}(a,b;\alpha,\beta)\cdot r_s^+=a+b\beta$.
Put $y^{\star}=N_{0,s}(y,z;r_s^*,r_s^+)$, and $z^{\star}=N_{1,s}(y,z;r_s^*,r_s^+)$.

\begin{sublem}\label{sublem:b1}
${\rm Box}(s,2)=[\gamma_s^{\min},\gamma_s^{\max}]\times[y^{\star}+z^{\star}l_s^-,y^{\star}+z^{\star}r_s^+]$.
%%%%
\end{sublem}

Put $\zeta^{\star}=(\gamma_s^{\max}-\gamma_{s+1}^{\max})/3^s$.
Note that $\zeta^{\star}>0$ since $\gamma_s^{\max}>\gamma_{s+1}^{\max}$.
We then again define {\em corner blocks}.
\begin{align*}
{\rm Box}(s,3)=&[\gamma_{s+1}^{\max},\gamma_{s+1}^{\max}+\zeta^{\star}]\times[y^{\star}+z^{\star}l_s^-,y^{\star}+z^{\star}r_s^*],\\
Z_t(s,3)=&[\lrcorner]^s_t([\gamma_{s+1}^{\max},\gamma_{s+1}^{\max}+\zeta^{\star}]\times[y^{\star}+z^{\star}l_s^-,y^{\star}+z^{\star}r_s^+])\cap{\rm Box}(s,3),\\
{\rm Box}(s,4)=&[\gamma_{s+1}^{\max},\gamma_{s+1}^{\max}+\zeta^{\star}]\times[y^{\star}+z^{\star}r_s^*,y^{\star}+z^{\star}r_s^+],\\
Z_t(s,4)=&[\urcorner]^s_t({\rm Box}(s,4)).
\end{align*}

Next, {\em a straight block from $\gamma_s^{\min}$ to $\gamma_{s+1}^{\max}$} is defined as follows:
\begin{align*}
{\rm Box}(s,5)&=[\gamma_{s+1}^{\min},\gamma_{s+1}^{\max}]\times[y^{\star}+z^{\star}r_s^*,y^{\star}+z^{\star}r_s^+],\\
Z_t(s,5)&=[-]^s_t[{\rm Box}(s,5)].
\end{align*}

Put $y^{\star\star}=N_{0,s}(y^{\star},z^{\star};r_s^*,r_s^+)$, and $z^{\star\star}=N_{1,s}(y^{\star},z^{\star};r_s^*,r_s^+)$.

\begin{sublem}\label{sublem:b2}
${\rm Box}(s,5)=[\gamma_s^{\min},\gamma_s^{\max}]\times[y^{\star\star}+z^{\star\star}l_s^-,y^{\star\star}+z^{\star\star}r_s^+]$.
%%%%
\end{sublem}

Put $\zeta^{\star\star}=(\gamma_{s+1}^{\min}-\gamma_{s}^{\min})/3^s$.
Note that $\zeta^{\star\star}>0$ since $\gamma_{s+1}^{\min}>\gamma_{s}^{\max}$.
{\em The end box at stage $s+1$} is:
\[Z(s,6)=[\gamma_{s+1}^{\min}-\zeta^{\star\star},\gamma_{s+1}^{\min}]\times[y^{\star\star}+z^{\star\star}l_s^-,y^{\star\star}+z^{\star\star}r_s^+].\]
Then put $Z_{s+1,t}^{\rm st}=Z_{t}(s,5)$, $Z_{s+1}^{\rm st}=Z_{s+1,s+1}^{\rm st}$, and $Z_{s+1}^{\rm end}=Z(s,6)$.
{\em The active block at stage $s+1$} is the set $Z_{s+1,s+1}^{\rm st}\cup Z_{s+1}^{\rm end}$, and {\em the collection of $(s+1)$-blocks at stage $t$} is defined by $\mathcal{Z}_t(s+1)=\{Z_t(s,i):i\leq 5\}$.
Clearly, our definition satisfies the induction hypothesis (IH3) at stage $s+1$. 

\begin{figure}[t]\centering
 \begin{minipage}{0.48\hsize}
  \begin{center}
%WinTpicVersion3.08
\unitlength 0.1in
\begin{picture}( 20.0000, 14.0000)(  6.0000,-16.0000)
% BOX 2 5 0 0
% 2 2200 1400 990 1200
% 
\special{pn 8}%
\special{sh 0.600}%
\special{pa 2200 1400}%
\special{pa 990 1400}%
\special{pa 990 1200}%
\special{pa 2200 1200}%
\special{pa 2200 1400}%
\special{ip}%
% BOX 2 5 0 0
% 2 2200 600 990 400
% 
\special{pn 8}%
\special{sh 0.600}%
\special{pa 2200 600}%
\special{pa 990 600}%
\special{pa 990 400}%
\special{pa 2200 400}%
\special{pa 2200 600}%
\special{ip}%
% BOX 2 5 0 0
% 2 1000 1400 800 400
% 
\special{pn 8}%
\special{sh 0.600}%
\special{pa 1000 1400}%
\special{pa 800 1400}%
\special{pa 800 400}%
\special{pa 1000 400}%
\special{pa 1000 1400}%
\special{ip}%
% BOX 2 5 1 0
% 2 867 600 2200 1200
% 
\special{pn 8}%
\special{sh 0.300}%
\special{pa 868 600}%
\special{pa 2200 600}%
\special{pa 2200 1200}%
\special{pa 868 1200}%
\special{pa 868 600}%
\special{ip}%
% BOX 2 5 1 0
% 2 867 600 933 467
% 
\special{pn 8}%
\special{sh 0.300}%
\special{pa 868 600}%
\special{pa 934 600}%
\special{pa 934 468}%
\special{pa 868 468}%
\special{pa 868 600}%
\special{ip}%
% BOX 2 5 1 0
% 2 933 467 1800 533
% 
\special{pn 8}%
\special{sh 0.300}%
\special{pa 934 468}%
\special{pa 1800 468}%
\special{pa 1800 534}%
\special{pa 934 534}%
\special{pa 934 468}%
\special{ip}%
% BOX 2 5 1 0
% 2 1800 467 1767 417
% 
\special{pn 8}%
\special{sh 0.300}%
\special{pa 1800 468}%
\special{pa 1768 468}%
\special{pa 1768 418}%
\special{pa 1800 418}%
\special{pa 1800 468}%
\special{ip}%
% BOX 2 5 1 0
% 2 1767 433 1200 417
% 
\special{pn 8}%
\special{sh 0.300}%
\special{pa 1768 434}%
\special{pa 1200 434}%
\special{pa 1200 418}%
\special{pa 1768 418}%
\special{pa 1768 434}%
\special{ip}%
% VECTOR 2 0 3 0
% 4 600 1600 600 200 600 1600 2600 1600
% 
\special{pn 8}%
\special{pa 600 1600}%
\special{pa 600 200}%
\special{fp}%
\special{sh 1}%
\special{pa 600 200}%
\special{pa 580 268}%
\special{pa 600 254}%
\special{pa 620 268}%
\special{pa 600 200}%
\special{fp}%
\special{pa 600 1600}%
\special{pa 2600 1600}%
\special{fp}%
\special{sh 1}%
\special{pa 2600 1600}%
\special{pa 2534 1580}%
\special{pa 2548 1600}%
\special{pa 2534 1620}%
\special{pa 2600 1600}%
\special{fp}%
% LINE 2 2 3 0
% 8 2200 1600 2200 1400 2200 400 2200 200 1000 1600 1000 1400 1000 400 1000 200
% 
\special{pn 8}%
\special{pa 2200 1600}%
\special{pa 2200 1400}%
\special{dt 0.045}%
\special{pa 2200 400}%
\special{pa 2200 200}%
\special{dt 0.045}%
\special{pa 1000 1600}%
\special{pa 1000 1400}%
\special{dt 0.045}%
\special{pa 1000 400}%
\special{pa 1000 200}%
\special{dt 0.045}%
% LINE 2 2 3 0
% 4 1200 1600 1200 1400 1200 400 1200 200
% 
\special{pn 8}%
\special{pa 1200 1600}%
\special{pa 1200 1400}%
\special{dt 0.045}%
\special{pa 1200 400}%
\special{pa 1200 200}%
\special{dt 0.045}%
% STR 2 0 3 0
% 3 2135 1710 2135 1760 2 0
% $\gamma_s^{\max}$
\put(21.3500,-17.6000){\makebox(0,0)[lb]{$\gamma_s^{\max}$}}%
% STR 2 0 3 0
% 3 930 1710 930 1760 2 0
% $\gamma_s^{\min}$
\put(9.3000,-17.6000){\makebox(0,0)[lb]{$\gamma_s^{\min}$}}%
% STR 2 0 3 0
% 3 1165 1715 1165 1765 2 0
% $\gamma_{s+1}^{\min}$
\put(11.6500,-17.6500){\makebox(0,0)[lb]{$\gamma_{s+1}^{\min}$}}%
% STR 2 0 3 0
% 3 1665 1715 1665 1765 2 0
% $\gamma_{s+1}^{\max}$
\put(16.6500,-17.6500){\makebox(0,0)[lb]{$\gamma_{s+1}^{\max}$}}%
% LINE 2 2 3 0
% 4 1740 1600 1740 1400 1740 400 1740 200
% 
\special{pn 8}%
\special{pa 1740 1600}%
\special{pa 1740 1400}%
\special{dt 0.045}%
\special{pa 1740 400}%
\special{pa 1740 200}%
\special{dt 0.045}%
\end{picture}%
  \end{center}
 \vspace{-0.5em}
\caption{$Z_s(s-1,5)\cup\bigcup\mathcal{Z}_s(s+1)$.}
  \label{fig:mthm4}
 \end{minipage}
 \begin{minipage}{0.48\hsize}
  \begin{center}
%WinTpicVersion3.08
\unitlength 0.1in
\begin{picture}( 25.4000, 14.0000)(  1.6000,-16.0000)
% BOX 2 5 0 0
% 2 2300 607 980 1201
% 
\special{pn 8}%
\special{sh 0.600}%
\special{pa 2300 608}%
\special{pa 980 608}%
\special{pa 980 1202}%
\special{pa 2300 1202}%
\special{pa 2300 608}%
\special{ip}%
% BOX 2 5 0 0
% 2 980 610 1046 478
% 
\special{pn 8}%
\special{sh 0.600}%
\special{pa 980 610}%
\special{pa 1046 610}%
\special{pa 1046 478}%
\special{pa 980 478}%
\special{pa 980 610}%
\special{ip}%
% BOX 2 5 0 0
% 2 1046 475 1904 540
% 
\special{pn 8}%
\special{sh 0.600}%
\special{pa 1046 476}%
\special{pa 1904 476}%
\special{pa 1904 540}%
\special{pa 1046 540}%
\special{pa 1046 476}%
\special{ip}%
% BOX 2 5 0 0
% 2 1904 475 1871 425
% 
\special{pn 8}%
\special{sh 0.600}%
\special{pa 1904 476}%
\special{pa 1872 476}%
\special{pa 1872 426}%
\special{pa 1904 426}%
\special{pa 1904 476}%
\special{ip}%
% BOX 2 5 0 0
% 2 1871 441 1303 425
% 
\special{pn 8}%
\special{sh 0.600}%
\special{pa 1872 442}%
\special{pa 1304 442}%
\special{pa 1304 426}%
\special{pa 1872 426}%
\special{pa 1872 442}%
\special{ip}%
% BOX 2 5 1 0
% 2 2300 706 997 1102
% 
\special{pn 8}%
\special{sh 0.300}%
\special{pa 2300 706}%
\special{pa 998 706}%
\special{pa 998 1102}%
\special{pa 2300 1102}%
\special{pa 2300 706}%
\special{ip}%
% BOX 2 5 1 0
% 2 1025 485 1889 521
% 
\special{pn 8}%
\special{sh 0.300}%
\special{pa 1026 486}%
\special{pa 1890 486}%
\special{pa 1890 522}%
\special{pa 1026 522}%
\special{pa 1026 486}%
\special{ip}%
% BOX 2 5 1 0
% 2 1894 488 1881 428
% 
\special{pn 8}%
\special{sh 0.300}%
\special{pa 1894 488}%
\special{pa 1882 488}%
\special{pa 1882 428}%
\special{pa 1894 428}%
\special{pa 1894 488}%
\special{ip}%
% BOX 2 5 1 0
% 2 1881 428 1310 435
% 
\special{pn 8}%
\special{sh 0.300}%
\special{pa 1882 428}%
\special{pa 1310 428}%
\special{pa 1310 436}%
\special{pa 1882 436}%
\special{pa 1882 428}%
\special{ip}%
% VECTOR 2 0 3 0
% 4 700 1600 700 200 700 1600 2700 1600
% 
\special{pn 8}%
\special{pa 700 1600}%
\special{pa 700 200}%
\special{fp}%
\special{sh 1}%
\special{pa 700 200}%
\special{pa 680 268}%
\special{pa 700 254}%
\special{pa 720 268}%
\special{pa 700 200}%
\special{fp}%
\special{pa 700 1600}%
\special{pa 2700 1600}%
\special{fp}%
\special{sh 1}%
\special{pa 2700 1600}%
\special{pa 2634 1580}%
\special{pa 2648 1600}%
\special{pa 2634 1620}%
\special{pa 2700 1600}%
\special{fp}%
% STR 2 0 3 0
% 3 2235 1710 2235 1760 2 0
% $\gamma_s^{\max}$
\put(22.3500,-17.6000){\makebox(0,0)[lb]{$\gamma_s^{\max}$}}%
% STR 2 0 3 0
% 3 1030 1710 1030 1760 2 0
% $\gamma_s^{\min}$
\put(10.3000,-17.6000){\makebox(0,0)[lb]{$\gamma_s^{\min}$}}%
% STR 2 0 3 0
% 3 1265 1715 1265 1765 2 0
% $\gamma_{s+1}^{\min}$
\put(12.6500,-17.6500){\makebox(0,0)[lb]{$\gamma_{s+1}^{\min}$}}%
% STR 2 0 3 0
% 3 1765 1715 1765 1765 2 0
% $\gamma_{s+1}^{\max}$
\put(17.6500,-17.6500){\makebox(0,0)[lb]{$\gamma_{s+1}^{\max}$}}%
% LINE 2 2 3 0
% 10 2300 1600 2300 1200 2300 600 2300 200 1100 1600 1100 1200 1100 600 1100 540 1100 470 1100 200
% 
\special{pn 8}%
\special{pa 2300 1600}%
\special{pa 2300 1200}%
\special{dt 0.045}%
\special{pa 2300 600}%
\special{pa 2300 200}%
\special{dt 0.045}%
\special{pa 1100 1600}%
\special{pa 1100 1200}%
\special{dt 0.045}%
\special{pa 1100 600}%
\special{pa 1100 540}%
\special{dt 0.045}%
\special{pa 1100 470}%
\special{pa 1100 200}%
\special{dt 0.045}%
% LINE 2 2 3 0
% 8 1840 1600 1840 1200 1840 600 1840 540 1840 470 1840 440 1840 420 1840 200
% 
\special{pn 8}%
\special{pa 1840 1600}%
\special{pa 1840 1200}%
\special{dt 0.045}%
\special{pa 1840 600}%
\special{pa 1840 540}%
\special{dt 0.045}%
\special{pa 1840 470}%
\special{pa 1840 440}%
\special{dt 0.045}%
\special{pa 1840 420}%
\special{pa 1840 200}%
\special{dt 0.045}%
% LINE 2 2 3 0
% 6 1300 1600 1300 1200 1300 600 1300 540 1300 470 1300 200
% 
\special{pn 8}%
\special{pa 1300 1600}%
\special{pa 1300 1200}%
\special{dt 0.045}%
\special{pa 1300 600}%
\special{pa 1300 540}%
\special{dt 0.045}%
\special{pa 1300 470}%
\special{pa 1300 200}%
\special{dt 0.045}%
% STR 2 0 3 0
% 3 1400 895 1400 945 2 0
% A copy of $P_{s+1}$
\put(14.0000,-9.4500){\makebox(0,0)[lb]{A copy of $P_{s+1}$}}%
% LINE 2 2 3 0
% 4 700 1200 980 1200 975 600 700 600
% 
\special{pn 8}%
\special{pa 700 1200}%
\special{pa 980 1200}%
\special{dt 0.045}%
\special{pa 976 600}%
\special{pa 700 600}%
\special{dt 0.045}%
% STR 2 0 3 0
% 3 160 1210 160 1260 2 0
% $y+zl_{s+1}^-$
\put(1.6000,-12.6000){\makebox(0,0)[lb]{$y+zl_{s+1}^-$}}%
% STR 2 0 3 0
% 3 160 620 160 670 2 0
% $y+zl_{s+1}^+$
\put(1.6000,-6.7000){\makebox(0,0)[lb]{$y+zl_{s+1}^+$}}%
% BOX 2 5 1 0
% 2 1030 710 997 492
% 
\special{pn 8}%
\special{sh 0.300}%
\special{pa 1030 710}%
\special{pa 998 710}%
\special{pa 998 492}%
\special{pa 1030 492}%
\special{pa 1030 710}%
\special{ip}%
\end{picture}%
  \end{center}
 \vspace{-0.5em}
\caption{$Z_{s+1}(s-1,5)\cup\bigcup\mathcal{Z}_{s+1}(s+1)$.}
  \label{fig:mthm5}
 \end{minipage}
\end{figure}

\begin{sublem}\label{sublem:3}
$Z_t(s,i)\subseteq Z_v(s,i)$ for each $t\geq v\geq s+1$ and $i\leq 5$.
%%%%
\end{sublem}

\begin{sublem}\label{sublem:4}
For any $t\geq s+1$,
\[Z_{s,t}^{\rm st}\touch{[\leftarrow]}Z_t(s,0)\touch{[\uparrow]}Z_t(s,1)\touch{[\rightarrow]}Z_t(s,2)\touch{[\rightarrow]}Z_t(s,3)\touch{[\uparrow]}Z_t(s,4)\touch{[\leftarrow]}Z_t(s,5).\]
\end{sublem}

\begin{proof}\upshape
It follows straightforwardly from the definition of these blocks $Z_t(s,i)$, and Sublemma \ref{sublem:b1} and \ref{sublem:b2}.
%%%%
\end{proof}

\begin{sublem}\label{sublem:5}
$\bigcup_{2\leq i\leq 6}Z_t(s,i)\subseteq Z_s^{\rm st}\cap [\gamma_s^{\min},\gamma_s^{\max}]\times (y+zr_s,y+zr_s^+]$.
Hence, $\left(\bigcup_{2\leq i\leq 6}Z_t(s,i)\right)\cap Z^{\rm st}_{s,s+1}=\emptyset$
%%%%
\end{sublem}

Consequently, we can show the following extension property.

\begin{sublem}\label{sublem:6}
Assume that we have a computable function $f_s:\mathbb{R}^2\to\mathbb{R}^2$ such that $f_s\res\bigcup\bigcup_{u\leq s}\mathcal{Z}_{t}(u)$ is a computable homeomorphism between $\bigcup\bigcup_{u\leq s}\mathcal{Z}_{t}(u)$ and $P_t\times[1/(s+2),1]$ for any $t\geq s$.
Then we can effectively find a computable function $f_{s+1}:\mathbb{R}^2\to\mathbb{R}^2$ extending $f_s\res\bigcup\bigcup_{u\leq s}\mathcal{Z}_{s+1}(u)$ such that $f_{s+1}\res\bigcup\bigcup_{u\leq s+1}\mathcal{Z}_{t}(u)$ is a computable homeomorphism between $\bigcup\bigcup_{u\leq s+1}\mathcal{Z}_{t}(u)$ and $P_t\times[1/(s+3),1]$ for any $t\geq s+1$.
\end{sublem}

\begin{proof}\upshape
By Sublemma \ref{sublem:2}, \ref{sublem:4}, and \ref{sublem:5}.
%%%%
\end{proof}

By Sublemma \ref{sublem:3} and \ref{sublem:6}, induction hypothesis (IH1) and (IH2) are satisfied.
Since $Z_{s+1}^{\rm end}\cup\bigcup\mathcal{Z}_{s+1}(s+1)\subseteq Z_s^{\rm st}\cup Z_s^{\rm end}$ by Sublemma \ref{sublem:1} and \ref{sublem:5}, and $\bigcup\mathcal{Z}_{s+1}(u)\subseteq\bigcup\mathcal{Z}_s(u)$ for each $u\leq s$, by induction hypothesis (IH1), we have the following:
\[Q_{s+1}=Z_{s+1}^{\rm end}\cup\bigcup\bigcup_{u\leq s+1}\mathcal{Z}_{s+1}(u)\subseteq Z_s^{\rm st}\cup Z_s^{\rm end}\cup\bigcup\bigcup_{u\leq s}\mathcal{Z}_{s}(u)\subseteq Q_s.\]

\noindent
{\bf Injured Case.}
Secondly we consider the case that our construction {\em is injured}.
This means that $[\gamma_{s+1}^{\min},\gamma_{s+1}^{\max}]\not\subseteq[\gamma_s^{\min},\gamma_s^{\max}]$.
In this case, indeed, we have $[\gamma_{s+1}^{\min},\gamma_{s+1}^{\max}]\cap[\gamma_s^{\min},\gamma_s^{\max}]=\emptyset$.
Fix the greatest stage $p\leq s$ such that $[\gamma_{s+1}^{\min},\gamma_{s+1}^{\max}]\subseteq[\gamma_p^{\min},\gamma_p^{\max}]$ occurs.
We again, inside the end box $Z_s^{\rm end}$ at stage $s$, define corner blocks $Z_t(s,0)$ and $Z_t(s,1)$ as non-injuring stage, whereas the construction of $Z_t(s,i)$ for $i\geq 2$ differs from non-injuring stage.
The end box of our construction at stage $s+1$ will turn back along all blocks belonging $\mathcal{Z}_s(u)$ for $p<u\leq s$ in the reverse ordering of $\prec$.
Let $\{Z_i:i<k_{s+1}\}$ be an enumeration of all blocks in $\mathcal{Z}_s(u)$ for $p<u\leq s$, under the reverse ordering of $\prec$.
In other words, $Z_i$ is the successor block of $Z_{i+1}$ under $\touch{}$, for each $i<k_{s+1}-1$.
There are two kind of blocks; one is {\em a straight block}, and another is {\em a corner block}.
We will define blocks $Z_t(s,i,j)$ for $i<k_{s+1}$ and $j<3$.
Now we check the direction $\lrangle{\delta_i,\varepsilon_i}$ of $Z_i$.
Here, we may consistently assume that the condition $Z_0\touch{[\leftarrow]}$ holds.

\medskip

\noindent
{\bf Subcase 1.}
If $\delta_i(0)=\varepsilon_i(0)$ then $Z_i$ is a straight block.
In this case, we only construct $Z_t(s,i,0)$.
Since $Z_i$ is straight, there are $y_i,z_i,\alpha,\beta\in\mathbb{Q}$ and $u\leq s$ such that, for $B_i(0)=[\alpha,\beta]$ and $B_i(1)=[y_i+z_il_u^-,y_i+z_ir_u^+]$ such that ${\rm Box}(Z_i)=B_i(\delta_i(0))\times B_i(1-\delta_i(0))$.
If $\delta_i(1)=0$, then set $y_i^{\star}=N_{0,s}(y_i,z_i;l_s^-,l_s^*)$ and $z_i^{\star}=N_{1,s}(y_i,z_i;l_s^-,l_s^*)$.
If $\delta_i(1)=1$, then set $y_i^{\star}=N_{0,s}(y_i,z_i;r_s^-,r_s^+)$ and $z_i^{\star}=N_{1,s}(y_i,z_i;r_s^*,r_s^+)$.
Then, we define $Z_t(s,i,0)$ as the following straight block:
\begin{align*}
B_i^{\star}(0)=B_i(0);\quad B_i^{\star}(1)=[y_i^{\star}+z_i^{\star}l_s^-,y_i^{\star}+z_i^{\star}r_s^+];\\
Z_t(s,i,0)=[\delta_i(0)]^s_t(B_i^\star(\delta_i(0))\times B_i^\star(1-\delta_i(0))).
\end{align*}
Here, ${\rm Box}(Z_t(s,i,0))$ is defined by $B_i^\star(\delta_i(0))\times B_i^\star(1-\delta_i(0))$.
\begin{figure}[t]\centering
 \begin{minipage}{0.48\hsize}
  \begin{center}
%WinTpicVersion3.08
\unitlength 0.1in
\begin{picture}( 21.7000, 15.7000)(  4.3000,-20.0000)
% POLYGON 2 5 0 0
% 6 1200 1100 1200 1200 2400 1200 2400 1100 2400 1100 1200 1100
% 
\special{pn 8}%
\special{sh 0.600}%
\special{pa 1200 1100}%
\special{pa 1200 1200}%
\special{pa 2400 1200}%
\special{pa 2400 1100}%
\special{pa 2400 1100}%
\special{pa 1200 1100}%
\special{ip}%
% POLYGON 2 5 0 0
% 6 1200 1400 1200 1500 2400 1500 2400 1400 2400 1400 1200 1400
% 
\special{pn 8}%
\special{sh 0.600}%
\special{pa 1200 1400}%
\special{pa 1200 1500}%
\special{pa 2400 1500}%
\special{pa 2400 1400}%
\special{pa 2400 1400}%
\special{pa 1200 1400}%
\special{ip}%
% VECTOR 2 0 3 0
% 4 1000 2000 1000 600 1000 2000 2600 2000
% 
\special{pn 8}%
\special{pa 1000 2000}%
\special{pa 1000 600}%
\special{fp}%
\special{sh 1}%
\special{pa 1000 600}%
\special{pa 980 668}%
\special{pa 1000 654}%
\special{pa 1020 668}%
\special{pa 1000 600}%
\special{fp}%
\special{pa 1000 2000}%
\special{pa 2600 2000}%
\special{fp}%
\special{sh 1}%
\special{pa 2600 2000}%
\special{pa 2534 1980}%
\special{pa 2548 2000}%
\special{pa 2534 2020}%
\special{pa 2600 2000}%
\special{fp}%
% STR 2 0 3 0
% 3 1530 500 1530 600 2 0
% $\touch{[\rightarrow]}Z_i\touch{[\rightarrow]}$
\put(15.3000,-6.0000){\makebox(0,0)[lb]{$\touch{[\rightarrow]}Z_i\touch{[\rightarrow]}$}}%
% LINE 2 2 3 0
% 8 1200 2000 1200 1800 2400 2000 2400 1800 1200 1800 1000 1800 1200 800 1000 800
% 
\special{pn 8}%
\special{pa 1200 2000}%
\special{pa 1200 1800}%
\special{dt 0.045}%
\special{pa 2400 2000}%
\special{pa 2400 1800}%
\special{dt 0.045}%
\special{pa 1200 1800}%
\special{pa 1000 1800}%
\special{dt 0.045}%
\special{pa 1200 800}%
\special{pa 1000 800}%
\special{dt 0.045}%
% STR 2 0 3 0
% 3 2330 2040 2330 2140 2 0
% $\beta$
\put(23.3000,-21.4000){\makebox(0,0)[lb]{$\beta$}}%
% STR 2 0 3 0
% 3 1160 2010 1160 2110 2 0
% $\alpha$
\put(11.6000,-21.1000){\makebox(0,0)[lb]{$\alpha$}}%
% STR 2 0 3 0
% 3 450 1760 450 1860 2 0
% $y_i+z_il_u^-$
\put(4.5000,-18.6000){\makebox(0,0)[lb]{$y_i+z_il_u^-$}}%
% STR 2 0 3 0
% 3 430 770 430 870 2 0
% $y_i+z_ir_u^+$
\put(4.3000,-8.7000){\makebox(0,0)[lb]{$y_i+z_ir_u^+$}}%
% BOX 2 5 1 0
% 2 1200 1200 2400 1400
% 
\special{pn 8}%
\special{sh 0.300}%
\special{pa 1200 1200}%
\special{pa 2400 1200}%
\special{pa 2400 1400}%
\special{pa 1200 1400}%
\special{pa 1200 1200}%
\special{ip}%
% BOX 1 0 3 0
% 2 1200 800 2400 1800
% 
\special{pn 13}%
\special{pa 1200 800}%
\special{pa 2400 800}%
\special{pa 2400 1800}%
\special{pa 1200 1800}%
\special{pa 1200 800}%
\special{fp}%
% LINE 2 2 3 0
% 4 1200 1100 1000 1100 1200 1500 1000 1500
% 
\special{pn 8}%
\special{pa 1200 1100}%
\special{pa 1000 1100}%
\special{dt 0.045}%
\special{pa 1200 1500}%
\special{pa 1000 1500}%
\special{dt 0.045}%
% STR 2 0 3 0
% 3 440 1060 440 1160 2 0
% $y_i+z_ir_s^+$
\put(4.4000,-11.6000){\makebox(0,0)[lb]{$y_i+z_ir_s^+$}}%
% STR 2 0 3 0
% 3 450 1450 450 1550 2 0
% $y_i+z_il_s^-$
\put(4.5000,-15.5000){\makebox(0,0)[lb]{$y_i+z_il_s^-$}}%
\end{picture}%
  \end{center}
 \vspace{-0.5em}
\caption{The block $Z_i$.}
  \label{fig:inj1}
 \end{minipage}
 \begin{minipage}{0.48\hsize}
  \begin{center}
%WinTpicVersion3.08
\unitlength 0.1in
\begin{picture}( 16.0000, 10.6000)( 12.0000,-16.3000)
% POLYGON 2 5 0 0
% 6 1200 800 1200 1000 2400 1000 2400 800 2400 800 1200 800
% 
\special{pn 8}%
\special{sh 0.600}%
\special{pa 1200 800}%
\special{pa 1200 1000}%
\special{pa 2400 1000}%
\special{pa 2400 800}%
\special{pa 2400 800}%
\special{pa 1200 800}%
\special{ip}%
% POLYGON 2 5 0 0
% 6 1200 1400 1200 1600 2400 1600 2400 1400 2400 1400 1200 1400
% 
\special{pn 8}%
\special{sh 0.600}%
\special{pa 1200 1400}%
\special{pa 1200 1600}%
\special{pa 2400 1600}%
\special{pa 2400 1400}%
\special{pa 2400 1400}%
\special{pa 1200 1400}%
\special{ip}%
% STR 2 0 3 0
% 3 1600 1700 1600 1800 2 0
% $\touch{[\rightarrow]}Z_i\touch{[\rightarrow]}$
\put(16.0000,-18.0000){\makebox(0,0)[lb]{$\touch{[\rightarrow]}Z_i\touch{[\rightarrow]}$}}%
% BOX 2 5 1 0
% 2 1200 1000 2400 1400
% 
\special{pn 8}%
\special{sh 0.300}%
\special{pa 1200 1000}%
\special{pa 2400 1000}%
\special{pa 2400 1400}%
\special{pa 1200 1400}%
\special{pa 1200 1000}%
\special{ip}%
% BOX 2 5 1 0
% 2 1200 830 2400 880
% 
\special{pn 8}%
\special{sh 0.300}%
\special{pa 1200 830}%
\special{pa 2400 830}%
\special{pa 2400 880}%
\special{pa 1200 880}%
\special{pa 1200 830}%
\special{ip}%
% STR 2 0 3 0
% 3 1720 1160 1720 1260 2 0
% $Z_i$
\put(17.2000,-12.6000){\makebox(0,0)[lb]{$Z_i$}}%
% VECTOR 2 0 3 0
% 2 2800 1000 2400 850
% 
\special{pn 8}%
\special{pa 2800 1000}%
\special{pa 2400 850}%
\special{fp}%
\special{sh 1}%
\special{pa 2400 850}%
\special{pa 2456 892}%
\special{pa 2450 870}%
\special{pa 2470 856}%
\special{pa 2400 850}%
\special{fp}%
% STR 2 0 3 0
% 3 2740 1090 2740 1190 2 0
% $Z_s(s,i,0)$
\put(27.4000,-11.9000){\makebox(0,0)[lb]{$Z_s(s,i,0)$}}%
% STR 2 0 3 0
% 3 1390 640 1390 740 2 0
% $\touch{[\leftarrow]}Z_t(s,i,0)\touch{[\leftarrow]}$
\put(13.9000,-7.4000){\makebox(0,0)[lb]{$\touch{[\leftarrow]}Z_t(s,i,0)\touch{[\leftarrow]}$}}%
\end{picture}%
  \end{center}
 \vspace{-0.5em}
\caption{The block $Z_t(s,i,0)$.}
  \label{fig:inj2}
 \end{minipage}
\end{figure}

\begin{sublem}
$Z_t(s,i,0)\subseteq Z_i$.
\end{sublem}

\begin{proof}\upshape
By our definition of $N_{0,s}$ and $N_{1,s}$, we have $B_i^{\star}(1)=[y_i+z_il_s^-,y_i+z_il_s^*]$ or $B_i^{\star}(1)=[y_i+z_ir_s^*,y_i+z_ir_s^+]$.
%%%%
\end{proof}

\noindent
{\bf Subcase 2.}
If $\delta_i(0)\not=\delta_i(2)$ then $Z_i$ is a corner block.
We will construct three blocks; one corner block $Z_t(s,i,0)$, and two straight blocks $Z_t(s,i,1)$ and $Z_t(s,i,2)$.
%Then the corner block $Z_i$ has a box code $I_i=\lrangle{a,b,c,d,e,f,u}$.
%We assume $f=1$.
%Thickness for $Z_i$ is computed from the code $I_i$; put $z_i=d/(r_u^+-l_u^-)$, and $\zeta_i=c/(r_u^+-l_u^-)$.
%{\em The basic position} is also computed as $x_i=a-\zeta_i l_u^-$ and $y_i=b-z_i l_u^-$.
We may assume that $Z_i$ is of the following form:
\begin{align*}
Z_i=[e]^u_s&([x_i+\zeta_il_u^-,x_i+\zeta_ir_u^+]\times[y_i+z_il_u^-,y_i+z_ir_u^+]),\\
\text{or }Z_i=[e]^u_s&([x_i+\zeta_il_u^-,x_i+\zeta_ir_u^+]\times[y_i+z_il_u^-,y_i+z_ir_u^+])\\
&\cap([x_i+\zeta_il_u^-,x_i+\zeta_ir_u^+]\times[y_i+z_il_u^-,y_i+z_ir_u^*])
\end{align*}
Set $z=0$ if the former case occurs; otherwise, set $z=1$.
Let $\{p_n:n<6\}$ be an enumeration of $\{l_u^-,l_s^-,l_s^*,r_s^*,r_s^+,r_u^+\}$ in increasing order, and let $p_6$ be $r_u^*$.
First we compute the value ${\tt rot}=2|\varepsilon_i(0)-|\delta_i(1)-\varepsilon_i(1)||+1$.
Note that ${\tt rot}\in\{1,3\}$, and, if $Z_i$ rotates clockwise then ${\tt rot}=1$; and if $Z_i$ rotates counterclockwise then ${\tt rot}=3$.
If $\touch{[\rightarrow]}Z_i$ or $Z_i\touch{[\rightarrow]}$, then put $D(0)=1$; otherwise put $D(0)=3$.
If $\touch{[\downarrow]}Z_i$ or $Z_i\touch{[\downarrow]}$, then put $D(1)=1$; otherwise put $D(1)=3$.
If $\touch{[\rightarrow]}Z_i$ or $Z_i\touch{[\leftarrow]}$, then put $E(0)=0$; otherwise put $E(0)=5-{\rm rot}$.
If $\touch{[\uparrow]}Z_i$ or $Z_i\touch{[\downarrow]}$, then put $E(1)=0$; otherwise put $E(1)=5-{\rm rot}$.
Then we now define $Z_t(s,i,j)$ for $j<3$ as follows:
\begin{align*}
{\rm Box}(s,i,0)&=[x_i+\zeta_ip_{D(0)},x_i+\zeta_ip_{D(0)+2}]\times[y_i+z_ip_{D(1)},y_i+z_ip_{D(1)+2}],\\
{\rm Box}(s,i,1)&=[x_i+\zeta_ip_{E(0)},x_i+\zeta_ip_{E(0)+{\tt rot}}]\times[y_i+z_ip_{D(1)},y_i+z_ip_{D(1)+2}],\\
{\rm Box}(s,i,2)&=[x_i+\zeta_ip_{D(0)},x_i+\zeta_ip_{D(0)+2}]\times[y_i+z_ip_{E(1)},y_i+z_ip_{E(1)+{\tt rot}+z}],\\
Z_t(s,i,0)&=[e]^s_t({\rm Box}(s,i,0)),\\
Z_t(s,i,1)&=[-]^s_t({\rm Box}(s,i,1)),\\
Z_t(s,i,2)&=[\;\mid\;]^s_t({\rm Box}(s,i,2)).
\end{align*}

Intuitively, $D(0)=1$ (resp.\ $D(0)=3$) indicates that $Z_t(s,i,0)$ passes the west (resp.\ the east) of $Z_i$; $D(1)=1$ (resp.\ $D(1)=3$) indicates that $Z_t(s,i,0)$ passes the south (resp.\ the north) of $Z_i$; $E(0)=0$ (resp.\ $E(0)=5-{\tt rot}$) indicates that $Z_t(s,i,1)$ passes the west (resp.\ the east) border of the bounding box of $Z_i$; and $E(1)=0$ (resp.\ $E(1)=5-{\tt rot}$) indicates that $Z_t(s,i,2)$ passes the south (resp.\ the north) border of the bounding box of $Z_i$.
Note that the corner block $Z_t(s,i,0)$ leaves $Z_i$ on his right, and $Z_t(s,i,0)$ has the reverse direction to $Z_i$.

\begin{figure}[t]\centering
 \begin{minipage}{0.48\hsize}
  \begin{center}
%WinTpicVersion3.08
\unitlength 0.1in
\begin{picture}( 19.1500, 17.4500)(  4.0000,-20.5000)
% BOX 1 0 3 0
% 2 675 595 2075 1995
% 
\special{pn 13}%
\special{pa 676 596}%
\special{pa 2076 596}%
\special{pa 2076 1996}%
\special{pa 676 1996}%
\special{pa 676 596}%
\special{fp}%
% BOX 1 5 1 0
% 2 1175 595 1575 1395
% 
\special{pn 13}%
\special{sh 0.300}%
\special{pa 1176 596}%
\special{pa 1576 596}%
\special{pa 1576 1396}%
\special{pa 1176 1396}%
\special{pa 1176 596}%
\special{ip}%
% BOX 1 5 1 0
% 2 1575 1195 2075 1395
% 
\special{pn 13}%
\special{sh 0.300}%
\special{pa 1576 1196}%
\special{pa 2076 1196}%
\special{pa 2076 1396}%
\special{pa 1576 1396}%
\special{pa 1576 1196}%
\special{ip}%
% BOX 1 5 0 0
% 2 975 595 1175 1495
% 
\special{pn 13}%
\special{sh 0.600}%
\special{pa 976 596}%
\special{pa 1176 596}%
\special{pa 1176 1496}%
\special{pa 976 1496}%
\special{pa 976 596}%
\special{ip}%
% BOX 1 5 0 0
% 2 1575 595 1775 1195
% 
\special{pn 13}%
\special{sh 0.600}%
\special{pa 1576 596}%
\special{pa 1776 596}%
\special{pa 1776 1196}%
\special{pa 1576 1196}%
\special{pa 1576 596}%
\special{ip}%
% BOX 1 5 0 0
% 2 1775 1095 2075 1195
% 
\special{pn 13}%
\special{sh 0.600}%
\special{pa 1776 1096}%
\special{pa 2076 1096}%
\special{pa 2076 1196}%
\special{pa 1776 1196}%
\special{pa 1776 1096}%
\special{ip}%
% BOX 1 5 0 0
% 2 1175 1395 2075 1495
% 
\special{pn 13}%
\special{sh 0.600}%
\special{pa 1176 1396}%
\special{pa 2076 1396}%
\special{pa 2076 1496}%
\special{pa 1176 1496}%
\special{pa 1176 1396}%
\special{ip}%
% STR 2 0 3 0
% 3 1315 1200 1315 1250 2 0
% $Z_i$
\put(13.1500,-12.5000){\makebox(0,0)[lb]{$Z_i$}}%
% LINE 2 1 3 0
% 2 1075 1695 775 1695
% 
\special{pn 8}%
\special{pa 1076 1696}%
\special{pa 776 1696}%
\special{da 0.070}%
% VECTOR 2 1 3 0
% 2 775 1695 775 1295
% 
\special{pn 8}%
\special{pa 776 1696}%
\special{pa 776 1296}%
\special{da 0.070}%
\special{sh 1}%
\special{pa 776 1296}%
\special{pa 756 1362}%
\special{pa 776 1348}%
\special{pa 796 1362}%
\special{pa 776 1296}%
\special{fp}%
% BOX 2 5 1 0
% 2 1675 595 1725 1145
% 
\special{pn 8}%
\special{sh 0.300}%
\special{pa 1676 596}%
\special{pa 1726 596}%
\special{pa 1726 1146}%
\special{pa 1676 1146}%
\special{pa 1676 596}%
\special{ip}%
% BOX 2 5 1 0
% 2 1725 1145 2075 1120
% 
\special{pn 8}%
\special{sh 0.300}%
\special{pa 1726 1146}%
\special{pa 2076 1146}%
\special{pa 2076 1120}%
\special{pa 1726 1120}%
\special{pa 1726 1146}%
\special{ip}%
% STR 2 0 3 0
% 3 1205 375 1205 475 2 0
% $Z_s(s,i,2)$
\put(12.0500,-4.7500){\makebox(0,0)[lb]{$Z_s(s,i,2)$}}%
% BOX 2 0 3 0
% 2 1775 1095 1625 595
% 
\special{pn 8}%
\special{pa 1776 1096}%
\special{pa 1626 1096}%
\special{pa 1626 596}%
\special{pa 1776 596}%
\special{pa 1776 1096}%
\special{fp}%
% BOX 2 0 3 0
% 2 1625 1095 1775 1170
% 
\special{pn 8}%
\special{pa 1626 1096}%
\special{pa 1776 1096}%
\special{pa 1776 1170}%
\special{pa 1626 1170}%
\special{pa 1626 1096}%
\special{fp}%
% BOX 2 0 3 0
% 2 1775 1095 2075 1170
% 
\special{pn 8}%
\special{pa 1776 1096}%
\special{pa 2076 1096}%
\special{pa 2076 1170}%
\special{pa 1776 1170}%
\special{pa 1776 1096}%
\special{fp}%
% VECTOR 2 0 3 0
% 2 2275 1195 2095 1125
% 
\special{pn 8}%
\special{pa 2276 1196}%
\special{pa 2096 1126}%
\special{fp}%
\special{sh 1}%
\special{pa 2096 1126}%
\special{pa 2150 1168}%
\special{pa 2146 1144}%
\special{pa 2164 1132}%
\special{pa 2096 1126}%
\special{fp}%
% STR 2 0 3 0
% 3 2255 1235 2255 1335 2 0
% $Z_s(s,i,1)$
\put(22.5500,-13.3500){\makebox(0,0)[lb]{$Z_s(s,i,1)$}}%
% VECTOR 2 1 3 0
% 2 1875 995 1795 1085
% 
\special{pn 8}%
\special{pa 1876 996}%
\special{pa 1796 1086}%
\special{da 0.070}%
\special{sh 1}%
\special{pa 1796 1086}%
\special{pa 1854 1048}%
\special{pa 1830 1046}%
\special{pa 1824 1022}%
\special{pa 1796 1086}%
\special{fp}%
% STR 2 0 3 0
% 3 2315 735 2315 835 2 0
% $Z_s(s,i,0)$
\put(23.1500,-8.3500){\makebox(0,0)[lb]{$Z_s(s,i,0)$}}%
% LINE 2 0 3 0
% 2 2275 795 1875 995
% 
\special{pn 8}%
\special{pa 2276 796}%
\special{pa 1876 996}%
\special{fp}%
% LINE 2 2 3 0
% 2 1875 595 1875 995
% 
\special{pn 8}%
\special{pa 1876 596}%
\special{pa 1876 996}%
\special{dt 0.045}%
% VECTOR 2 2 3 0
% 2 1875 995 2075 995
% 
\special{pn 8}%
\special{pa 1876 996}%
\special{pa 2076 996}%
\special{dt 0.045}%
\special{sh 1}%
\special{pa 2076 996}%
\special{pa 2008 976}%
\special{pa 2022 996}%
\special{pa 2008 1016}%
\special{pa 2076 996}%
\special{fp}%
% STR 2 0 3 0
% 3 400 2170 400 2220 2 0
% $Z_s(s,i,2)\touch{\downarrow}Z_s(s,i,0)\touch{\rightarrow}Z_s(s,i,1)$
\put(4.0000,-22.2000){\makebox(0,0)[lb]{$Z_s(s,i,2)\touch{\downarrow}Z_s(s,i,0)\touch{\rightarrow}Z_s(s,i,1)$}}%
% VECTOR 2 0 3 0
% 2 1475 495 1700 580
% 
\special{pn 8}%
\special{pa 1476 496}%
\special{pa 1700 580}%
\special{fp}%
\special{sh 1}%
\special{pa 1700 580}%
\special{pa 1646 538}%
\special{pa 1650 562}%
\special{pa 1632 576}%
\special{pa 1700 580}%
\special{fp}%
\end{picture}%
  \end{center}
 \vspace{-0.5em}
\caption{${\rm rot}=1$.}
  \label{fig:inj4}
 \end{minipage}
 \begin{minipage}{0.48\hsize}
  \begin{center}
%WinTpicVersion3.08
\unitlength 0.1in
\begin{picture}( 17.0000, 17.0500)(  4.0000,-16.5000)
% BOX 1 0 3 0
% 2 600 215 2000 1615
% 
\special{pn 13}%
\special{pa 600 216}%
\special{pa 2000 216}%
\special{pa 2000 1616}%
\special{pa 600 1616}%
\special{pa 600 216}%
\special{fp}%
% BOX 1 5 1 0
% 2 1100 215 1500 1015
% 
\special{pn 13}%
\special{sh 0.300}%
\special{pa 1100 216}%
\special{pa 1500 216}%
\special{pa 1500 1016}%
\special{pa 1100 1016}%
\special{pa 1100 216}%
\special{ip}%
% BOX 1 5 1 0
% 2 1500 815 2000 1015
% 
\special{pn 13}%
\special{sh 0.300}%
\special{pa 1500 816}%
\special{pa 2000 816}%
\special{pa 2000 1016}%
\special{pa 1500 1016}%
\special{pa 1500 816}%
\special{ip}%
% BOX 1 5 0 0
% 2 900 215 1100 1115
% 
\special{pn 13}%
\special{sh 0.600}%
\special{pa 900 216}%
\special{pa 1100 216}%
\special{pa 1100 1116}%
\special{pa 900 1116}%
\special{pa 900 216}%
\special{ip}%
% BOX 1 5 0 0
% 2 1500 215 1700 815
% 
\special{pn 13}%
\special{sh 0.600}%
\special{pa 1500 216}%
\special{pa 1700 216}%
\special{pa 1700 816}%
\special{pa 1500 816}%
\special{pa 1500 216}%
\special{ip}%
% BOX 1 5 0 0
% 2 1700 715 2000 815
% 
\special{pn 13}%
\special{sh 0.600}%
\special{pa 1700 716}%
\special{pa 2000 716}%
\special{pa 2000 816}%
\special{pa 1700 816}%
\special{pa 1700 716}%
\special{ip}%
% BOX 1 5 0 0
% 2 1100 1015 2000 1115
% 
\special{pn 13}%
\special{sh 0.600}%
\special{pa 1100 1016}%
\special{pa 2000 1016}%
\special{pa 2000 1116}%
\special{pa 1100 1116}%
\special{pa 1100 1016}%
\special{ip}%
% STR 2 0 3 0
% 3 1240 820 1240 870 2 0
% $Z_i$
\put(12.4000,-8.7000){\makebox(0,0)[lb]{$Z_i$}}%
% LINE 2 1 3 0
% 2 1800 215 1800 615
% 
\special{pn 8}%
\special{pa 1800 216}%
\special{pa 1800 616}%
\special{da 0.070}%
% VECTOR 2 1 3 0
% 2 1800 615 2000 615
% 
\special{pn 8}%
\special{pa 1800 616}%
\special{pa 2000 616}%
\special{da 0.070}%
\special{sh 1}%
\special{pa 2000 616}%
\special{pa 1934 596}%
\special{pa 1948 616}%
\special{pa 1934 636}%
\special{pa 2000 616}%
\special{fp}%
% BOX 1 5 1 0
% 2 950 215 1000 1090
% 
\special{pn 13}%
\special{sh 0.300}%
\special{pa 950 216}%
\special{pa 1000 216}%
\special{pa 1000 1090}%
\special{pa 950 1090}%
\special{pa 950 216}%
\special{ip}%
% BOX 1 5 1 0
% 2 1000 1090 2000 1065
% 
\special{pn 13}%
\special{sh 0.300}%
\special{pa 1000 1090}%
\special{pa 2000 1090}%
\special{pa 2000 1066}%
\special{pa 1000 1066}%
\special{pa 1000 1090}%
\special{ip}%
% BOX 2 0 3 0
% 2 900 215 1050 1040
% 
\special{pn 8}%
\special{pa 900 216}%
\special{pa 1050 216}%
\special{pa 1050 1040}%
\special{pa 900 1040}%
\special{pa 900 216}%
\special{fp}%
% BOX 2 0 3 0
% 2 900 1115 1050 1040
% 
\special{pn 8}%
\special{pa 900 1116}%
\special{pa 1050 1116}%
\special{pa 1050 1040}%
\special{pa 900 1040}%
\special{pa 900 1116}%
\special{fp}%
% BOX 2 0 3 0
% 2 1050 1115 2000 1040
% 
\special{pn 8}%
\special{pa 1050 1116}%
\special{pa 2000 1116}%
\special{pa 2000 1040}%
\special{pa 1050 1040}%
\special{pa 1050 1116}%
\special{fp}%
% LINE 2 2 3 0
% 2 1000 1215 800 1215
% 
\special{pn 8}%
\special{pa 1000 1216}%
\special{pa 800 1216}%
\special{dt 0.045}%
% VECTOR 2 2 3 0
% 2 800 1215 800 915
% 
\special{pn 8}%
\special{pa 800 1216}%
\special{pa 800 916}%
\special{dt 0.045}%
\special{sh 1}%
\special{pa 800 916}%
\special{pa 780 982}%
\special{pa 800 968}%
\special{pa 820 982}%
\special{pa 800 916}%
\special{fp}%
% VECTOR 2 0 3 0
% 2 2100 1215 2005 1080
% 
\special{pn 8}%
\special{pa 2100 1216}%
\special{pa 2006 1080}%
\special{fp}%
\special{sh 1}%
\special{pa 2006 1080}%
\special{pa 2028 1146}%
\special{pa 2036 1124}%
\special{pa 2060 1124}%
\special{pa 2006 1080}%
\special{fp}%
% STR 2 0 3 0
% 3 2070 1275 2070 1325 2 0
% $Z_s(s,i,1)$
\put(20.7000,-13.2500){\makebox(0,0)[lb]{$Z_s(s,i,1)$}}%
% VECTOR 2 0 3 0
% 2 800 115 965 200
% 
\special{pn 8}%
\special{pa 800 116}%
\special{pa 966 200}%
\special{fp}%
\special{sh 1}%
\special{pa 966 200}%
\special{pa 916 152}%
\special{pa 918 176}%
\special{pa 898 188}%
\special{pa 966 200}%
\special{fp}%
% STR 2 0 3 0
% 3 600 65 600 115 2 0
% $Z_s(s,i,2)$
\put(6.0000,-1.1500){\makebox(0,0)[lb]{$Z_s(s,i,2)$}}%
% VECTOR 2 0 3 0
% 2 1100 1215 1000 1115
% 
\special{pn 8}%
\special{pa 1100 1216}%
\special{pa 1000 1116}%
\special{fp}%
\special{sh 1}%
\special{pa 1000 1116}%
\special{pa 1034 1176}%
\special{pa 1038 1154}%
\special{pa 1062 1148}%
\special{pa 1000 1116}%
\special{fp}%
% STR 2 0 3 0
% 3 1125 1205 1125 1255 2 0
% $Z_s(s,i,0)$
\put(11.2500,-12.5500){\makebox(0,0)[lb]{$Z_s(s,i,0)$}}%
% STR 2 0 3 0
% 3 400 1770 400 1820 2 0
% $Z_s(s,i,1)\touch{\leftarrow}Z_s(s,i,0)\touch{\uparrow}Z_s(s,i,2)$
\put(4.0000,-18.2000){\makebox(0,0)[lb]{$Z_s(s,i,1)\touch{\leftarrow}Z_s(s,i,0)\touch{\uparrow}Z_s(s,i,2)$}}%
\end{picture}%
  \end{center}
 \vspace{-0.5em}
\caption{${\rm rot}=3$.}
  \label{fig:inj5}
 \end{minipage}
\end{figure}

\begin{sublem}
$Z_t(s,i,2-\delta_i(0))\touch{\varepsilon^{\circ}}Z_t(s,i,0)\touch{\delta^{\circ}}Z_t(s,t,1+\delta_i(0))$.
%%%%
\end{sublem}

\begin{sublem}
$Z_t(s,i,j)\subseteq Z_i$.
%%%%
\end{sublem}

For each $i<k_{s+1}$, we have already constructed $\mathcal{Z}_t(s+1;i)=\{Z_t(s,i,j):j<3\}$.
All of these blocks constructed at the current stage are included in $Z_s^{\rm end}\cup\bigcup\bigcup_{p<u\leq s}\mathcal{Z}_s(u)$.
Let $Z^{0}[i]$ (resp.\ $Z^{1}[i]$) be the $\prec$-least (resp. the $\prec$-greatest) element of $\{\lambda t.Z_t(s,i,j):j<3\}$.
It is not hard to see that our construction ensures the following condition.

\begin{sublem}
$Z_t^1[i]\touch{}Z_t^0[i+1]$.
%%%%
\end{sublem}

Thus, $\bigcup_{i<k_{s+1}}\mathcal{Z}_t(s+1;i)$ is computably homeomorphic to $P_t\times[0,1]$, uniformly in $t\geq s+1$.
Therefore, we can connect blocks $Z_s(s,i)$ for $i<k_{s+1}$, and we succeed to return back on the current approximation of the $\prec$-greatest $p$-block $Z_s(p)=Z_{p,s}^{\rm st}\in\mathcal{Z}_s(p)$.
Then we construct blocks $Z_t(s,k)$ for $2\leq k\leq 6$ on the block $Z_s(p)$.
The construction is essentially similar as the non-injuring case.
By induction hypothesis (IH3), we note that $Z_s(p)$ must be of the following form for some $y_p,z_p\in\mathbb{Q}$:
\[Z_s(p)=[-]^p_s([\gamma_p^{\rm min},\gamma_p^{\max}]\times[y_p+z_pl_p^-,y_p+z_pr_p^+]).\]
On $Z_s(p)$, we define {\em a straight block from $\gamma_p^{\min}$ to $\gamma_{s+1}^{\max}$} as follows:
\[Z_t(s,2)=[-]^p_s([\gamma_p^{\rm min},\gamma_{s+1}^{\max}]\times[y_p+z_pr_s^*,y_p+z_pr_s^+]).\]
Here, by our assumption, $\gamma^{\max}_{s+1}<\gamma^{\max}_p$ holds since $\gamma^{\max}_{s+1}\leq\gamma^{\max}_p$.
The blocks $Z_t(s,k)$ for $3\leq k\leq 6$ are defined as in the same method as non-injuring case.
The active block at stage $s+1$ is $Z_{s+1}(s,5)$, and the end box at stage $s+1$ is $Z_{s+1}(s,6)$.
{\em $(s+1)$-blocks at stage $t$} are $Z_t(s,i)$ for $i<6$, and $Z_t(s,i,j)$ for $i<k_{s+1}$ and $j<3$ if it is constructed.
$\mathcal{Z}_t(s+1)$ denotes {\em the collection of $(s+1)$-blocks at stage $t$}.

\begin{sublem}
$Z_{s+1}^{\rm end}\cup\bigcup\mathcal{Z}_{s+1}(s+1)\subseteq Z_s^{\rm end}\cup\bigcup\bigcup_{p\leq u\leq s}\mathcal{Z}_s(u)$.
%%%%
\end{sublem}

Thus we again have the following:
\[Q_{s+1}=Z_{s+1}^{\rm end}\cup\bigcup\bigcup_{u\leq s+1}\mathcal{Z}_{s+1}(u)\subseteq Z_s^{\rm st}\cup Z_s^{\rm end}\cup\bigcup\bigcup_{u\leq s}\mathcal{Z}_{s}(u)\subseteq Q_s.\]

\begin{sublem}\label{sublem:15}
Assume that we have a computable function $f_s:\mathbb{R}^2\to\mathbb{R}^2$ such that $f_s\res\bigcup\bigcup_{u\leq s}\mathcal{Z}_{t}(u)$ is a computable homeomorphism between $\bigcup\bigcup_{u\leq s}\mathcal{Z}_{t}(u)$ and $P_t\times[1/(s+2),1]$ for any $t\geq s$.
Then we can effectively find a computable function $f_{s+1}:\mathbb{R}^2\to\mathbb{R}^2$ extending $f_s\res\bigcup\bigcup_{u\leq s}\mathcal{Z}_{s+1}(u)$ such that $f_{s+1}\res\bigcup\bigcup_{u\leq s+1}\mathcal{Z}_{t}(u)$ is a computable homeomorphism between $\bigcup\bigcup_{u\leq s+1}\mathcal{Z}_{t}(u)$ and $P_t\times[1/(s+3),1]$ for any $t\geq s+1$.
%%%%
\end{sublem}

\begin{figure}[t]\centering
  \begin{center}
%WinTpicVersion3.08
\unitlength 0.1in
\begin{picture}( 32.0000, 21.7500)(  2.0000,-24.0500)
% BOX 2 5 0 0
% 2 200 1600 3400 2200
% 
\special{pn 8}%
\special{sh 0.600}%
\special{pa 200 1600}%
\special{pa 3400 1600}%
\special{pa 3400 2200}%
\special{pa 200 2200}%
\special{pa 200 1600}%
\special{ip}%
% BOX 2 5 1 0
% 2 400 1800 3200 2000
% 
\special{pn 8}%
\special{sh 0.300}%
\special{pa 400 1800}%
\special{pa 3200 1800}%
\special{pa 3200 2000}%
\special{pa 400 2000}%
\special{pa 400 1800}%
\special{ip}%
% BOX 2 5 1 0
% 2 3200 2200 3000 2000
% 
\special{pn 8}%
\special{sh 0.300}%
\special{pa 3200 2200}%
\special{pa 3000 2200}%
\special{pa 3000 2000}%
\special{pa 3200 2000}%
\special{pa 3200 2200}%
\special{ip}%
% BOX 2 5 1 0
% 2 400 1800 600 1600
% 
\special{pn 8}%
\special{sh 0.300}%
\special{pa 400 1800}%
\special{pa 600 1800}%
\special{pa 600 1600}%
\special{pa 400 1600}%
\special{pa 400 1800}%
\special{ip}%
% LINE 2 0 3 0
% 24 490 1595 490 1395 490 1395 1890 1395 1890 1395 1890 1195 1890 1195 690 1195 690 1195 690 995 690 995 1290 995 1290 995 1290 895 1290 895 890 895 890 895 890 795 890 795 1190 795 1190 795 1190 695 1190 695 990 695
% 
\special{pn 8}%
\special{pa 490 1596}%
\special{pa 490 1396}%
\special{fp}%
\special{pa 490 1396}%
\special{pa 1890 1396}%
\special{fp}%
\special{pa 1890 1396}%
\special{pa 1890 1196}%
\special{fp}%
\special{pa 1890 1196}%
\special{pa 690 1196}%
\special{fp}%
\special{pa 690 1196}%
\special{pa 690 996}%
\special{fp}%
\special{pa 690 996}%
\special{pa 1290 996}%
\special{fp}%
\special{pa 1290 996}%
\special{pa 1290 896}%
\special{fp}%
\special{pa 1290 896}%
\special{pa 890 896}%
\special{fp}%
\special{pa 890 896}%
\special{pa 890 796}%
\special{fp}%
\special{pa 890 796}%
\special{pa 1190 796}%
\special{fp}%
\special{pa 1190 796}%
\special{pa 1190 696}%
\special{fp}%
\special{pa 1190 696}%
\special{pa 990 696}%
\special{fp}%
% LINE 2 0 3 0
% 32 990 695 990 670 990 670 1215 670 1215 670 1215 820 1215 820 915 820 915 820 915 870 915 870 1340 870 1340 870 1340 1045 1340 1045 740 1045 740 1045 740 1145 740 1145 1790 1145 1790 1145 1790 1095 1790 1095 1490 1095 1490 1095 1490 1070 1490 1070 1690 1070 1690 1070 1690 1045 1690 1045 1590 1045
% 
\special{pn 8}%
\special{pa 990 696}%
\special{pa 990 670}%
\special{fp}%
\special{pa 990 670}%
\special{pa 1216 670}%
\special{fp}%
\special{pa 1216 670}%
\special{pa 1216 820}%
\special{fp}%
\special{pa 1216 820}%
\special{pa 916 820}%
\special{fp}%
\special{pa 916 820}%
\special{pa 916 870}%
\special{fp}%
\special{pa 916 870}%
\special{pa 1340 870}%
\special{fp}%
\special{pa 1340 870}%
\special{pa 1340 1046}%
\special{fp}%
\special{pa 1340 1046}%
\special{pa 740 1046}%
\special{fp}%
\special{pa 740 1046}%
\special{pa 740 1146}%
\special{fp}%
\special{pa 740 1146}%
\special{pa 1790 1146}%
\special{fp}%
\special{pa 1790 1146}%
\special{pa 1790 1096}%
\special{fp}%
\special{pa 1790 1096}%
\special{pa 1490 1096}%
\special{fp}%
\special{pa 1490 1096}%
\special{pa 1490 1070}%
\special{fp}%
\special{pa 1490 1070}%
\special{pa 1690 1070}%
\special{fp}%
\special{pa 1690 1070}%
\special{pa 1690 1046}%
\special{fp}%
\special{pa 1690 1046}%
\special{pa 1590 1046}%
\special{fp}%
% VECTOR 2 0 3 0
% 4 2200 2400 3000 2400 2200 2400 600 2400
% 
\special{pn 8}%
\special{pa 2200 2400}%
\special{pa 3000 2400}%
\special{fp}%
\special{sh 1}%
\special{pa 3000 2400}%
\special{pa 2934 2380}%
\special{pa 2948 2400}%
\special{pa 2934 2420}%
\special{pa 3000 2400}%
\special{fp}%
\special{pa 2200 2400}%
\special{pa 600 2400}%
\special{fp}%
\special{sh 1}%
\special{pa 600 2400}%
\special{pa 668 2420}%
\special{pa 654 2400}%
\special{pa 668 2380}%
\special{pa 600 2400}%
\special{fp}%
% VECTOR 2 0 3 0
% 4 2600 600 2400 600 2600 600 2800 600
% 
\special{pn 8}%
\special{pa 2600 600}%
\special{pa 2400 600}%
\special{fp}%
\special{sh 1}%
\special{pa 2400 600}%
\special{pa 2468 620}%
\special{pa 2454 600}%
\special{pa 2468 580}%
\special{pa 2400 600}%
\special{fp}%
\special{pa 2600 600}%
\special{pa 2800 600}%
\special{fp}%
\special{sh 1}%
\special{pa 2800 600}%
\special{pa 2734 580}%
\special{pa 2748 600}%
\special{pa 2734 620}%
\special{pa 2800 600}%
\special{fp}%
% STR 2 0 3 0
% 3 400 300 400 400 2 0
% Overview of the upside of the frontier $p$-block.
\put(4.0000,-4.0000){\makebox(0,0)[lb]{Overview of the upside of the frontier $p$-block.}}%
% LINE 2 0 3 0
% 56 1590 1045 1540 1045 1540 1045 1540 1032 1540 1032 1702 1032 1702 1032 1702 1078 1702 1078 1500 1078 1500 1078 1500 1088 1500 1088 1800 1088 1800 1088 1800 1155 1800 1155 728 1155 728 1155 728 1032 728 1032 1328 1032 1328 1032 1328 880 1328 880 908 880 908 880 908 812 908 812 1208 812 1208 812 1208 678 1208 678 998 678 998 678 998 688 998 688 1198 688 1198 688 1198 802 1198 802 898 802 898 802 898 888 898 888 1298 888 1298 888 1298 1008 1298 1008 702 1008 702 1008 702 1182 702 1182 1915 1182 1915 1182 1915 1420
% 
\special{pn 8}%
\special{pa 1590 1046}%
\special{pa 1540 1046}%
\special{fp}%
\special{pa 1540 1046}%
\special{pa 1540 1032}%
\special{fp}%
\special{pa 1540 1032}%
\special{pa 1702 1032}%
\special{fp}%
\special{pa 1702 1032}%
\special{pa 1702 1078}%
\special{fp}%
\special{pa 1702 1078}%
\special{pa 1500 1078}%
\special{fp}%
\special{pa 1500 1078}%
\special{pa 1500 1088}%
\special{fp}%
\special{pa 1500 1088}%
\special{pa 1800 1088}%
\special{fp}%
\special{pa 1800 1088}%
\special{pa 1800 1156}%
\special{fp}%
\special{pa 1800 1156}%
\special{pa 728 1156}%
\special{fp}%
\special{pa 728 1156}%
\special{pa 728 1032}%
\special{fp}%
\special{pa 728 1032}%
\special{pa 1328 1032}%
\special{fp}%
\special{pa 1328 1032}%
\special{pa 1328 880}%
\special{fp}%
\special{pa 1328 880}%
\special{pa 908 880}%
\special{fp}%
\special{pa 908 880}%
\special{pa 908 812}%
\special{fp}%
\special{pa 908 812}%
\special{pa 1208 812}%
\special{fp}%
\special{pa 1208 812}%
\special{pa 1208 678}%
\special{fp}%
\special{pa 1208 678}%
\special{pa 998 678}%
\special{fp}%
\special{pa 998 678}%
\special{pa 998 688}%
\special{fp}%
\special{pa 998 688}%
\special{pa 1198 688}%
\special{fp}%
\special{pa 1198 688}%
\special{pa 1198 802}%
\special{fp}%
\special{pa 1198 802}%
\special{pa 898 802}%
\special{fp}%
\special{pa 898 802}%
\special{pa 898 888}%
\special{fp}%
\special{pa 898 888}%
\special{pa 1298 888}%
\special{fp}%
\special{pa 1298 888}%
\special{pa 1298 1008}%
\special{fp}%
\special{pa 1298 1008}%
\special{pa 702 1008}%
\special{fp}%
\special{pa 702 1008}%
\special{pa 702 1182}%
\special{fp}%
\special{pa 702 1182}%
\special{pa 1916 1182}%
\special{fp}%
\special{pa 1916 1182}%
\special{pa 1916 1420}%
\special{fp}%
% LINE 2 0 3 0
% 4 1915 1420 740 1420 740 1420 740 1600
% 
\special{pn 8}%
\special{pa 1916 1420}%
\special{pa 740 1420}%
\special{fp}%
\special{pa 740 1420}%
\special{pa 740 1600}%
\special{fp}%
% BOX 2 5 1 0
% 2 710 1600 770 1685
% 
\special{pn 8}%
\special{sh 0.300}%
\special{pa 710 1600}%
\special{pa 770 1600}%
\special{pa 770 1686}%
\special{pa 710 1686}%
\special{pa 710 1600}%
\special{ip}%
% BOX 2 5 1 0
% 2 770 1685 2800 1655
% 
\special{pn 8}%
\special{sh 0.300}%
\special{pa 770 1686}%
\special{pa 2800 1686}%
\special{pa 2800 1656}%
\special{pa 770 1656}%
\special{pa 770 1686}%
\special{ip}%
% BOX 2 5 1 0
% 2 2800 1685 2830 1615
% 
\special{pn 8}%
\special{sh 0.300}%
\special{pa 2800 1686}%
\special{pa 2830 1686}%
\special{pa 2830 1616}%
\special{pa 2800 1616}%
\special{pa 2800 1686}%
\special{ip}%
% BOX 2 5 1 0
% 2 2800 1615 2400 1630
% 
\special{pn 8}%
\special{sh 0.300}%
\special{pa 2800 1616}%
\special{pa 2400 1616}%
\special{pa 2400 1630}%
\special{pa 2800 1630}%
\special{pa 2800 1616}%
\special{ip}%
% LINE 2 2 3 0
% 8 2400 600 2400 1600 2800 600 2800 1600 600 2400 600 2200 3000 2400 3000 2200
% 
\special{pn 8}%
\special{pa 2400 600}%
\special{pa 2400 1600}%
\special{dt 0.045}%
\special{pa 2800 600}%
\special{pa 2800 1600}%
\special{dt 0.045}%
\special{pa 600 2400}%
\special{pa 600 2200}%
\special{dt 0.045}%
\special{pa 3000 2400}%
\special{pa 3000 2200}%
\special{dt 0.045}%
% LINE 2 2 3 0
% 4 1540 1020 1540 600 1680 1020 1680 600
% 
\special{pn 8}%
\special{pa 1540 1020}%
\special{pa 1540 600}%
\special{dt 0.045}%
\special{pa 1680 1020}%
\special{pa 1680 600}%
\special{dt 0.045}%
% STR 2 0 3 0
% 3 2220 490 2220 590 2 0
% $\gamma_{s+1}^{\min}$
\put(22.2000,-5.9000){\makebox(0,0)[lb]{$\gamma_{s+1}^{\min}$}}%
% STR 2 0 3 0
% 3 2740 490 2740 590 2 0
% $\gamma_{s+1}^{\max}$
\put(27.4000,-5.9000){\makebox(0,0)[lb]{$\gamma_{s+1}^{\max}$}}%
% STR 2 0 3 0
% 3 470 2470 470 2570 2 0
% $\gamma_p^{\min}$
\put(4.7000,-25.7000){\makebox(0,0)[lb]{$\gamma_p^{\min}$}}%
% STR 2 0 3 0
% 3 2930 2470 2930 2570 2 0
% $\gamma_p^{\max}$
\put(29.3000,-25.7000){\makebox(0,0)[lb]{$\gamma_p^{\max}$}}%
% STR 2 0 3 0
% 3 1420 500 1420 600 2 0
% $\gamma_s^{\min}$
\put(14.2000,-6.0000){\makebox(0,0)[lb]{$\gamma_s^{\min}$}}%
% STR 2 0 3 0
% 3 1700 500 1700 600 2 0
% $\gamma_s^{\max}$
\put(17.0000,-6.0000){\makebox(0,0)[lb]{$\gamma_s^{\max}$}}%
% VECTOR 2 0 3 0
% 2 1800 900 1600 1030
% 
\special{pn 8}%
\special{pa 1800 900}%
\special{pa 1600 1030}%
\special{fp}%
\special{sh 1}%
\special{pa 1600 1030}%
\special{pa 1668 1010}%
\special{pa 1646 1002}%
\special{pa 1646 978}%
\special{pa 1600 1030}%
\special{fp}%
% LINE 2 0 3 0
% 2 1800 900 2940 900
% 
\special{pn 8}%
\special{pa 1800 900}%
\special{pa 2940 900}%
\special{fp}%
% STR 2 0 3 0
% 3 3000 900 3000 1000 2 0
% The active block $Z_{s-1}^{\rm st}$.
\put(30.0000,-10.0000){\makebox(0,0)[lb]{The active block $Z_{s-1}^{\rm st}$.}}%
\end{picture}%
  \end{center}
 \vspace{-0.5em}
\caption{Outline of our construction of the injured case.}
  \label{fig:mthm10}
\end{figure}
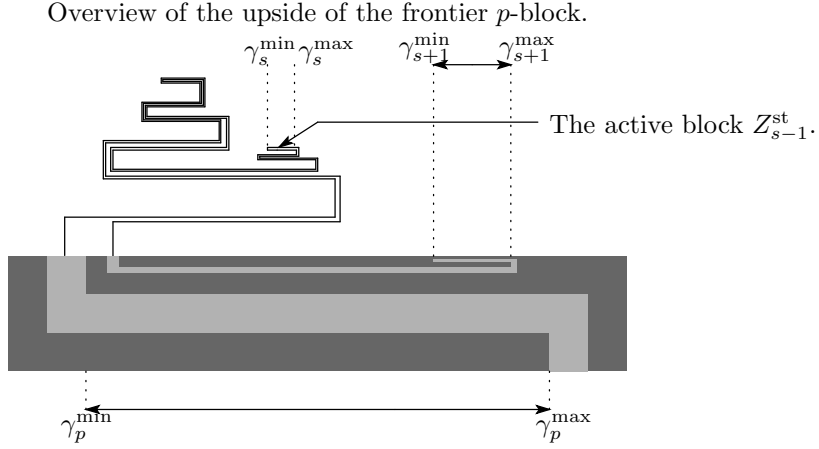

Finally we put $Q=\bigcap_{s\in\mathbb{N}}Q_s$ and $\mathcal{Z}^*=\bigcup_{u\in\mathbb{N}}\mathcal{Z}(u)$.
The construction is completed.

\medskip

\noindent
{\bf Verification.}
Now we start to verify our construction.

\begin{lemma}\label{lem:verify:6}
$Q$ is $\Pi^0_1$.
%%%%
\end{lemma}

\begin{sublem}\label{sublem:16}
$\bigcap_{t\in\mathbb{N}}\bigcup_{Z\in\mathcal{Z}^*}Z_t=\bigcup_{Z\in\mathcal{Z}^*}\bigcap_{t\in\mathbb{N}}Z_t$.
\end{sublem}

\begin{proof}\upshape
The intersection $Z_s(p)\cap Z^i_s$ for $i<2$ is included in some line segment $L_i\in\{[0,1]\times\{b\},\{b\}\times[0,1]:b\in\mathbb{R}\}$, and $Z_s(p)\cap L_i=Z_s(p)\cap Z^i_s$ holds.
%%%%
\end{proof}

\begin{sublem}\label{sublem:18}
$\bigcup_{Z\in\mathcal{Z}(u)}\bigcap_{t\in\mathbb{N}}Z_t$ is computably homeomorphic to $[0,1]\times P$, for each $u\in\nn$.
\end{sublem}

\begin{proof}\upshape
By the induction hypothesis (IH2).
%%%%
\end{proof}

\begin{sublem}\label{sublem:17}
$\bigcup_{Z\in\mathcal{Z}^*}\bigcap_{t\in\mathbb{N}}Z_t$ is homeomorphic to $(0,1]\times P$.
\end{sublem}

\begin{proof}\upshape
By Sublemma \ref{sublem:6} and \ref{sublem:15}.
%%%%
\end{proof}

\begin{lemma}\label{lem:verify:7}
$Q$ is homeomorphic to a Cantor fan.
\end{lemma}

\begin{proof}\upshape
By Sublemma \ref{sublem:16}, there exists a real $y_0\in\mathbb{R}$ such that the following holds:
\[Q=\left(\bigcup_{Z\in\mathcal{Z}^*}\bigcap_{t\in\mathbb{N}}Z_t\right)\cup\{\lrangle{\gamma,y_0}\}.\]
Therefore, by Sublemma \ref{sublem:17}, $Q$ is homeomorphic to the one-point compactification of $(0,1]\times P$.
%%%%
\end{proof}

\begin{lemma}\label{lem:verify:8}
$Q$ contains no computable point.
\end{lemma}

\begin{proof}\upshape
By Sublemma \ref{sublem:18}, $\bigcup_{Z\in\mathcal{Z}^*}\bigcap_{t\in\mathbb{N}}Z_t$ contains no computable point.
%%%%
\end{proof}

By Lemmata \ref{lem:verify:6}, \ref{lem:verify:7}, and \ref{lem:verify:8}, $Q$ is the desired dendroid.
%%%%
\end{proof}

\begin{remark}
Since dendroids are compact and simply connected, Theorem \ref{thm:special_roid} is the solution to the question of Le Roux and Ziegler \cite{RZ}.
Indeed, the dendroid constructed in the proof of Theorem \ref{thm:special_roid} is contractible.
\end{remark}

\begin{cor}
Not every nonempty contractible $\Pi^0_1$ subset of $[0,1]^2$ contains a computable point.
\end{cor}

\begin{question}
Does every locally connected planar $\Pi^0_1$ set contain a computable point?
\end{question}

\section{Immediate Consequences}

\subsection{Effective Hausdorff Dimension}

For basic definition and properties of the the effective Hausdorff dimension of a point of Euclidean plane, see Lutz-Weihrauch \cite{LW}.
For any $I\subseteq[0,2]$, let ${\rm DIM}^I$ denote the set of all points in $\mathbb{R}^2$ whose effective Hausdorff dimensions lie in $I$.
Lutz-Weihrauch \cite{LW} showed that ${\rm DIM}^{[1,2]}$ is path-connected, but ${\rm DIM}^{(1,2]}$ is totally disconnected.
In particular, ${\rm DIM}^{(1,2]}$ has no nondegenerate connected subset.
It is easy to see that ${\rm DIM}^{(0,2]}$ has no nonempty $\Pi^0_1$ simple curve, since every $\Pi^0_1$ simple curve contains a computable point, and the effective Hausdorff dimension of each computable point is zero.

\begin{theorem}
${\rm DIM}^{[1,2]}$ has a nondegenerate contractible $\Pi^0_1$ subset.
\end{theorem}

\begin{proof}\upshape
For any strictly increasing computable function $f:\nn\to\nn$ with $f(0)=0$ and any infinite binary sequence $\alpha\in 2^\nn$, put $\iota_f(\alpha)=\prod_{i\in\nn}\lrangle{\alpha(i),\alpha(f(i)),\alpha(f(i)+1),\dots,\alpha(f(i+1)-1)}$, where $\sigma\times\tau$ denotes the concatenation of binary strings $\sigma$ and $\tau$.
Then, $r:2^\nn\to\mathbb{R}$ is defined as $r(\alpha)=\sum_{i\in\nn}(\alpha(i)\cdot 2^{-(i+1)})$.
Note that $\alpha\not=\beta$ and $r(\alpha)=r(\beta)$ hold if and only if there is a common initial segment $\sigma\in 2^{<\nn}$ of $\alpha$ and $\beta$ such that $\sigma 0$ and $\sigma 1$ are initial segments of $\alpha$ and $\beta$ respectively, and that $\alpha(m)=1$ and $\beta(m)=0$ for any $m>lh(\sigma)$, where $lh(\sigma)$ denotes the length of $\sigma$.
In this case, we say that $\alpha$ {\em sticks to $\beta$ on $\sigma$}.
If $r(\alpha)\not=r(\beta)$, then clearly $r\circ\iota_f(\alpha)\not=r\circ\iota_f(\beta)$.
Assume that $\alpha$ sticks to $\beta$ on $\sigma$.
Then there are $m_0<m_1$ such that $\iota_f(\alpha)(m_0)=\iota_f(\alpha)(m_1)=\alpha(lh(\sigma))=0$ and $\iota_f(\beta)(m_0)=\iota_f(\beta)(m_1)=\beta(lh(\sigma))=1$ by our definition of $\iota_f$.
Therefore, $\iota_f(\alpha)$ does not stick to $\iota_f(\beta)$.
Hence, $r\circ\iota_f(\alpha)\not=r\circ\iota_f(\beta)$ whenever $\alpha\not=\beta$.
Actually, $r\circ\iota_f:2^\nn\to\mathbb{R}$ is a computable embedding.
For each $n\in\nn$, put $k_f(n)=\#\{s:f(s)<n\}$.
Then, there is a constant $c\in\nn$ such that, for any $\alpha\in 2^\nn$ and $n\in\nn$, we have $K(\iota_f(\alpha)\res n+k_f(n)+1)\geq K(\alpha\res n)-c$, where $K$ denotes the prefix-free Kolmogorov complexity.
Therefore, for any sufficiently fast-growing function $f:\nn\to\nn$ and any Martin-L\"of random sequence $\alpha\in 2^\nn$, the effective Hausdorff dimension of $r\circ\iota_f(\alpha)$ must be $1$.
Thus, for any nonempty $\Pi^0_1$ set $R\subseteq 2^\nn$ consisting of Martin-L\"of random sequences, $\{0\}\times(r\circ\iota_f(R))$ is a $\Pi^0_1$ subset of ${\rm DIM}^{\{1\}}$.
Let $Q$ be the dendroid constructed from $P=r\circ\iota_f(R)$ as in the proof of Theorem \ref{thm:special_roid}, where we choose $\gamma=\rho(B)$ as Chaitin's halting probability $\Omega$.
For every point $\lrangle{x_0,x_1}\in Q$, the effective Hausdorff dimension of $x_i$ for some $i<2$ is equivalent to that of an element of $P$ or that of $\Omega$.
Consequently, $Q\subseteq{\rm DIM}^{[1,2]}$.
%%%%
\end{proof}

\subsection{Reverse Mathematics}

\begin{theorem}
For every $\Pi^0_1$ set $P\subseteq 2^\nn$, there is a contractible planar $\Pi^0_1$ set $Q$ such that $Q$ is Turing-degree-isomorphic to $P$, i.e., $\{\deg_T(x):x\in P\}=\{\deg_T(x):x\in D\}$.
\end{theorem}

\begin{proof}\upshape
We choose $B$ as a c.e.\ set of the same degree with the leftmost path of $P$.
Then, the dendroid $Q$ constructed from $P$ and $B$ as in the proof of Theorem \ref{thm:special_roid} is the desired one.
\end{proof}

A compact $\Pi^0_1$ subset $P$ of a computable topological space is {\em Muchnik complete} if every element of $P$ computes the set of all theorems of $T$ for some consistent complete theory $T$ containing Peano arithmetic.
By Scott Basis Theorem (see Simpson \cite{Sim}), $P$ is Muchnik complete if and only if $P$ is nonempty and every element of $P$ computes an element of any nonempty $\Pi^0_1$ set $Q\subseteq 2^\nn$.

\begin{cor}
There is a Muchnik complete contractible planar $\Pi^0_1$ set.
\end{cor}

A compact $\Pi^0_1$ subset $P$ of a computable topological space is {\em Medvedev complete} (see also Simpson \cite{Sim}) if there is a uniform computable procedure $\Phi$ such that, for any name $x\in\nn^\nn$ of an element of $P$, $\Phi(x)$ is the set of all theorems of $T$ for some consistent complete theory $T$ containing Peano arithmetic.

\begin{question}
Does there exist a Medvedev complete simply connected planar $\Pi^0_1$ set?
Does there exist a Medvedev complete contractible Euclidean $\Pi^0_1$ set?
\end{question}

Our Theorem \ref{thm:special_roid} also provides a reverse mathematical consequence.
For basic notation for Reverse Mathematics, see Simpson \cite{SimRM}.
Let ${\sf RCA}_0$ denote the subsystem of second order arithmetic consisting of $I\Sigma^0_1$ (Robinson arithmetic with induction for $\Sigma^0_1$ formulas) and $\Delta^0_1$-${\sf CA}$ (comprehension for $\Delta^0_1$ formulas).
Over ${\sf RCA}_0$, we say that a sequence $(B_i)_{i\in\nn}$ of open rational balls is {\em flat} if there is a homeomorphism between $\bigcup_{i<n}B_i$ and the open square $(0,1)^2$ for any $n\in\nn$.
It is easy to see that ${\sf RCA}_0$ proves that every flat cover of $[0,1]$ has a finite subcover.

\begin{theorem}
The following are equivalent over ${\sf RCA}_0$.
\begin{enumerate}
\item Weak K\"onig's Lemma: every infinite binary tree has an infinite path.
\item Every open cover of $[0,1]$ has a finite subcover.
\item Every flat open cover of $[0,1]^2$ has a finite subcover.
\end{enumerate}
\end{theorem}

\begin{proof}\upshape
The equivalence of the item (1) and (2) is well-known.
It is not hard to see that ${\sf RCA}_0$ proves the existence of the sequence $\{Q_s\}_{s\in\nn}$ as in our construction of the dendroid $Q$ in Theorem \ref{thm:special_roid}, by formalizing our proof in Theorem \ref{thm:special_roid} in ${\sf RCA}_0$.
Here we may assume that $\{Q_s\}_{s\in\nn}$ is constructed from the set of all infinite paths of a given infinite binary tree $T\subseteq 2^{<\nn}$, and a c.e.\ complete set $B\subseteq\nn$.
Note that $\bigcup_{s<t}([0,1]^2\setminus Q_s)$ does not cover $[0,1]^2$ for every $t\in\nn$.
Over ${\sf RCA}_0$, there is a flat sequence $\{[0,1]^2\setminus Q^*_s\}_{s\in\nn}$ of open rational balls such that $\bigcap_{s<t}Q^*_s\supseteq\bigcap_{s<t}Q_s$ for any $t\in\nn$, and that an open rational ball $U$ is removed from some $Q^*_s$ if and only if an open rational ball $U$ is removed from some $Q_u$.
However, if $T$ has no infinite path, then $Q$ has no element.
In other words, $\{[0,1]^2\setminus Q^*_s\}_{s\in\nn}$ covers $[0,1]^2$.
%%%%
\end{proof}

\begin{ack}\upshape
The author thank Douglas Cenzer, Kojiro Higuchi and Sam Sandars for valuable comments and helpful discussion.
\end{ack}


\begin{thebibliography}{}
	\bibitem{Bra08} V. Brattka, Plottable real number functions and the computable graph theorem, {\em SIAM J. Comput.} {\bf 38} (2008), pp. 303-328.
	\bibitem{BP} V. Brattka, and G. Presser, Computability on subsets of metric spaces, {\em Theoretical Computer Science}, {\bf 305} (2003), pp. 43--76.
	\bibitem{BW} V. Brattka, K. Weihrauch, Computability on subsets of Euclidean space. I. Closed and compact subsets. Computability and complexity in analysis, {\em Theoret. Comput. Sci.} 219 (1999), 65--93.
	\bibitem{CKWW} D. Cenzer, T. Kihara, R. Weber, and G. Wu, Immunity and non-cupping for closed sets, {\em Tbilisi Math. J.}, {\bf 2} (2009), pp. 77--94.
	\bibitem{Her} P. Hertling, Is the Mandelbrot set computable? {\em Math. Log. Q.}, {\bf 51} (2005), pp. 5--18.
	\bibitem{KrLa} G. Kreisel, and D. Lacombe, Ensembles r\'ecursivement measurables et ensembles r\'ecursivement ouverts ou ferm\'e, {\em Compt. Rend. Acad. des Sci. Paris} {\bf 245} (1957), pp.1106--1109.
	\bibitem{Ilj} Z. Iljazovi\'c, Chainable and circularly chainable co-r.e. sets in computable metric space, {\em J. Universal Comp. Sci.}, {\bf 15} (2009), pp. 1206--1235.
	\bibitem{IlNa} A. Illanes, and S. Nadler, {\em Hyperspaces: Fundamentals and Recent Advances}, Marcel Dekker, Inc., New York, 1999.
	\bibitem{LW} Jack H. Lutz and Klaus Weihrauch, Connectivity properties of dimension level sets, {\em Mathematical Logic Quarterly}, {\bf 54} (2008), pp.\ 483--491. 
	\bibitem{Mil} J. S. Miller, Effectiveness for embedded spheres and balls, {\em Electronic Notes in Theore. Comp. Sci.}, {\bf 66} (2002), pp. 127--138.
	\bibitem{Nad} S. B. Nadler, {\em Continuum Theory}, Marcel Dekker, Inc., New York, 1992.
	\bibitem{Pen} R. Penrose, {\em Emperor's New Mind. Concerning Computers, Minds and The Laws of Physics.} Oxford University Press, New York, 1989.
	\bibitem{RZ} S. Le Roux, and M. Ziegler, Singular coverings and non-uniform notions of closed set computability, {\em Math. Log. Q.}, {\bf 54} (2008), pp. 545--560.
	\bibitem{SimRM} S. G. Simpson, {\em Subsystems of Second Order Arithmetic}, Springer-Verlag, 1999.
	\bibitem{Sim} S. G. Simpson, Mass problems and randomness, Bulletin of Symbolic Logic, {\bf 11}, pp. 1-27, (2005).
	\bibitem{Soa} R. I. Soare, {\em Recursively Enumerable Sets and Degrees}, Perspectives in Mathematical Logic, Springer, Heidelberg, XVIII+437 pages, (1987).
	\bibitem{Tan} H. Tanaka, On a $\Pi^0_1$ set of positive measure, {\em Nagoya Math. J.}, {\bf 38} (1970), pp. 139--144.
	\bibitem{Wei} K. Weihrauch, {\em Computable Analysis: an introduction}, Springer, Berlin, 2000.
\end{thebibliography}
\end{document}